\documentclass[a4paper,11pt]{article}
\usepackage{amsmath,amssymb,amsthm,graphicx}
\usepackage{hyperref,url,enumitem,bm,esint,makecell}
\usepackage{color}
\usepackage[margin=30mm]{geometry}
\usepackage{caption}
\usepackage{subcaption}

\newcommand{\pd}{\partial}
\newcommand{\eps}{\varepsilon}
\newcommand{\Lap}{\Delta}
\newcommand{\bu}{\bm{u}}
\newcommand{\bv}{\bm{v}}
\newcommand{\E}{\mathcal{E}}
\newcommand{\Eu}{\mathcal{E}(\bu)}

\newcommand{\CC}{\mathbb{C}}
\renewcommand{\div}{\mathrm{div}\,}
\newcommand{\R}{\mathbb{R}}
\newcommand{\pdnu}{\pd_{\bm{n}}}
\newcommand{\inn}[2]{\langle #1, #2 \rangle}

\newcommand{\dxt}{\, \mathrm{dx} \mathrm{dt}}
\newcommand{\dx}{\, \mathrm{dx}}

\newcommand{\dt}{\, \mathrm{dt}}

\newcommand{\vp}{\varphi}
\newcommand{\Th}{\mathcal{T}_h}
\newcommand{\Sh}{\mathcal{S}_h}
\newcommand{\Ih}{\mathrm{I}_h}
\newcommand{\bS}{\mathbb{S}}
\newcommand{\bM}{\mathbb{M}}
\newcommand{\bP}{\mathbb{P}}
\newcommand{\bX}{\mathbb{X}}

\theoremstyle{plain}
\newtheorem{thm}{Theorem}[section]
\newtheorem{lem}{Lemma}[section]
\newtheorem{prop}{Proposition}[section]

\newtheorem{remark}{Remark}[section]

\numberwithin{equation}{section}

\usepackage{easyReview}

\title{Numerical analysis of a FE/SAV scheme for a Caginalp phase field model with mechanical effects in stereolithography}
\author{Xingguang Jin \footnotemark[1] \and Kei Fong Lam \footnotemark[2] \and Changqing Ye \footnotemark[1]}
\date{ }

\begin{document}
\maketitle

\renewcommand{\thefootnote}{\fnsymbol{footnote}}
\footnotetext[1]{Department of Mathematics, The Chinese University of Hong Kong, Shatin, Hong Kong ({\tt $\{$xgjin,cqye$\}$@math.cuhk.edu.hk}).}
\footnotetext[2]{Department of Mathematics, Hong Kong Baptist University, Kowloon Tong, Hong Kong ({\tt akflam@hkbu.edu.hk}).}

\begin{abstract} 
In this work we propose a phase field model based on a Caginalp system with mechanical effects to study the underlying physical and chemical processes behind stereolithography, which is an additive manufacturing (3D printing) technique that builds objects in a layer-by-layer fashion by using an ultraviolet laser to solidify liquid polymer resins. Existence of weak solutions is established by demonstrating the convergence of a numerical scheme based on a first order scalar auxiliary variable temporal discretization and a finite element spatial discretization. We further establish uniqueness and regularity of solutions, as well as optimal error estimates for the Caginalp system that are supported by numerical simulations. We also present some qualitative two-dimensional simulations of the stereolithography processes captured by the model.
\end{abstract}



\section{Introduction}
In this work we study the following system of equations posed on a bounded domain $\Omega \subset \R^{d}$, $d \in \{2,3\}$, with boundary $\pd \Omega$ and for a fixed terminal time $T > 0$:
\begin{subequations}\label{model}
\begin{alignat}{3}
\alpha \pd_t \vp &= \lambda \eps \Lap \vp - \lambda \tfrac{1}{\eps} W'(\vp) - \gamma(\theta - \theta_c) p(\vp) && \text{ in } \Omega_T := \Omega \times (0,T), \label{model:1} \\
\delta \pd_t \theta &= \gamma p(\vp) \pd_t \vp + \Lap \theta && \text{ in } \Omega_T, \label{model:2} \\
\bm{0} & = \div ( \CC(\vp)(\Eu - \E_c(\vp) + \beta (\theta - \theta_0) \mathbb{I})) && \text{ in } \Omega_T, \label{model:3} \\
\bu & = \bm{0}, \quad \pdnu \vp = \pdnu \theta = 0 && \text{ on } \pd \Omega \times (0,T), \label{model:bc} \\
\vp(0,\cdot) &= \vp_0(\cdot), \quad \theta(0,\cdot) = \theta_0(\cdot) && \text{ in } \Omega.\label{model:ic}
\end{alignat}
\end{subequations}
In the above the primary variables of the model are $\vp$ (the phase field variable), $\theta$ (the temperature) and $\bu$ (the elastic displacement). With $\bm{n}$ as the outer unit normal of $\pd \Omega$ we denote by $\pdnu$ the outward normal derivative, i.e., $\pdnu f = \nabla f \cdot \bm{n}$. The parameters $\alpha$, $ \beta$, $\gamma$, $\delta$, $\eps$, $\lambda$ and $\theta_c$ are fixed positive constants, and the model \eqref{model} consists of a phase field system for $(\vp, \theta)$ that resembles the Caginalp model \cite{Cag}, coupling with a quasi-static linearized elasticity system. In \eqref{model:1}, $W'$ is the derivative of a double well potential $W$, while $p = P'$ is the derivative of a non-negative and bounded function $P$.  In the setting where $\gamma = 0$, $\alpha = \eps$ and $\lambda = 1$ this reduces to the familiar Allen--Cahn equation:
\[
\eps \pd_t \vp = \eps \Delta \vp - \frac{1}{\eps} W'(\vp),
\]
while in the setting where $p(\vp) = 1$, $\theta_c = \alpha = \delta = 0$ and $\gamma = \lambda = 1$, we obtain the Cahn--Hilliard equation from \eqref{model:1}-\eqref{model:2} with $-\theta$ playing the role of the associated chemical potential for $\vp$:
\[
\pd_t \vp = - \Delta \theta, \quad - \theta = -  \eps \Delta \vp +  \frac{1}{\eps} W'(\vp).
\]
The classical example for $W$ is the quartic potential $W(s) = \frac{1}{4}(s^2-1)^2$, so that $W'(s) = s^3 - s$. In \eqref{model:3}, the quantity $\Eu := \frac{1}{2}(\nabla \bu + (\nabla \bu)^{\top})$ is the symmetric strain tensor, $\E_c(\vp)$ is an eigenstrain, $\mathbb{I}$ is the second order identity tensor, and $\CC(\vp)$ is a symmetric and positive definite fourth order elasticity tensor depending on $\vp$. We furnish \eqref{model:1}-\eqref{model:3} with the boundary conditions \eqref{model:bc} (homogeneous Neumann for $\vp$ and $\theta$, and homogeneous Dirichlet for $\bu$) and initial conditions \eqref{model:ic}. 

We propose the system \eqref{model} as a phenomenological description for the physical processes behind stereolithography, which is an additive manufacturing (also colloquially known as 3D printing) technique that utilizes ultraviolet lasers to cure/solidify photosensitive liquid polymer resin in order to build objects and products in a layer-by-layer fashion. Despite being one of the earliest forms of additive manufacturing, stereolithography still remains a popular choice among modern practitioners to fabricate complex geometric shapes in an inexpensive, rapid and scalable fashion. However, much of the technological expertise and operating procedures are based on empirical experiments and work experience, while the understanding of the underlying physical and chemical changes behind the curing process remained incomplete. These mechanisms are instrumental in improving the product quality and printing precision in order to address some of the challenges preventing additive manufacturing as a whole into integrating with existing manufacturing infrastructures, see e.g.~the review article \cite{Abdu} for more details.

In the literature there have been several contributions on the development of mathematical models for stereolithography, all of which have the common feature that decompose the physical and chemical processes involved into multiple submodels that are consecutively coupled. We provide a derivation of our model \eqref{model} in Section~\ref{sec:model} and a comparison with several previous approaches. Our proposal to use a phase field model \eqref{model:1}-\eqref{model:2} to encode the evolution of the curing process, similarly done in \cite{Zhou}, is motivated from viewing the photopolymerization
of the liquid polymer akin to that of solidification in material sciences. Indeed, the liquid polymers polymerized into a solid state only when a critical temperature is exceeded, and mathematically this can be captured with models such as the classical Stefan problem employing a sharp interface free boundary description, or by the Caginalp model (and by close association also the Allen--Cahn and Cahn--Hilliard equations) employing a diffuse interface description. Included in \eqref{model:1} is a phenomenological term $\gamma (\theta - \theta_c) p(\vp)$ that allows for a mechanism to determine the energetically favorable phase based on the value of the temperature. We then coupled \eqref{model:1}-\eqref{model:2} to a quasi-static linear elastic system \eqref{model:3} to model the build up of mechanical properties of the cured polymers.

The main contribution of this work is the proposal and analysis of a fully discrete numerical scheme for \eqref{model} based on finite element (FE) spatial discretization and the recently popularized scalar auxiliary variable (SAV) time discretization approach \cite{SAV1,SAV2}. In order to achieve unconditional numerical stability in the presence of the nonlinear term $W'(\vp)$ in \eqref{model:1}, various approaches have been proposed by many authors, among which we mention the convex-concave splitting approach \cite{EllStu,Eyre} (resulting in nonlinear discrete systems), the stabilized linearly implicit approach \cite{ShenYang} (requiring large stabilization parameters and modified potential function), the Lagrange multiplier approach \cite{IEQ}, the invariant energy quadratization approach \cite{IEQ:diblock} and the scalar auxiliary variable approach \cite{SAV1,SAV2}. These aforementioned methods are applicable to a large class of equations arising as gradient flows of appropriate energy functionals, and we choose the SAV approach primarily due to its implementational advantage where advancing to the next time iteration for \eqref{model:1}-\eqref{model:2} requires only solving linear systems rather than employing complicated Newton iterations.

Our numerical analysis demonstrates the convergence of discrete solutions to a weak solution of \eqref{model} as the discretization parameters tend to zero. We additionally establish uniqueness and regularity of the continuous solution, thereby providing a well-posedness result for our proposed model. This is in contrast to previous works for stereolithography where the well-posedness of the models were not addressed. As a consequence, for the Caginalp submodel we can derive optimal error estimates
for our numerical scheme without assuming additional smoothness of the exact solutions. For related works on the numerical analysis of the SAV approach applied to Allen--Cahn and Cahn--Hilliard type systems we point to \cite{LamWangSAV,MetzgerSAV,SAVconv} for convergence of discrete solutions to weak solutions, and to \cite{ChenSAV,GSAV:err,Li:CHS,SAVconv,Zheng:CHHS,Zhang} for error estimates under additional smoothness of the exact solutions.

The rest of this paper is organized as follows: in Section \ref{sec:model} we provide a derivation of \eqref{model} and give a comparison with existing models in the literature proposed for stereolithography. The fully discrete numerical scheme is introduced in Section \ref{sec:numer}, where we present stability estimates and discuss the convergence of discrete solution to a weak solution as the discretization parameters tend to zero. In Section \ref{sec:wellposed} we establish the well-posedness of the model \eqref{model} and improve the regularity of weak solutions, so that in Section \ref{sec:error} we carry out an error analysis to derive error estimates in terms of the discretization parameters for the Caginalp submodel. Supporting numerical simulations are displayed in Section \ref{sec:simulation}.

\section{Model derivation}\label{sec:model}
Consider a bounded domain $\Omega \subset \R^3$ containing a partially constructed object positioned on a platform surrounded by a viscous liquid resin.  The object is made of the same type of material as the resin, but is fully cured (solidified).  An ultraviolet laser positioned outside the domain $\Omega$ serves to trace out the design of the printed object layer-by-layer. The energy from the laser induces a polymerization reaction that causes the exposed liquid resin to transition to a cured phase. After a layer has been created on top of the existing object, the platform which it rests on is then lowered, and a re-coating blade moves across the surface to level the viscous resin covering the newly cured layer.  The laser then traces out the next layer and the process repeats until the final layer is built.  The schematics is summarized in Figure~\ref{fig:stereo} and more details can be found in \cite{Bartolo}.

\begin{figure}[h]
\centering
\includegraphics[width=0.8\textwidth]{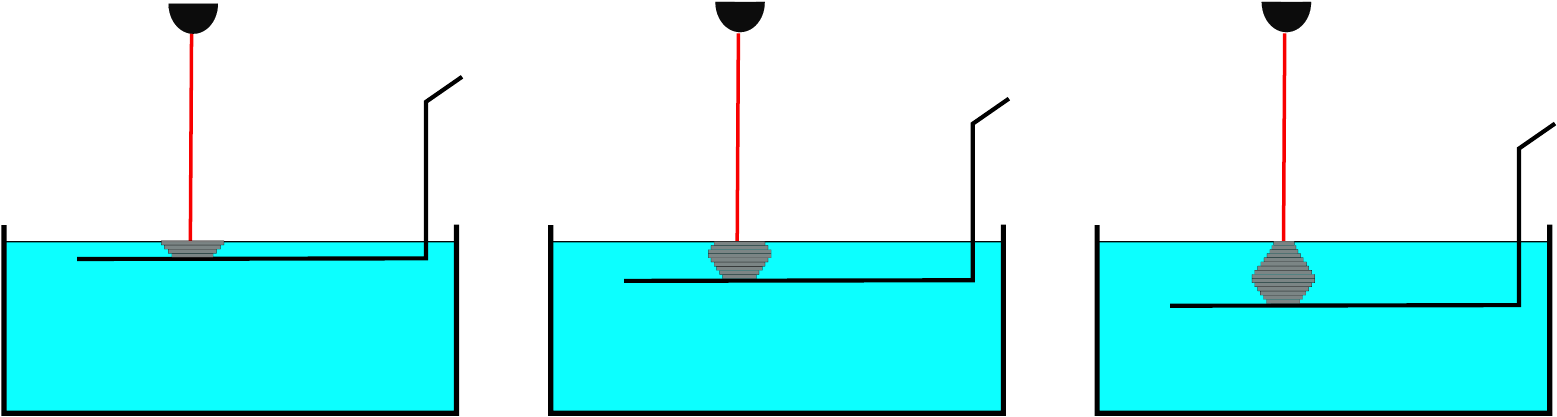}
\caption{Section view of the printing process involved in stereolithography. An object is printed layer by layer on an adjustable platform submerged in a vat of liquid resin. The resin hardens when struck by an ultraviolet laser positioned outside the vat, and the platform lowers in order to harden the next layer of resin directly on top of the previous one.}
\label{fig:stereo}
\end{figure}

The multiphysical nature of stereolithography has led to a modeling approach coupling individual physical submodels for each of the following four processes: (i) the irradiation and absorption of the laser energy; (ii) the conversion of liquid monomers to solid polymers via polymerization; (iii) the propagation of heat from polymerization; and (iv) the build up of mechanical properties during the curing process. 

\subsection{Laser irradiation}
As the laser is positioned outside the domain $\Omega$, it is sufficient to focus on the attenuation (gradual reduction of intensity) inside the domain. Many models have been proposed, see e.g.~\cite{Bartolo,Classens,Goodner,Jacobs,Perry,Schmocker,Watts} and the references cited therein. While this part is not the main focus of the present work, for convenience let us provide the simplest description.  We use the notation that the top layer of the three-dimensional domain where the laser hits the resin is $\{ z = 0\}$. Assuming a Gaussian beam profile and using the Beer--Lambert law for the laser attenuation, the laser intensity $I$ can be modeled as
\[
I(\bm{x},t) = I_0(t) \exp \Big ( \frac{-(x^2 + y^2)}{w_0^2} \Big ) \exp \Big ( \frac{z}{D_p} \Big ) \quad \text{ for } \bm{x} = (x,y,z) \in \Omega, \, t \geq 0, 
\]
with incident laser intensity $I_0(t)$ as a function of time, beam radius $w_0$ and penetration depth $D_p>0$. Refraction effects can be neglected as the distance between the resin surface $\{z = 0\}$ and the submerged object is typically small. Further extensions can be found in \cite[Chap.~8.2]{Bartolo} for refraction effects, in \cite{Westbeek} for scattering and diffraction effects, and in \cite{Brighenti} for diffusion and absorption effects. In this work we assume the laser intensity $I$ is a prescribed function and enters into the next submodel via a heat source.

\subsection{Phase field model for polymerization and heat propagation}
In many works, a full reaction kinetic description employing systems of coupled ordinary differential equations or reaction-diffusion equations has been used for the photopolymerization of monomers in Part (ii), see e.g.~\cite{Bartolo,Goodner,Lange,Watts,Westbeek}, in order to account for the evolution of the concentrations of different monomers and radicals (molecules with at least one unpaired valence electron) during the polymerization process. The key variable that quantifies the amount of monomer conversion to polymer is the {\it degree of conversion} \cite{Watts}, defined as the relative difference between the initial and current concentrations of monomers. However, the complexity of these types of model increases with the number of distinct monomer and radical species, and from a measurement viewpoint, accurately quantifying individual monomer and radical concentrations can be difficult in experimental settings.

Hence, we follow \cite{Classens,Zhou} and propose a phenomenological model that encodes the degree of conversion without accessing individual species concentrations. The degree of conversion and the status of polymerization can be implicitly summarized by the location of the interfaces between the liquid resin (sol phase) and the cured resin (gel phase), whose evolution can be tracked and measured in more accessible ways \cite{Schmocker}. 

In Part (iii), heat propagation during the polymerization process is often modeled using a heat equation with source term depending on the time evolution of the degree of conversion \cite{Classens,Goodner,Westbeek}.  We take a slightly different approach by combining Parts (ii) and (iii) into one coupled model, where the primary variables are the phase field variable $\vp$ representing the degree of conversion and a temperature variable $\theta$. We propose an energy functional of the form
\[
\mathbb{E}(\vp, \theta) :=  \int_\Omega \frac{\lambda \eps}{2} |\nabla \vp|^2 + \frac{\lambda}{\eps} W(\vp)  - \frac{\delta}{2}|\theta|^2 + \gamma P(\vp)(\theta - \theta_c) \dx
\]
where the first two terms constitute the Ginzburg--Landau functional involving a non-negative potential $W$ with two equal minima at $\pm 1$ and a constant surface tension coefficient $\lambda > 0$. We make use of the well-known behavior that for $0 < \eps \ll 1$, minimizers of the Ginzburg--Landau functional attain values close to the stable constant minima $\pm 1$ of $W$ and transition smoothly from one value to the another in thin layers whose thickness scales with $\eps$. Thus, it is expected that the model develops large regions in $\Omega$ where $\vp$ is close to $\pm 1$, which allows us to associate the sol phase as the region where $\{ \vp \sim -1\}$, and the gel phase as the region where $\{\vp \sim 1\}$. In particular, $\vp$ can be interpreted as the difference in the volume fractions of the gel phase and of the sol phase.

In the third term of $E$ the constant $\delta$ corresponds to the specific heat coefficient, and we proposed the fourth term in order to capture the following behavior \cite{Zhou}: The gel phase is energetically preferable when the temperature $\theta$ exceeds a constant critical temperature $\theta_c$, while the sol phase is preferred when $\theta < \theta_c$. Together with a function $P:[-1,1] \to [0,\infty)$ that has a maximum at $s = -1$ and a minimum at $s = 1$, we see that this behavior is captured when we try to minimize the fourth term in $E$. Two examples satisfying the requirements are $P(s) = \frac{1}{2}(1-s)$ and $P(s) = \frac{1}{4}(1-s)^2(2+s)$ for $s \in [-1,1]$ with constant extensions $P(s) = 1$ for $s < -1$ and $P(s) = 0$ for $s > 1$. Furthermore, the constant $\gamma$ corresponds to the latent heat coefficient.

Based on the energy $E$, the evolution of $\vp$ can be obtained with a non-conserving gradient flow:
\[
\alpha \pd_t \vp = - \frac{\delta E}{\delta \vp} = \lambda \eps \Delta \vp - \frac{\lambda}{\eps} W'(\vp) - \gamma(\theta - \theta_c) p(\vp),
\]
where $\alpha > 0$ and $p(s) = P'(s)$. For the temperature, we assume an isobaric (constant pressure) process, and the change in the enthalpy $H := - \frac{\delta E}{\delta \theta} = \delta \theta - \gamma P(\vp)$ is captured by the balance law
\[
\pd_t H = \div \bm{q} + f
\]
with thermal flux $\bm{q}$ and heat source $f$.  We take Fourier's law $\bm{q} = - \nabla \theta$ which leads to
\[
\delta \pd_t \theta - \gamma p(\vp) \pd_t \vp = \Delta \theta + f.
\]
The term $p(\vp) \pd_t \vp$ can be interpreted as the curing rate \cite{Classens}.  This leads to the submodel \eqref{model:1}-\eqref{model:2}, which is also known in the phase field literature as the Caginalp model \cite{Cag} when $p(s) = 1$. Note that we can incorporate the laser intensity from Part (i) into the heat source $f$ as a prescribed function, see e.g.~\cite{Westbeek}. For the subsequent mathematical analysis we set $f = 0$ for simplicity.

\subsection{Mechanical effects}
For Part (iv), it is assumed that mechanical properties only develop for the gel phase, and mechanical stresses do not influence the polymerization process in Parts (ii) and (iii), see also Remark \ref{rem:CahnLarche} below. However, in order to obtain a non-degenerate system of equations amenable to further analysis, we make use of the ersatz material approach in phase field-based structural topology optimization, see e.g.~\cite{Blank,Bourdin,Takezawa,Wang}, which treats the sol phase as a very soft elastic material, so that we can define a displacement vector $\bu$ over the entirety of $\Omega$. 

Similar to \cite{Classens,Westbeek,Zhou} we decompose the total strain tensor $\bm{e}$ into a sum of an infinitesimal linear elastic strain $\Eu = \frac{1}{2}(\nabla \bu + (\nabla \bu)^{\top})$, a thermal strain $\E_\theta := \beta(\theta - \theta_0) \mathbb{I}$ with identity tensor $\mathbb{I}$, initial temperature $\theta_0$ and thermal expansion coefficient $\beta> 0$, and an induced chemical strain $\E_c$ arising from the increase in density due to polymerization that results in shrinkage of the cured resin counteracting the thermal expansion \cite{Bartolo,Classens,Westbeek,Zhou}. This induced strain is assumed to be isotropic and only occurs in the gel phase \cite{Westbeek}. One example we can take is $\E_c(\vp) = \zeta (1-P(\vp)) \mathbb{I} =: m(\vp) \mathbb{I}$, where $\zeta > 0$ is a scalar corresponding to the maximum shrinkage strain, so that $\E_c(-1) = \bm{0}$. 

As both the sol and gel phase are modeled elastic materials, let $\CC^{(0)}$ and $\CC^{(1)}$ denote their corresponding constant elasticity tensors that are positive definite on symmetric matrices and fulfil the usual symmetric conditions of linear elasticity (see \eqref{ass:CC} below), respectively. We introduce a interpolation fourth order elasticity tensor $\CC(\vp) = (1-k(\vp)) \CC^{(0)} + k(\vp) \CC^{(1)}$ for some function $k:\R \to [0,1]$ satisfying $k(1) = 1$ and $k(-1) = 0$. Then, the balance of linear momentum in the absence of external body forces yields
\[
\div (\CC(\vp) (\Eu - \E_c(\vp) + \beta(\theta - \theta_0) \mathbb{I})) = \bm{0}.
\]
Combining the three derived equations for $(\vp, \theta, \bu)$ and furnishing initial-boundary conditions leads to our phase field model \eqref{model}.

Let us provide two examples of $\CC^{(1)}$ for applications relevant to additive manufacturing. Assuming gel phases is an isotropic linearly elastic material, then $\CC^{(1)}$ takes the form
\begin{align*}
\CC^{(1)}_{ijmn} & = \lambda \delta_{ij} \delta_{mn} + \mu (\delta_{im} \delta_{jn} + \delta_{in} \delta_{jm}) \\
& = \frac{E\nu}{(1+\nu)(1-2\nu)}\delta_{ij} \delta_{mn} + \frac{E}{2(1+\nu)}(\delta_{im} \delta_{jn} + \delta_{in} \delta_{jm})
\end{align*}
with gel phase Lam\`e constants $\lambda$ and $\mu$, Young's modulus $E$ and Poisson ratio $\nu$ that are related via the relations
\[
\lambda = \frac{E\nu}{(1+\nu)(1-2\nu)}, \quad \mu = \frac{E}{2(1+\nu)}, \quad E = \frac{\mu(3\lambda + 2 \mu)}{\lambda + \mu}, \quad \nu = \frac{\lambda}{2(\lambda + \mu)}.
\]
For the sol phase that is treated as an ersatz material, we fix $0 < \kappa \ll 1$ and consider the sol phase Lam\`e constants to be $\kappa \lambda$ and $\kappa \mu$, i.e., $\CC^{(0)} = \kappa \CC^{(1)}$, and this leads to an interpolation elasticity tensor of the form 
\begin{align}\label{linear:elas}
\CC(\vp)_{ijmn} = (\kappa \lambda + k(\vp)(1-\kappa) \lambda) \delta_{ij} \delta_{mn} + (\kappa \mu + k(\vp)(1-\kappa) \mu)(\delta_{im} \delta_{jn} + \delta_{in} \delta_{jm}).
\end{align}
Assuming as in \cite{Classens} that the Poisson ratio $\nu$ is constant throughout polymerization, i.e., the Poisson ratios of both sol and gel phases are equal, while setting the Young's modulus of the sol phase as $\kappa E$, then an alternative interpolation elasticity tensor to \eqref{linear:elas} is
\begin{align}\label{linear:elas:Young}
\CC(\vp)_{ijmn} = \frac{\nu(\kappa E + k(\vp)(1-\kappa) E)}{(1+\nu)(1-2 \nu)} \delta_{ij} \delta_{mn} + \frac{\kappa E + k(\vp)(1-\kappa) E}{2(1+\nu)} (\delta_{im} \delta_{jn} + \delta_{in} \delta_{jm}).
\end{align}

On the other hand, it is well known that objects built by additive manufacturing techniques exhibit anisotropic material properties due to the layer by layer printing process, which can be described by modeling both sol and gel phases as orthotropic materials \cite{Biswas,Huang}. In three spatial dimensions, the elasticity tensor of such materials has 9 independent components due to having three mutually orthogonal planes of reflection symmetry. 

For orthotropic materials, it is more common to express the compliance tensor (inverse of elasticity tensor) in terms of materials parameters. Let $\{E_i\}_{i=1}^3$ denote the Young's moduli in the three principal directions, $\{\nu_{ij}\}_{i,j = 1, \, i \neq j}^3$ denote the Poisson's ratio characterizing the transverse strain in the $j$th direction when the material is stressed in the $i$th direction, and $\{G_{ij}\}_{i,j=1, \, i < j}^3$ denote the shear modulus characterizing the ratio between the shear stress in the $i$th direction and the shear strain in the $j$th direction. Then, the fourth order compliance tensor $[\CC^{(1)}]^{-1}$ of an orthotropic material expressed in Voigt notation reads as 
\[
[\CC^{(1)}]^{-1} = \begin{pmatrix} \frac{1}{E_1} & - \frac{\nu_{21}}{E_2} & - \frac{\nu_{31}}{E_3} & 0 & 0 & 0 \\
-\frac{\nu_{12}}{E_1} & \frac{1}{E_2} & - \frac{\nu_{32}}{E_3} & 0 & 0 & 0 \\
-\frac{\nu_{13}}{E_1} & -\frac{\nu_{23}}{E_2} & \frac{1}{E_3} & 0 & 0 & 0 \\
0 & 0 & 0 & \frac{1}{2 G_{23}} & 0 & 0 \\
0 & 0 & 0 & 0 & \frac{1}{2 G_{13}} & 0 \\
0 & 0 & 0 & 0 & 0 & \frac{1}{2G_{12}} \end{pmatrix},
\]
where by symmetry requirements, it holds that 
\[
\frac{\nu_{21}}{E_2} = \frac{\nu_{12}}{E_1}, \quad \frac{\nu_{31}}{E_3} = \frac{\nu_{13}}{E_1}, \quad \frac{\nu_{32}}{E_3} = \frac{\nu_{23}}{E_2}.
\]
Taking $\CC^{(0)} = \kappa \CC^{(1)}$ with $0 < \kappa \ll 1$ for the sol phase, we can consider an interpolation elasticity tensor for orthotropic materials of the form 
\[
\CC(\vp) = [\kappa + k(\vp) (1-\kappa) ] \CC^{(1)}.
\]
\subsection{Boundary conditions}
From the schematics in Figure~\ref{fig:stereo}, we can assume that in a typical set-up the printed object (gel phase) is completely submerged in the liquid resin (sol phase). Hence, appropriate boundary conditions for $\vp$ can be the homogeneous Neumann condition $\pdnu \vp = 0$ or the Dirichlet boundary condition $\vp = -1$.  For the temperature we can prescribe homogeneous Neumann condition $\pdnu \theta = 0$ to describe thermal isolation of the construction environment \cite{Westbeek}. Alternatively, a Robin boundary condition $\pdnu \theta = \theta_\infty - \theta$ with ambient temperature $\theta_\infty$ is suitable, see \cite{Classens,Zhou}.  From the set-up in Figure~\ref{fig:stereo} and on account of the ersatz material approximation of the sol phase, an appropriate boundary condition for the displacement is a homogeneous Dirichlet condition $\bu = \bm{0}$. Another option is to consider a mixed boundary condition where the boundary $\pd \Omega$ is decomposed into a disjoint union of relatively open subsets $\Gamma_0 \cup \Gamma_1$ where traction-free conditions $\CC(\vp) (\Eu - \mathcal{E}_c(\vp) + \beta(\theta - \theta_0) \mathbb{I}) \bm{n} = \bm{0}$ are imposed on $\Gamma_1$, while $\bu = \bm{0}$ is imposed on $\Gamma_0$.

\subsection{Comparison with similar phenomenological models}
In this section we provide a comparison between \eqref{model} and similar phenomenological models with mechanical effects.  Let us comment that in choosing $\vp$ to be the difference in volume fractions of the gel and sol phases means that $\vp$ should belong to the physically relevant interval $[-1,1]$. Hence, in our model we emphasize the values of the interpolating functions $P(\vp)$ and $k(\vp)$ at $\pm 1$ in order to capture the relevant physical effects in the sol phase $\{ \vp \sim -1\}$ and in the gel phase $\{\vp \sim 1\}$.  However, it is also possible to choose $\vp$ as a different physical variable, such as the relative difference between initial and current monomer concentrations, leading to a different physically relevant interval $[0,1]$ and with the sol phase now defined as $\{\vp \sim 0 \}$ and the gel phase remaining as $\{\vp \sim 1\}$.  Mathematically, this only involves a composition of the potential $W$ with a suitable affine linear function so that the new potential has its minima at $0$ and $1$. One example of such a potential is the function $W(s) = s^2(s-1)^2$.

Our closest counterpart is the model of \cite{Zhou}, which also employs a phase field description where $\vp \in [0,1]$ is the volume fraction of the molecules that have undergone the sol-gel transition.  The equation for $\vp$ is of a similar Allen--Cahn type as \eqref{model:1}, which can be obtained by setting $\lambda = \eps$, and replacing $W(\vp)$ with $\frac{k_B}{v_a} \theta [ (1-\vp) \ln(1-\vp) + \chi \vp(1-\vp)]$ with the Boltzmann constant $k_B$, an interaction parameter $\chi$, and the volume of a single monomer/solvent molecule $v_a$, as well as replacing $\gamma (\theta - \theta_c) p(\vp)$ with $6 \vp(1-\vp) (f_{\mathrm{gel}} - f_{\mathrm{sol}})$ with bulk free energy density at the gel phase $f_{\mathrm{gel}}$ and the sol phase $f_{\mathrm{sol}}$, respectively:
\[
\pd_t \vp = \eps^2 \Delta \vp + \tfrac{k_B}{v_a} \theta (1+ \ln(1-\vp) + \chi(2\vp - 1)) - 6 \vp(1-\vp)(f_{\mathrm{gel}} - f_{\mathrm{sol}}).
\]
The main difference with \eqref{model:1} is that the temperature appears in the potential term as a prefactor, while the temperature equation reads as
\[
\rho \delta \pd_t \theta = k \Delta \theta
\]
with heat conductivity coefficient $k$ and mass density $\rho = \frac{\rho_0}{1+\div \bu}$ where $\rho_0$ represents the initial mass density. Notice that there is no explicit terms in $\vp$ appearing in the temperature equation, but the coupling with $\vp$ enters via the displacement $\bu$.  In addition to the decomposition of the total strain into a sum of elastic, thermal and chemical shrinkage strains, the mechanical behavior considered in \cite{Zhou} is described by a rheological model also includes viscoelastic effects, which we have neglected in this work.

In \cite{Westbeek} (see also \cite{Classens} for the one-dimensional analogue), $\vp \in [0,1]$ is chosen as the relative difference between initial and current monomer concentrations, and a phenomenological model is proposed taking the form of the following differential equation
\[
\pd_t \vp = r f(\vp) I^b,
\]
with rate constant $r$ that can depend on temperature $\theta$, reaction kinetics model function $f$ (linear in \cite{Westbeek} and $m$th order polynomial in \cite{Classens}), and laser intensity $I$ with exponent $b \in [0.5,1]$, see \cite[Chap.~9.3]{Bartolo}. Notice that a similar type of equation can be obtained from \eqref{model:1} by setting $\lambda = 0$ and choosing $p(s)$ appropriately.  The temperature equation adapted from \cite{Goodner} reads as
\[
\rho \delta \pd_t \theta = \div (k \nabla \theta) - (\Delta H) \pd_t \vp
\]
with constant mass density $\rho$, thermal conductivity $k$ and heat of polymerization $\Delta H$, which we note the resemblance to \eqref{model:2}.  For the mechanical response, a linear elastic behavior is assumed employing a similar decomposition of the total strain into a sum of  elastic, thermal and chemical components. A $\vp$-dependent Young's modulus $E(\vp)$ is prescribed:
\[
E(\vp) = \begin{cases}
e_0 E_{\mathrm{pol}} & \text{ for } \vp < \vp_{\mathrm{gel}}, \\
\Big ( \frac{1-e_0}{1- \vp_{\mathrm{gel}}} (\vp - \vp_{\mathrm{gel}}) \Big ) E_{\mathrm{pol}} & \text{ for } \vp \geq \vp_{\mathrm{gel}},
\end{cases}
\]
for constants $0 < e_0 \ll 1$, Young modulus of the gel phase $E_{\mathrm{pol}}$, and gel-point $\vp_{\mathrm{gel}}$ for conversion defined as the point where the liquid resin transforms to a solid polymer. Then, with a constant Poisson's ratio $\nu$, the phase dependent Lam\`e constants
\[
\lambda(\vp) = \frac{E(\vp)}{(1+\nu)(1-2\nu)}, \quad \mu(\vp) = \frac{E(\vp)}{2(1+\nu)},
\]
enter into an elasticity tensor $\CC(\vp)$ of the form \eqref{linear:elas} and the resulting momentum balance equation resembles \eqref{model:3} and \eqref{linear:elas:Young}.

\begin{remark}\label{rem:CahnLarche}
We note that \eqref{model} somewhat resembles the Cahn--Larch\'e model \cite{CahnLarche,Onuki} (where we use the notation $\mu$ to denote the associated chemical potential)
\begin{align*}
\pd_t \vp & = \Delta \mu, \\
\mu & = - \eps \Delta \vp + \eps^{-1} W'(\vp) \\
& \quad + \tfrac{1}{2}(\Eu - \E_c(\vp)): \CC'(\vp)(\Eu - \E_c(\vp)) - \CC(\vp)(\Eu - \E_c(\vp)): \E_c'(\vp), \\
\bm{0} & = \div(\CC(\vp)(\Eu - \E_c(\vp))),
\end{align*}
with the associated energy functional
\[
\mathbb{E}(\vp, \bu) := \int_\Omega \frac{\eps}{2} |\nabla \vp|^2 + \frac{1}{\eps} W(\vp) + \frac{1}{2}(\Eu - \E_c(\vp)) : \CC(\vp) (\Eu - \E_c(\vp)) \, dx.
\]
The most significant difference is the inclusion of an elastic contribution in $\mathbb{E}$, leading to an additional term appear in the equation for $\mu$ and thus allowing elastic stress to influence the evolution of $\vp$. 
This elastic influence is not present in \eqref{model} due to how we build the model by connecting submodels in a sequence. To the best of our knowledge, it is not entirely clear if this mechanical feedback is detected in the physical and chemical processes of stereolithography, and further investigations would warrant a comprehensive study in model calibration and validation with experimental data. We mention that a numerical analysis of a finite element scheme for the Cahn--Larch\'e system can be found in \cite{Garcke}.
\end{remark}

\section{Numerical discretization}\label{sec:numer}

\subsection{Preliminaries and assumptions}

\textbf{Notation}: For any $p \in [1,\infty]$ and $k >0$, the standard Lebesgue and Sobolev spaces over $\Omega$ are denoted by $L^p(\Omega)$ and $W^{k,p}(\Omega)$ with the corresponding norms $\|\cdot \|_{L^p}$ and $\|\cdot \|_{W^{k,p}}$.  In the special case $p = 2$, these become Hilbert spaces and we employ the notation $H^k := H^k(\Omega) = W^{k,2}(\Omega)$ with the corresponding norm $\| \cdot \|_{H^k}$. We denote the topological dual of $H^1(\Omega)$ by $(H^1(\Omega))^*$ and the corresponding duality pairing by $\inn{\cdot}{\cdot}$.  We use $\| \cdot \|$ and $(\cdot,\cdot)$ for the norm and inner product of $L^2(\Omega)$. Furthermore we define the Sobolev spaces $H^2_{n}(\Omega) := \{ f \in H^2(\Omega) \, : \, \pdnu f = 0 \text{ on } \pd \Omega \}$ and $X(\Omega) := \{ \bm{f} \in H^1(\Omega, \R^d) \, : \, \bm{f} = \bm{0} \text{ on } \pd \Omega \}$. Then, by \cite[Thm.~6.15-4, pp.~409--410]{Ciarlet} a Korn-type inequality is valid in $X(\Omega)$, where there exists a positive constant $C_K$ such that 
\[
\| \bu \|_{H^1} \leq C_K \| \Eu \| \quad \forall \bu \in X(\Omega).
\]
The discrete Gronwall inequality will often be invoked: if $e_n, a_n, b_n \geq 0$ for all $n \geq 0$, then
\[
e_n \leq a_n + \sum_{i=0}^{n-1} b_i e_i \quad \forall n \geq 0 \quad \implies \quad e_n \leq a_n \cdot \exp \Big (\sum_{i=0}^{n-1} b_i \Big ) \quad \forall n \geq 0.
\]

We make the following assumptions for the model:
\begin{enumerate}[label=$(\mathrm{A \arabic*})$, ref = $\mathrm{A \arabic*}$]
\item \label{ass:dom} $\Omega \subset \R^{d}$, $d \in \{2,3\}$, is a bounded convex domain with polygonal (if $d = 2$) or polyhedral (if $d = 3$) boundary $\pd \Omega$.
\item \label{ass:const:m} The constants $\alpha$, $\beta$, $\gamma$, $\delta$, $\lambda$, $\theta_c$ and $\eps$ are positive and fixed.
\item \label{ass:W} $W: \R \to \R$ is non-negative with $W \in C^2(\R)$, and there exists positive constants $c_0 > 0$, $c_1 \in \R$, $c_2 > 0$ and $c_3 > 0$ such that 
\[
c_0 |s|^3 - c_1 \leq |W'(s)| \leq c_2(1+|s|^3), \quad W''(s) \geq - c_3 \quad \forall s \in \R.
\]
\item \label{ass:p} $P:\R \to \R$ is a non-negative function with $P \in W^{2,\infty}(\R)$.
\item \label{ass:ini} The initial conditions satisfy $\phi_0, \theta_0 \in H^2_n(\Omega)$.
\item \label{ass:CC} The fourth order elasticity tensor $\CC(\vp)$ is of the form
\[
\CC(\vp) = \CC^{(0)} + k(\vp) (\CC^{(1)} - \CC^{(0)}) = (1-k(\vp)) \CC^{(0)} + k(\vp) \CC^{(1)},
\]
with constant fourth order tensors $\CC^{(0)}$ and $\CC^{(1)}$ that are positive definite in the sense that there exists positive constants $c_4$, $c_5$, $c_6$ and $c_7$ such that for all non-zero symmetric matrices $\bm{A} \in \R^{d \times d}_{\mathrm{sym}} \setminus \{ \bm{0} \}$:
\[
c_4 |\bm{A}|^2 \leq \CC^{(0)} \bm{A} : \bm{A} \leq c_5 | \bm{A}|^2, \quad c_6 |\bm{A}|^2 \leq \CC^{(1)} \bm{A} : \bm{A} \leq c_7 | \bm{A}|^2,
\]
where $\bm{A} : \bm{B} = \sum_{i,j=1}^d A_{ij} B_{ij}$, $|\bm{A}| = \sqrt{\bm{A} : \bm{A}}$, and satisfy the usual symmetry conditions of linear elasticity: 
\begin{align*}
\CC_{ijmn} = \CC_{ijnm} = \CC_{jimn}, \quad \CC_{ijmn} = \CC_{mnij}, \quad \forall i,j,m,n \in \{1, \dots, d\},
\end{align*}
while $k : \R \to [0,1]$ satisfies $k \in C^{1,1}(\R)$.
\item \label{ass:Ec} The induced eigenstrain $\E_c : \R \to \R^{d \times d}$ is of the form $\E_c(s) = m(s) \mathbb{I}$ for $m \in W^{1,\infty}(\R)$.
\end{enumerate}

\begin{remark}
The assumption \eqref{ass:ini} is only technical for our consideration of the numerical scheme defined later.  One may consider alternate approaches, such as a Faedo-Galerkin approximation, to show the existence of a solution with regularity stated in Theorem~\ref{thm:conv} under weaker assumptions, such as $\vp_0, \theta_0 \in H^1(\Omega)$.
\end{remark}

\subsection{Scalar auxiliary variable (SAV) formulation}
The weak formulation of \eqref{model} is 
\begin{subequations}\label{model:weak}
\begin{alignat}{2}
0 & = \int_{\Omega_T} \Big ( \alpha \pd_t \vp + \frac{\lambda}{\eps} W'(\vp) + \gamma(\theta - \theta_c) p(\vp) \Big )\psi + \lambda \eps \nabla \vp \cdot \nabla \psi \dxt, \label{original:weak:1} \\
0 & = \int_{\Omega_T} \Big ( \delta \pd_t \theta - \gamma p(\vp) \pd_t \vp \Big ) \psi +  \nabla \theta \cdot \nabla \psi \dxt, \\
0 & = \int_{\Omega_T} \CC(\vp)(\Eu - m(\vp)\mathbb{I} + \beta(\theta - \theta_0) \mathbb{I}) : \E(\bv) \dxt
\end{alignat}
\end{subequations}
holding for all $\psi \in L^2(0,T;H^1(\Omega))$ and $\bv \in L^2(0,T;X(\Omega))$. As $W$ is non-negative, we introduce a scalar auxiliary variable $q$ and a function $Q$ defined as 
\begin{align}\label{defn:Q}
q(t) := \Big ( \frac{1}{\eps}  \int_\Omega W(\vp(t,x)) \dx + 1 \Big )^{1/2}
, \quad Q(\phi) := \Big ( \frac{1}{\eps}  \int_\Omega W(\phi) \dx + 1 \Big )^{1/2},
\end{align}
so that $q(t) = Q(\vp)(t)$ holds. Differentiating in time formally yields an ordinary differential equation
\[
q'(t) = \frac{1}{2 \eps Q(\vp)(t)} \int_\Omega W'(\vp(t,x)) \pd_t \vp(t,x) \dx
\]
furnished with the initial condition $q(0) = Q(\vp_0)$.  Then, an equivalent weak formulation for \eqref{model} based on the scalar auxiliary variable approach of \cite{SAV1,SAV2} is
\begin{subequations}\label{model:weak:SAV}
\begin{alignat}{2}
0 & = \int_{\Omega} \Big ( \alpha \pd_t \vp + \frac{\lambda q}{\eps Q(\vp)} W'(\vp) + \gamma (\theta- \theta_c) p(\vp) \Big )\psi + \lambda \eps \nabla \vp \cdot \nabla \psi \dx, \label{weak:SAV:1} \\
0 & = \int_{\Omega} \Big ( \delta \pd_t \theta - \gamma p(\vp) \pd_t \vp \Big ) \psi + \nabla \theta \cdot \nabla \psi \dx, \label{weak:SAV:2} \\
0 & = q' - \frac{1}{2 \eps Q(\vp)} \int_\Omega W'(\vp) \pd_t \vp \dx, \label{SAV:q} \\
0 & = \int_{\Omega} \CC(\vp)(\Eu - m(\vp)\mathbb{I} + \beta(\theta - \theta_0) \mathbb{I}) : \E(\bv)  \dx, \label{weak:SAV:3}
\end{alignat}
\end{subequations}
holding for a.e.~$t \in (0,T)$, for all  $\psi \in H^1(\Omega)$ and $\bv \in X(\Omega)$. Our numerical discretization of \eqref{model} is based on \eqref{model:weak:SAV}.

\subsection{Fully discrete finite element approximation}
Dividing the time interval $[0,T]$ into a uniform partition of subintervials $[t^{n-1}, t^n]$ for $n = 1, 2, \dots, N_\tau$ with $\tau = t^n - t^{n-1}$ as the time step and $N_\tau \tau = T$. Let $\{\mathcal{T}_h\}_{h > 0}$ denote a regular family of conformal quasiuniform triangulations that partition $\Omega$ into disjoint open simplices $K$ such that $\max_{K \in \Th} \text{diam}(K) \leq h$.  Let $\Sh$ be the finite element space of continuous and piecewise linear functions:
\[
\Sh := \{ f_h \in C^0(\overline{\Omega}) \, : \, f_h \vert_K \in \mathcal{P}_1(K) \quad  \forall K \in \mathcal{T}_h \} \subset H^1(\Omega),
\]
where $\mathcal{P}_1$ is the set of affine linear functions on $K$. Associated to $\Sh$ is the set of basis functions $\{\chi_{k}^{h}\}_{k=1, \dots, Z_h}$, $Z_h := \text{dim}(\Sh)$, that forms a dual basis to the set of nodes $\{\bm{x}_k\}_{k=1, \dots, Z_h}$.  For the approximation of the elasticity system, we introduce the function space:
\[
\bm{\mathcal{S}}_{h,0} = \{ \bm{f}_h \in (\Sh)^d \, : \, \bm{f}_h = \bm{0} \text{ on } \pd \Omega \}.
\]
We recall the nodal interpolation operator $\Ih: C^0(\overline{\Omega}) \to \Sh$ defined as $\Ih (\eta)(\bm{x}) = \sum_{k=1}^{Z_h} \eta(\bm{x}_k) \chi_{k}^h(\bm{x})$, so that $\Ih (\eta)(\bm{x}_k) = \eta(\bm{x}_k)$ for all nodes $\{\bm{x}_k\}_{k=1, \dots, Z_h}$ of $\Th$. 

Let us recall some well-known results concerning the nodal interpolation operator: there exist positive constants $c$ and $C$ independent of $h$, such that 
\begin{align}
c \| f_h \|_{L^p(\Omega)} \leq \big ( \int_\Omega \Ih(|f_h|^p) \dx \big )^{1/p} \leq C \| f_h \|_{L^p(\Omega)} \quad & \forall f_h \in \Sh, \, p \in [1,\infty), \label{norm:equiv} \\[1ex]
\| f - \Ih(f) \| + h \| \nabla (f - \Ih(f)) \| \leq C h^2 \| f \|_{H^2} \quad & \forall f \in H^2(\Omega), \label{Interpol} \\[1ex]
\| f - \Ih(f) \|_{L^q} + h \| \nabla (f - \Ih(f)) \|_{L^q} \leq C h \| f \|_{W^{1,q}} \quad & \forall f \in W^{1,q}(\Omega), \, q \in (2,\infty] \label{Interpol2} \\[1ex]
\lim_{h \to 0} \| f - \Ih(f) \|_{L^\infty} = 0 \quad & \forall f \in C^0(\overline{\Omega}). \label{Interpol:Linfty} 
\end{align}
We also note that for any $\eta \in C^0(\R)$ such that $\eta(s) \in [k_0, k_1]$ with constants $k_0 < k_1$ and $s \in \R$, by the definition of the nodal interpolation operator, for any $\bm{x} \in \overline{\Omega}$,
\begin{align}\label{nodal:bdd}
k_0 = k_0 \sum_{k=1}^{Z_h} \chi_{k}^h(\bm{x}) \leq \Ih(\eta)(\bm{x}) = \sum_{k=1}^{Z_h} \eta(\bm{x}_k) \chi_{k}^h(\bm{x}) \leq k_1 \sum_{k=1}^{Z_h} \chi_{k}^h(\bm{x}) = k_1.
\end{align}
Let us recall the discrete Neumann-Laplacian $\Delta_h: \Sh \to \Sh$ for a function $q_h \in \Sh$ as
\begin{align}\label{disc:lap}
(\Delta_h q_h, \zeta_h) := - (\nabla q_h, \nabla \zeta_h) \quad \forall \zeta_h \in \Sh.
\end{align}
Under \eqref{ass:dom} we have from \cite[Lem.~3.1]{BLN} that for any $f_h \in \Sh$,
\begin{align}\label{BLN:est}
\| \nabla f_h \|_{L^r} \leq C \| \Delta_h f_h \| \text{ where } \begin{cases}
r <\infty & \text{ if } d = 2,\\
r = 6 & \text{ if } d = 3. \end{cases}
\end{align}
For the error analysis in Section \ref{sec:error} we will make use of the Ritz projection operator $R_h : H^1(\Omega) \to \Sh$ defined as
\[
(\nabla f - \nabla (R_h f), \nabla \psi_h) = 0 \quad \forall \psi_h \in \Sh
\]
satisfying $\int_\Omega R_h f \dx = \int_\Omega f \dx$. Then, the following properties hold (see e.g.~\cite{Thomee}):
\begin{align}
\| f - R_h f \| + h \| \nabla (f - R_h f) \| & \leq C h^s \| f \|_{H^s}, \label{Ritz1} \\
\|f - R_h f \|_{L^\infty} & \leq C h^s \ell_h \| f \|_{W^{s,\infty}}, \label{Ritz2}
\end{align}
where $\ell_h := \max(1, \log(1/h))$ and $1 \leq s \leq 2$.

Then, our proposed numerical scheme for \eqref{model} reads as follows: For $n = 1, \dots, N_\tau$, given $(\vp^{n-1}_h, \theta^{n-1}_h, q^{n-1}_h ) \in \Sh \times \Sh\times \R$, find $(\vp^{n}_h, \theta^{n}_h, q^{n}_h) \in \Sh \times \Sh \times \R$ such that for all $\psi_h \in \Sh$,
\begin{subequations}\label{SAV}
\begin{alignat}{2}
0 & = \frac{\alpha}{\tau} (\vp^{n}_h - \vp^{n-1}_h, \psi_h) +   \frac{\lambda q^n_h }{\eps Q_h^{n-1}} (W'(\vp^{n-1}_h), \psi_h) + \gamma((\theta^{n}_h - \theta_c) p(\vp^{n-1}_h), \psi_h)   \label{SAV:1} \\
\notag & \quad + \lambda \eps (\nabla \vp^{n}_h, \nabla \psi_h), \\
0 & = \frac{\delta}{\tau} (\theta^n_h - \theta^{n-1}_h, \psi_h) - \frac{\gamma}{\tau}( p(\vp^{n-1}_h) (\vp^n_h - \vp^{n-1}_h), \psi_h) + (\nabla \theta^n_h, \nabla \psi_h), \label{SAV:2} \\
0 & = q^n_h - q^{n-1}_h - \frac{(W'(\vp^{n-1}_h), \vp^{n}_h - \vp^{n-1}_h)}{2 \eps Q^{n-1}_h}, \label{SAV:3}
\end{alignat}
\end{subequations}
where
\begin{align}\label{defn:Qk}
Q^k_h := \Big ( \frac{1}{\eps} \int_\Omega W(\vp^k_h) \dx + 1 \Big)^{1/2}.
\end{align}
This fully discrete finite element scheme is linear with respect to $(\vp^n_h, \theta^n_h, q^n_h)$ and can be initialize with the choices $\vp_h^0 = R_h \vp_0$, $\theta_h^0 = R_h\theta_0$ and $q_h^0 = Q(\vp_h^0)$. Then, we find $\bu^n_h \in \bm{\mathcal{S}}_{h,0}$ such that for all $\bv_h \in \bm{\mathcal{S}}_{h,0}$,
\begin{align}
\label{SAV:u} \Big ( \Ih(\CC(\vp^{n}_h)) \E(\bu^{n}_h) , \E(\bv_h) \Big ) & = \Big ( \Ih(k(\vp^{n}_h) (m(\vp^{n}_h) - \beta(\theta^n_h - \theta^0_h))) \CC^{(1)} \mathbb{I}, \E(\bv_h) \Big ) \\
\notag & \quad + \Big ( \Ih([1-k(\vp^{n}_h)] (m(\vp^{n}_h) - \beta(\theta^n_h - \theta^0_h))) \CC^{(0)} \mathbb{I}, \E(\bv_h) \Big ).
\end{align}


\subsection{Unique solvability}\label{sec:UniqSolv}
\begin{prop}
For any $n = 1, \dots, N_\tau$, given $(\vp^{n-1}_h, \theta^{n-1}_h, q^{n-1} ) \in \Sh \times \Sh \times \R$, there exists a unique quadruple $(\vp^{n}_h, \theta^{n}_h, q^{n}_h, \bu^n_h) \in \Sh \times \Sh \times \R \times \bm{\mathcal{S}}_{h,0}$ satisfying \eqref{SAV}-\eqref{SAV:u}. 
\end{prop}

\begin{proof}
Suppose we have two quadruple of solutions $\{(\vp^{n}_{h,i}, \theta^{n}_{h,i}, q^{n}_{h,i}, \bu^n_{h,i})\}_{i=1,2}$ satisfying \eqref{SAV}-\eqref{SAV:u}.  Denoted their differences by $(\vp, \theta, q, \bu) \in \Sh \times \Sh \times \R \times \bm{\mathcal{S}}_{h,0}$, we see that $(\vp, \theta, q)$ fulfill
\begin{subequations}\label{SAV:diff}
\begin{alignat}{2}
0 & = \frac{\alpha}{\tau}( \vp, \psi_h) + \frac{\lambda q}{Q^{n-1}_h} (W'(\vp^{n-1}_h) , \psi_h) + \gamma (p(\vp_h^{n-1}) \theta, \psi_h ) +  \lambda \eps (\nabla \vp, \nabla \psi_h), \label{diff:1} \\
0 & = \delta(\theta, \psi_h) - \gamma (p(\vp^{n-1}_h) \vp, \psi_h) + \tau (\nabla \theta, \nabla \psi_h), \label{diff:2}  \\
0 & = q - \frac{(W'(\vp^{n-1}_h), \vp)}{2 \eps Q^{n-1}_h}. \label{diff:3}
\end{alignat}
\end{subequations}
Choosing $\psi_h = \vp$ in \eqref{diff:1}, $\psi_h = \theta$ in \eqref{diff:2}, and multiplying \eqref{diff:3} with $2 \lambda q$, then summing the resulting equalities gives
\[
\frac{\alpha}{\tau} \| \vp \|^2 + 2\lambda |q|^2 +  \delta \| \theta \|^2 +\lambda \eps \| \nabla \vp \|^2 + \tau \| \nabla \theta \|^2 = 0.
\]
It is clear that $\vp = \theta = q = 0$, i.e., $\vp^n_{h,1} = \vp^n_{h,2}$, $\theta^n_{h,1} = \theta^n_{h,2}$, and $q^n_{h,1} = q^n_{h,2}$.  Then, the difference $\bu$ satisfies
\[
0 = \Big ( \Ih(\CC(\vp^{n}_h)) \E(\bu) , \E(\bv_h) \Big ).
\]
Choosing $\bv_h = \bu$ yields
\begin{align*}
0  & = \int_\Omega (1-\Ih(k(\vp^n_h))) \CC^{(0)} \Eu : \Eu + \Ih(k(\vp^n_h)) \CC^{(1)} \Eu : \Eu \dx \\
& \geq  \int_\Omega (1-\Ih(k(\vp^n_h))) c_4 |\Eu|^2 + \Ih(k(\vp^n_h)) c_6 |\Eu|^2 \dx \geq \min(c_4, c_6) \| \Eu \|^2,
\end{align*}
where we have used that $0 \leq \Ih(k(\vp^n_h)) \leq 1$ inferred from \eqref{ass:CC} and \eqref{nodal:bdd}. Hence, $\Eu = \bm{0}$ and by Korn's inequality we deduce that $\bm{u} = \bm{0}$, i.e., $\bm{u}_{h,1}^n = \bm{u}_{h,2}^n$.  This gives uniqueness of the fully discrete solutions.

For existence, we express \eqref{SAV} in a matrix-vector form.  Using the basis functions $\{\chi_k^h\}_{k=1, \dots, Z_h}$, we introduce the mass matrix and stiffness matrix
\[
\bM_{i,j} = (\chi_{i}^h, \chi_{j}^h), \quad \bS_{i,j} = (\nabla \chi_{i}^h, \nabla \chi_{j}^h).
\]
With the dual basis of nodes, we introduce 
\begin{align*}
\bm{\vp}_h^n & := \bM^{-1} [(\vp_h^n, \chi_{k}^h)]_{k=1}^{Z_h} = (\vp_h^n(\bm{x}_1), \dots, \vp_h^n(\bm{x}_{Z_h}))^{\top},\\
\bm{\theta}_h^n & := \bM^{-1} [(\theta_h^n, \chi_{k}^h)]_{k=1}^{Z_h} = (\theta_h^n(\bm{x}_1), \dots, \theta_h^n(\bm{x}_{Z_h}))^{\top}, \\
\bm{b}^{n-1} & := \frac{\lambda W'(\bm{\vp}_h^{n-1})}{\eps Q_h^{n-1}}, \quad \bm{c}^{n-1} := \gamma \theta_c \bM^{-1} [ (p(\vp_h^{n-1}), \chi_{k}^h)]_{k=1}^{Z_h}, \quad \bP^{n-1}_{i,j} = (p(\vp_h^{n-1}) \chi_i^h, \chi_j^h),
\end{align*}
where the nonlinearities are applied component-wise. Then, choosing $\psi_h = \chi_{k}^h$ for $k = 1, \dots, Z_h$ in \eqref{SAV} leads to
\begin{subequations}\label{SAV:matrix}
\begin{alignat}{2}
0 & = \alpha \bM (\bm{\vp}_h^n - \bm{\vp}_h^{n-1}) + \tau q^n_h \bM \bm{b}^{n-1} + \tau \gamma \bP^{n-1} \bm{\theta}_h^n - \tau \bm{c}^{n-1}+ \lambda \eps \tau \bS \bm{\vp}_h^n, \label{SAV:matrix:1} \\
0 & = \delta \bM (\bm{\theta}_h^n - \bm{\theta}_h^{n-1}) - \gamma \bP^{n-1}(\bm{\vp}_h^n - \bm{\vp}_h^{n-1}) + \tau \bS \bm{\theta}_h^n, \label{SAV:matrix:2} \\
0 & = q^n_h - q^{n-1}_h -  \tfrac{1}{2} \bM \bm{b}^{n-1} \cdot (\bm{\vp}_h^n - \bm{\vp}_h^{n-1}). \label{SAV:matrix:3}
\end{alignat}
\end{subequations}
Expressing \eqref{SAV:matrix:2} as
\begin{align}\label{SAV:matrix:2:alt}
\bm{\theta}_h^n = (\delta \bM + \tau \bS)^{-1} (\delta \bM \bm{\theta}_h^{n-1} + \gamma \bP^{n-1}(\bm{\vp}_h^n - \bm{\vp}_h^{n-1}))
\end{align}
and substituting this and \eqref{SAV:matrix:3} into \eqref{SAV:matrix:1} yields the following:
\begin{align*}
& (\alpha \bM + \lambda \eps \tau \bS + \tau \gamma^2 \bP^{n-1}( \delta \bM + \tau \bS)^{-1} \bP^{n-1}) \bm{\vp}_h^n + \frac{\tau}{2} (\bM \bm{b}^{n-1} \cdot \bm{\vp}_h^{n-1}) \bM \bm{b}^{n-1}  \\
& \quad = \alpha \bM \bm{\vp}_h^{n-1} - \tau q^{n-1}_h \bM \bm{b}^{n-1} + \frac{\tau}{2} (\bM \bm{b}^{n-1} \cdot \bm{\vp}_h^{n-1}) \bM \bm{b}^{n-1} + \tau \bm{c}^{n-1}  \\
& \qquad - \tau \gamma \bP^{n-1}(\delta \bM + \tau \bS)^{-1} (\delta \bM \bm{\theta}_h^{n-1} - \gamma \bP^{n-1} \bm{\vp}_{h}^{n-1}) \\
& \quad =: \bm{d}^{n-1}.
\end{align*}
We define the matrix $\bX^{n-1} := (\alpha \bM + \lambda \eps \tau \bS + \tau \gamma^2 \bP^{n-1}( \delta \bM + \tau \bS)^{-1} \bP^{n-1})$, which can be verified to be positive definite. Then, applying the inverse of $\bX^{n-1}$ to both sides, and taking the product of the resulting equality with $\bM \bm{b}^{n-1}$ yields
\[
(\bM \bm{b}^{n-1} \cdot \bm{\vp}_h^{n}) \Big ( 1 + \frac{\tau}{2} (\bX^{n-1})^{-1} \bM \bm{b}^{n-1} \cdot \bM \bm{b}^{n-1} \Big ) = (\bX^{n-1})^{-1} \bm{d}^{n-1}.
\]
This yields an expression for $\bM \bm{b}^{n-1} \cdot \bm{\vp}_h^n$ as
\begin{align*}
(\bM \bm{b}^{n-1} \cdot \bm{\vp}_h^{n}) = \frac{(\bX^{n-1})^{-1} \bm{d}^{n-1}}{1 + \frac{\tau}{2} (\bX^{n-1})^{-1} \bM \bm{b}^{n-1} \cdot \bM \bm{b}^{n-1}},
\end{align*}
which then provides the update formula for $\bm{\vp}_h^n$,
\begin{align*}
\bm{\vp}_h^n = (\bX^{n-1})^{-1} \bm{d}^{n-1} - \frac{\tau}{2}(\bM \bm{b}^{n-1} \cdot \bm{\vp}_h^{n}) \bM \bm{b}^{n-1},
\end{align*}
while $\bm{\theta}_h^n$ and $q^n_h$ can be computed via \eqref{SAV:matrix:2:alt} and \eqref{SAV:matrix:3}, respectively.  This yields the existence of the discrete solutions $(\vp_h^n, \theta_h^n, q^n_h)$ for \eqref{SAV}.  

For the existence of $\bu_h^n \in \bm{\mathcal{S}}_{h,0}$ it suffices to show that when expressing \eqref{SAV:u} into matrix-vector form its left-hand side involves an invertible matrix with the global nodal displacement vector $\bm{U} \in \R^{d\cdot Z_{h,0}}$ defined as the concatenation of the (interior) nodal evaluations of $\bu_h^n$:
\[
\bm{U} = (u_{1,1}, u_{2,1}, \dots, u_{d,1}, u_{1,2}, \dots, u_{d,2}, \dots, u_{1,Z_{h,0}}, \dots, u_{d,Z_{h,0}})^{\top},
\]
where $\mathrm{dim}(\bm{\mathcal{S}}_{h,0}) = Z_{h,0}$. We focus on the case $d = 3$ as the case $d = 2$ can be treated similarly. For a fourth order tensor $\mathbb{C}$ and a second order tensor $\bm{A}$ we denote their corresponding representation in Voigt notation as $\widehat{\mathbb{C}}$ and $\widehat{\bm{A}}$, respectively. We introduce a geometrical matrix $\mathbb{B} \in \R^{6 \times 3Z_{h,0}}$ of the form $\mathbb{B} = \begin{pmatrix} \mathbb{B}_1 & \mathbb{B}_2 & \cdots & \mathbb{B}_{Z_{h,0}}  \end{pmatrix}$, such that 
\[
\widehat{\E(\bu_h^n)} = \begin{pmatrix} \eps_{1,1} \\ \eps_{2,2} \\ \eps_{3,3} \\ \tfrac{1}{2}(\eps_{2,3} + \eps_{3,2}) \\ \tfrac{1}{2}(\eps_{1,3} + \eps_{3,1}) \\ \tfrac{1}{2}(\eps_{1,2} + \eps_{2,1}) \end{pmatrix} = \mathbb{B} \bm{U}, \quad \text{ where } \mathbb{B}_i = \begin{pmatrix} 
\pd_1 \chi_i^h & 0 & 0 \\ 0 & \pd_2 \chi_i^h & 0 \\ 0 & 0 & \pd_3 \chi_i^h \\ 0 & \tfrac{1}{2} \pd_3 \chi_i^h & \tfrac{1}{2} \pd_2 \chi_i^h \\ \tfrac{1}{2} \pd_3 \chi_i^h & 0 & \tfrac{1}{2} \pd_1 \chi_i^h \\ \tfrac{1}{2} \pd_2 \chi_i^h & \tfrac{1}{2} \pd_1 \chi_i^h & 0
\end{pmatrix}
\]
for $i \in \{1, \dots, Z_{h,0}\}$ and $\eps_{i,j}$ denotes the $(i,j)$th entry of the second order tensor $\E(\bu_h^n)$. Then, by choosing $\bv_h$ as appropriate combinations of the basis functions in \eqref{SAV:u}, we obtain the following linear system 
\[
\mathbb{K}_n \bm{U} = \bm{F}_n
\]
with global stiffness matrix $\mathbb{K}_n$ and load vector $\bm{F}_n$ defined as
\begin{align*}
(\mathbb{K}_n)_{i,j} &  = (k(\vp_h^n),1)^h [\mathbb{B}^{\top} \widehat{\CC^{(1)}} \mathbb{B}]_{i,j} + (1-k(\vp_h^n),1)^h [\mathbb{B}^{\top} \widehat{\CC^{(0)}} \mathbb{B}]_{i,j}, \\
(\bm{F}_n)_{k} & = (k(\vp_h^n)(m(\vp_h^n) - \beta(\theta_h^n - \theta_h^0)), 1)^h \; (\mathbb{B}^{\top} \widehat{\CC^{(1)}} \widehat{\mathbb{I}})_k \\
& \quad + ((1- k(\vp_h^n)(m(\vp_h^n) - \beta(\theta_h^n - \theta_h^0)), 1)^h  \; (\mathbb{B}^{\top} \widehat{\CC^{(0)}} \widehat{\mathbb{I}})_k,
\end{align*}
for $1 \leq i, j ,k \leq 3Z_{h,0}$. In the above we used the notation $(f,g)^h := \int_\Omega \Ih(fg) \, dx$. It is straightforward to see that for any arbitrary vector $\bm{V} \in \R^{3 Z_{h,0}}$ corresponding to $\bm{v} \in  \bm{\mathcal{S}}_{h,0}$ via the relation $\widehat{\E(\bv)} = \mathbb{B} \bm{V}$, it holds that 
\[
\bm{V}^{\top} \mathbb{K}_n \bm{V} = (\Ih(\CC(\vp_h^n)) \E(\bv), \E(\bv)).
\]
Hence, using that $\Ih(\CC(\vp_h^n))$ is positive definite on symmetric matrices we deduce that $\mathbb{K}_n$ is invertible, which in turn leads to the existence of a solution $\bm{u}_h^n$ satisfying \eqref{SAV:u}.
\end{proof}
\begin{remark}
The computational cost for each update step in \eqref{SAV} amounts to solving three linear systems: $(\delta \bM + \tau \bS)^{-1}$ in $\bX^{n-1}$, $(\bX^{n-1})^{-1} \bm{d}^{n-1}$ and $(\bX^{n-1})^{-1} \bM \bm{b}^{n-1}$. Once these three quantities are computed, the update for $(\bm{\vp}_h^n, \bm{\theta}_h^n, q^n_h)$ no longer involve any matrix inversions.
\end{remark}

\subsection{Stability estimates}
For a better presentation, we set the constants $\alpha$, $\beta$, $\gamma$, $\delta$, $\lambda$ and $\eps$ to be equal to 1 as their values have no bearing on the analysis.
\begin{lem}\label{lem:stab}
There exists a positive constant $C$ depending only on model parameters, such that for any $\tau \in (0,1)$ and $n = 1, \dots, N_\tau$, the following estimate holds for the discrete solutions $\{\vp_h^n, \theta_h^n, q^n_h, \bu_h^n\}$ to \eqref{SAV}-\eqref{SAV:u}:
\begin{equation}\label{stab:est}
\begin{aligned}
& \max_{k = 1, \dots, N_\tau} \Big ( \| \vp_h^k \|_{H^1}^2 + \| \theta_h^k \|_{H^1}^2 + \| \bu_h^k \|_{H^1}^2 + |q^k_h |^2 \Big )\\
& \quad + \sum_{n=1}^{N_\tau} \tau \Big ( \| \Delta_h \vp_h^n \|^2 + \| \vp_h^n \|_{W^{1,r}}^2 + \| \Delta_h \theta_h^n \|^2 + \| \theta_h^n \|_{W^{1,r}}^2 \Big )  \\
& \quad + \sum_{n=1}^{N_\tau} \Big ( \| \vp_h^n - \vp_h^{n-1} \|_{H^1}^2 + \| \theta_h^n - \theta_h^{n-1} \|_{H^1}^2 + |q^n_h - q^{n-1}_h|^2 \Big )\\
& \quad + \sum_{n=1}^{N_\tau} \frac{1}{\tau} \Big (  \| \vp_h^n - \vp_h^{n-1} \|^2 + \| \theta_h^n - \theta_h^{n-1} \|^2 \Big ) \leq C,
\end{aligned}
\end{equation}
with exponent $r$ as in \eqref{BLN:est}, and for arbitrary $l \in \{1, \dots, N_\tau\}$,
\begin{align}\label{stab:est:2}
\sum_{n=0}^{N_\tau - l} \tau \Big ( \| \vp_h^{n+l} - \vp_h^n \|^2 + \| \theta_h^{n+l} - \theta_h^n \|^2 \Big ) \leq C l \tau.
\end{align}
\end{lem}
\begin{proof}
In the sequel the symbol $C$ denote positive constants independent of $h$, $\tau$ and $n \in \{1, \dots, N_\tau\}$, whose values may change line from line and within the same line.\\

\noindent {\bf First estimate.} Taking $\psi_h = \vp_h^n - \vp_h^{n-1}$ in \eqref{SAV:1}, $\psi_h = \tau (\theta_h^n - \theta_c)$ in \eqref{SAV:2} and multiplying \eqref{SAV:3} with $2 q^n_h$, upon summing and using the identity $(a-b)a = \frac{1}{2}( a^2 - b^2 + (a-b)^2)$ leads to 
\begin{align*}
& \frac{1}{2} \Big (\| \nabla \vp_h^n\|^2 - \| \nabla \vp_h^{n-1} \|^2 + \| \nabla (\vp_h^n - \vp_h^{n-1}) \|^2 \Big ) +  \Big ( |q^n_h|^2 - |q^{n-1}_h|^2 + |q^n_h - q^{n-1}_h |^2 \Big ) \\
& \quad + \frac{1}{2} \Big ( \| \theta_h^n - \theta_c \|^2 - \| \theta_h^{n-1} - \theta_c \|^2 + \| \theta_h^n - \theta_h^{n-1} \|^2 \Big ) + \tau \| \nabla \theta_h^n \|^2 + \frac{1}{\tau} \| \vp_h^n - \vp_h^{n-1} \|^2 = 0.
\end{align*}
Summing from $n = 1$ to $n = k$ for arbitrary $k \in \{1, \dots, N_\tau\}$ and using \eqref{Ritz1} yields
\begin{equation}\label{stab:1}
\begin{aligned}
& \frac{1}{2} \Big ( \| \nabla \vp_h^k \|^2 + \sum_{n=1}^k \| \nabla (\vp_h^n - \vp_h^{n-1}) \|^2 \Big ) +  \Big ( |q^k_h|^2 + \sum_{n=1}^k |q^n_h - q^{n-1}_h |^2 \Big ) \\
& \qquad + \frac{1}{2} \Big ( \| \theta_h^k - \theta_c \|^2 + \sum_{n=1}^k \| \theta_h^n - \theta_h^{n-1} \|^2 \Big ) + \sum_{n=1}^k \tau \| \nabla \theta_h^n \|^2 + \sum_{n=1}^k \frac{1}{\tau} \| \vp_h^n - \vp_h^{n-1} \|^2 \\
& \quad = \frac{1}{2} \| \nabla \vp_h^0 \|^2 +  |q^0_h|^2 + \frac{\delta}{2} \|\theta_h^0 - \theta_c \|^2 \leq C \| \vp_0 \|_{H^2}^2 + C \| \theta_0 - \theta_c \|_{H^2}^2 \leq C.
\end{aligned}
\end{equation} 

\noindent {\bf Second estimate.} Taking $\psi_h = \vp_h^n$ in \eqref{SAV:1} gives
\begin{align*}
& \frac{1}{2} \Big ( \| \vp_h^n \|^2 - \| \vp_h^{n-1} \|^2 + \| \vp_h^n - \vp_h^{n-1} \|^2 \Big ) + \tau \| \nabla \vp_h^n \|^2 \\
& \quad =- \tau  \frac{q^n_h}{Q_h^{n-1}}  (W'(\vp_h^{n-1}), \vp_h^n) - \tau  (p(\vp_h^{n-1})(\theta_h^n - \theta_c), \vp_h^n) \\
& \quad = - 2 \tau |q^n_h|^2 + 2 \tau q^n_h q^{n-1}_h - \tau  \frac{q^n_h (W'(\vp_h^{n-1}), \vp_h^{n-1})}{Q_h^{n-1}} -  \tau (p(\vp_h^{n-1})(\theta_h^n - \theta_c), \vp_h^n),
\end{align*}
where we have used \eqref{SAV:3}. From \eqref{ass:W} it follows that there exists a positive constant $C$ such that 
\[
|W'(s)s| \leq C(1+|s|^4) \leq C(1 + W(s)) \quad \forall s \in \R,
\]
and so 
\[
\Big | \frac{(W'(\vp_h^{n-1}), \vp_h^{n-1})}{Q_h^{n-1}}  \Big | \leq \frac{C}{Q_h^{n-1}} \Big (\int_\Omega W(\vp_h^{n-1}) \dx + 1 \Big ) \leq C.
\]
Then, using the boundedness of $|q^k_h|^2$ and $\|\theta_h^k - \theta_c\|^2$ from \eqref{stab:1} and also \eqref{ass:p} we see that 
\begin{align*}
& \frac{1}{2} \Big ( \| \vp_h^n \|^2 - \| \vp_h^{n-1} \|^2 + \| \vp_h^n - \vp_h^{n-1} \|^2 \Big ) +  \tau  \| \nabla \vp_h^n \|^2 \leq C\tau + \frac{1}{2} \tau \| \vp_h^n \|^2.  
\end{align*}
Summing from $n = 1$ to $n = k$ for arbitrary $k \in \{1, \dots, N_\tau\}$ yields
\[
\frac{1}{2}(1 - \tau) \| \vp_h^k \|^2 + \sum_{n=1}^k \frac{1}{2} \| \vp_h^n -\vp_h^{n-1} \|^2 +  \tau  \sum_{n=1}^k \| \nabla \vp_h^n\|^2 \leq C\sum_{n=1}^k \tau + C \sum_{n=0}^{k-1} \tau \| \vp_h^{n} \|^2.
\]
For $\tau < 1$, invoking the discrete Gronwall inequality provides
\begin{align}\label{stab:2}
\| \vp_h^k \|_{H^1}^2 + \sum_{n=1}^k \| \vp_h^n - \vp_h^{n-1} \|^2 \leq C.
\end{align}

\noindent {\bf Third estimate.} Taking $\psi_h = \theta_h^n - \theta_h^{n-1}$ in \eqref{SAV:2} gives
\begin{align*}
& \frac{1}{\tau} \| \theta_h^n - \theta_h^{n-1} \|^2 + \frac{1}{2} \Big ( \|\nabla \theta_h^n\|^2  + \| \nabla (\theta_h^{n} - \theta_h^{n-1}) \|^2 - \| \nabla \theta_h^{n-1} \|^2 \Big ) \\
& \quad = \frac{1}{\tau} (p(\vp_h^{n-1})(\vp_h^{n} - \vp_h^{n-1}), \theta_h^{n} - \theta_h^{n-1})  \leq \frac{C}{\tau} \| \vp_h^{n} - \vp_h^{n-1} \|^2 + \frac{1}{2\tau} \|\theta_h^n - \theta_h^{n-1} \|^2.
\end{align*}
Summing from $n = 1$ to $n = k$ for arbitrary $k \in \{1, \dots, N_\tau\}$ and applying \eqref{stab:1} yields
\begin{equation}\label{stab:3}
\begin{aligned}
\| \nabla \theta_h^k \|^2 + \sum_{n=1}^k \| \nabla (\theta_h^n - \theta_h^{n-1}) \|^2 + \sum_{n=1}^k \frac{1}{\tau} \| \theta_h^n - \theta_h^{n-1} \|^2 \leq C.
\end{aligned}
\end{equation}

\noindent {\bf Fourth estimate.} Using the discrete Neumann-Laplacian \eqref{disc:lap}, equations \eqref{SAV:1} and \eqref{SAV:2} can be expressed as
\begin{align*}
0 & = \Big (  \frac{\vp^{n}_h - \vp^{n-1}_h}{\tau} +  \frac{q^n_h  W'(\vp^{n-1}_h)}{\eps Q^{n-1}_h} + (\theta^{n}_h - \theta_c) p(\vp^{n-1}_h) -  \Delta_h \vp_h^n, \psi_h \Big ),\\
0 & = \Big (  \frac{\theta^n_h - \theta^{n-1}_h}{\tau} -  p(\vp^{n-1}_h) \frac{\vp^n_h - \vp^{n-1}_h}{\tau} - \Delta_h \theta_h^n, \psi_h \Big ).
\end{align*}
Employing \eqref{stab:1} for $q^k_h$ and $\|\theta_h^k - \theta_c \|^2$, \eqref{stab:2} for $\vp_h^k$, and \eqref{ass:p}, upon choosing $\psi_h = - \Delta_h \vp_h^n$ and $\psi_h = -\Delta_h \theta_h^n$, respectively, we obtain
\begin{align*}
\| \Delta_h \vp_h^n \| \leq \frac{1}{\tau} \| \vp_h^n - \vp_h^{n-1} \| + C, \quad
\| \Delta_h \theta_h^n \| \leq \frac{1}{\tau}\| \theta_h^n - \theta_h^{n-1} \| + \frac{C}{\tau} \| \vp_h^n - \vp_h^{n-1} \|.
\end{align*}
Squaring and multiplying by $\tau$ on both sides, then summing from $n = 1$ to $n = k$ for arbitrary $k \in \{1, \dots, N_\tau\}$ leads to 
\begin{align}\label{stab:4}
\sum_{n=1}^k \tau \| \Delta_h \vp_h^n \|^2 + \sum_{n=1}^k \tau \| \Delta_h \theta_h^n \|^2 \leq C.
\end{align}
By \eqref{BLN:est} we also infer for $r <\infty$ if $d = 2$ and $r = 6$ if $d = 3$,
\begin{align}\label{stab:4a}
\sum_{n=1}^k \tau \| \nabla \vp_h^n \|_{L^r}^2 + \sum_{n=1}^k \tau \| \nabla \theta_h^n \|_{L^r}^2 \leq C.
\end{align}

\noindent {\bf Fifth estimate.} For $l = 1, \dots, N_\tau$ and $m = 0, \dots, N_\tau - l$, we consider $\psi_h = \tau (\vp_h^{m+l} - \vp_h^m)$ in \eqref{SAV:1}, leading to 
\begin{align*}
0 & =  (\vp_h^n - \vp_h^{n-1}, \vp_h^{m+l} - \vp_h^m) +  \tau (\nabla \vp_h^n, \nabla (\vp_h^{m+l} - \vp_h^m)) \\
& \quad + \tau  \frac{q^n_h}{Q_h^{n-1}} ( W'(\vp_h^{n-1}), \vp_h^{m+l} - \vp_h^m) +  \tau ((\theta_h^n - \theta_c) p(\vp_h^{n-1}), \vp_h^{m+l} - \vp_h^m ).
\end{align*}
Summing this identity from $n = m+1$ to $m+l$ gives
\begin{align*}
\|\vp_h^{m+l} - \vp_h^m \|^2 & \leq C\tau \sum_{n=m+1}^{m+l} \Big ( \|  \vp_h^n \|_{H^1} + \| \vp_h^{n-1} \|_{H^1}^3 + |\theta_h^n - \theta_c|_h \Big ) \| \vp_h^{m+l} - \vp_h^m \|_{H^1}  \\
& \leq C l \tau \| \vp_h^{m+l} - \vp_h^m \|_{H^1} \leq Cl \tau,
\end{align*}
on account of \eqref{stab:1} and \eqref{stab:2}. Hence, we obtain after multiplying by $\tau$ and summing from $m =0$ to $N_\tau - l$ that
\begin{align}\label{stab:5}
\tau \sum_{m=0}^{N_\tau - l} \| \vp_h^{m + l} - \vp_h^m \|^2 \leq C l \tau.
\end{align}
Similarly, by considering $\psi_h = \tau (\theta_h^{m+l} - \theta_h^m)$ in \eqref{SAV:2} we have
\begin{align*}
0 & =  (\theta_h^n - \theta_h^{n-1}, \theta_h^{m+l} - \theta_h^m) + \tau (\nabla \theta_h^n, \nabla (\theta_h^{m+l} - \theta_h^{m})) -  (p(\vp_h^{n-1})(\vp_h^{n} - \vp_h^{n-1}), \theta_h^{m+l} - \theta_h^m).
\end{align*}
Then, summing from $n = m+1$ to $m + l$ gives
\begin{align*}
 \| \theta_h^{m+l} - \theta_h^m \|^2 & \leq C \tau \sum_{n=m+1}^{m+l} \Big ( \| \nabla \theta_h^n \| + \| \vp_h^n - \vp_h^{n-1} \| \Big )\| \theta_h^{m+l} - \theta_h^m \|_{H^1} \\
& \leq C l \tau \| \theta_h^{m+l} - \theta_h^{m} \|_{H^1} \leq C l \tau,
\end{align*}
on account of \eqref{stab:1} and \eqref{stab:3}. Hence, analogous to \eqref{stab:5} we obtain for any $l = 1, \dots, N_\tau$,
\begin{align}\label{stab:5a}
\tau \sum_{m=0}^{N_\tau - l} \| \theta_h^{m + l} - \theta_h^m \|^2 \leq C l \tau.
\end{align}

\noindent {\bf Sixth estimate.} Turning to \eqref{SAV:u}, choosing $\bv_h = \bu_h^n$ yields with \eqref{ass:CC} and \eqref{nodal:bdd}
\begin{align*}
\min(c_4,c_6) \| \E(\bu_h^n) \|^2 \leq C \| \E(\bu_h^n) \| \Big ( 1 + \| \theta_h^n - \theta_h^0 \| \Big ).
\end{align*}
By \eqref{stab:1} and Korn's inequality we deduce that for arbitrary $k \in \{1, \dots, N_\tau\}$,
\begin{align}\label{stab:6}
\| \bu_h^k \|_{H^1} \leq C.
\end{align}
\end{proof}

\subsection{Compactness and convergence of fully discrete solutions}
Let us recall the following compactness result from \cite[Sec.~8, Cor.~4]{Simon}:  For Banach spaces $X$, $B$ and $Y$ with compact embedding $X \subset \subset B$ and continuous embedding $B \subset Y$, then 
\begin{align}
\{ \zeta \in L^p(0,T;X) \, : \, \pd_t \eta \in L^1(0,T;Y) \} &\subset \subset L^p(0,T;B) \quad  &&\text{ for any } 1 \leq p < \infty, \label{Simon:1} \\
\notag \{ \zeta \in L^\infty(0,T;X) \, : \, \pd_t \eta \in L^r(0,T;Y)\} &\subset \subset C^0([0,T];B) \quad && \text{ for any } r > 1.
\end{align}
Moreover, we also require \cite[Sec.~8, Thm.~5]{Simon}: if $F$ is a bounded set in $L^p(0,T;X)$ for $1 \leq p \leq \infty$ and 
\begin{align}\label{Simon:2}
\| f(\cdot + s) - f(\cdot) \|_{L^p(0,T-s;Y)} \to 0 \text{ as } s \to 0 
\end{align}
uniformly for $f \in F$, then $F$ is relatively compact in $L^p(0,T;B)$ if $1 \leq p < \infty$, and in $C^0([0,T];B)$ if $p = \infty$.

Introducing the piecewise linear and piecewise constant extensions of time discrete functions $\{a^n\}_{n=0}^{N_\tau}$:
\begin{equation*}
\begin{alignedat}{2}
a^\tau(\cdot, t) & := \frac{t - t^{n-1}}{\tau} a^n(\cdot) + \frac{t^n - t}{\tau} a^{n-1}(\cdot) && \text{ for } t \in [t^{n-1}, t^n], \quad n \in \{1, \dots, N_\tau \}, \\
a^{-}(\cdot, t) & := a^{n-1}(\cdot) && \text{ for } t \in (t^{n-1}, t^n], \quad n \in \{1, \dots, N_\tau \}, \\
a^{+}(\cdot, t) &:= a^{n}(\cdot) && \text{ for } t \in (t^{n-1}, t^n], \quad n \in \{1, \dots, N_\tau \}.
\end{alignedat}
\end{equation*}
Then, multiplying \eqref{SAV:1}, \eqref{SAV:2} and \eqref{SAV:u} by $\tau$ and summing from $n = 1, \dots, N_\tau$, we obtain for arbitrary test functions $\psi_h \in L^2(0,T;\Sh)$ and $\bv_h \in L^2(0,T;\bm{\mathcal{S}}_{h,0})$ that
\begin{subequations}\label{cts:time}
\begin{alignat}{2}
0 & = \int_{\Omega_T} \Big (\pd_t \vp_h^\tau + \frac{q^+_h W'(\vp_h^-)}{Q(\vp_h^-)} +  (\theta_h^+ - \theta_c) p(\vp_h^-) \Big ) \psi_h + \nabla \vp_h^+ \cdot \nabla \psi_h \dxt, \label{cts:SAV:1} \\
0 & = \int_{\Omega_T} \Big ( \pd_t \theta_h^\tau - p(\vp_h^-) \pd_t \vp_h^\tau \Big ) \psi_h + \nabla \theta_h^+ \cdot \nabla \psi_h \dxt, \label{cts:SAV:2} \\
0 & = \int_{\Omega_T} \Big ( \Ih[\CC(\vp_h^+) ] \E(\bu_h^+) - \Ih \Big [ \CC(\vp_h^+) ((\theta_h^+ - \theta_h^0) - m(\vp_h^+)) \mathbb{I} \Big ] \Big ) : \E(\bv_h) \dxt. \label{cts:SAV:3}
\end{alignat}
\end{subequations}
Note that for $t \in (t^{n-1}, t^n]$, we have the relations
\[
a^\tau(t) - a^-(t) = \frac{t-t^{n-1}}{\tau} (a^n - a^{n-1}), \quad a^+(t) - a^\tau(t) = \frac{t^n - t}{\tau}(a^n - a^{n-1}),
\]
and together with Lemma \ref{lem:stab} we deduce that (using the notation $a^{\tau,\pm}$ to denote $\{a^\tau, a^{-}, a^+\}$ and $a^{\pm}$ to denote $\{a^-, a^+\}$)  
\begin{equation}\label{unif:bdd}
\begin{aligned}
& \| \vp_h^{\tau,\pm} \|_{L^\infty(0,T;H^1)}^2 + \|q^{\tau,\pm}_h\|_{L^\infty(0,T)}^2 + \| \theta_h^{\tau,\pm} \|_{L^\infty(0,T;H^1)}^2 + \| \bu_h^{\tau,\pm} \|_{L^\infty(0,T;H^1)}^2\\
& \qquad + \| \pd_t \vp_h^\tau \|_{L^2(0,T;L^2)}^2 + \| \pd_t \theta_h^{\tau} \|_{L^2(0,T;L^2)}^2  + \| \Delta_h \vp_h^{\pm} \|_{L^2(0,T;L^2)}^2 + \| \Delta_h \theta_h^{\pm} \|_{L^2(0,T;L^2)}^2 \\
& \qquad + \| \vp_h^{\tau,\pm} \|_{L^2(0,T;W^{1,r})}^2 + \| \theta_h^{\tau,\pm} \|_{L^2(0,T;W^{1,r})}^2 \\
& \qquad + \frac{1}{\tau} \| \vp_h^\tau - \vp_h^{\pm} \|_{L^2(0,T;H^1)}^2 + \frac{1}{\tau} \| \theta_h^\tau - \theta_h^{\pm} \|_{L^2(0,T;H^1)}^2 + \frac{1}{\tau} \| q^\tau_h - q^{\pm}_h \|_{L^2(0,T)}^2 \\
& \quad \leq C,
\end{aligned}
\end{equation}
while \eqref{stab:est:2} translates to 
\begin{align}\label{time:trans}
\| \vp_h^{\tau, \pm}(\cdot  + l\tau) - \vp_h^{\tau,\pm}(\cdot) \|_{L^2(0,T-l\tau;L^2)}^2 +
\| \theta_h^{\tau, \pm}(\cdot  + l\tau) - \theta_h^{\tau,\pm}(\cdot) \|_{L^2(0,T-l\tau;L^2)}^2 \leq Cl \tau,
\end{align}
for any $l \in \{1, \dots, N_\tau\}$.

\begin{prop}[Compactness]\label{prop:comp}
There exists a non-relabelled subsequence $(h, \tau) \to (0,0)$ and functions $\vp$, $\theta$, $\bu$ and $q$ satisfying
\begin{align*}
& \vp, \theta  \in L^\infty(0,T;H^1(\Omega)) \cap L^2(0,T;H^2_n(\Omega)) \cap H^1(0,T;L^2(\Omega)), \\
& \bu  \in L^\infty(0,T;X(\Omega)) , \quad q \in L^\infty(0,T),
\end{align*}
such that for $f \in \{ \vp, \theta\}$ and any $r \in [2,\infty), s \in [0,1)$ if $d = 2$ and $r \in [2,6], s \in [0,\frac{1}{2})$ if $d = 3$,
\begin{subequations}
\begin{alignat}{2}
f_h^{\tau,\pm} \to f & \text{ weakly* in } L^\infty(0,T;H^1(\Omega)), \label{comp:weak} \\
f_h^{\tau,\pm} \to f & \text{ weakly in } L^2(0,T;W^{1,r}(\Omega)), \label{comp:weak:W} \\
\pd_t f_h^\tau \to \pd_t f & \text{ weakly in } L^2(0,T;L^2(\Omega)), \label{comp:time} \\
\bu_h^{\tau,\pm} \to \bu & \text{ weakly in } L^2(0,T;H^1(\Omega)), \label{comp:u} \\
q^{\tau,\pm}_h \to q & \text{ weakly* in } L^\infty(0,T),  \label{comp:q} \\
\Delta_h f_h^{\tau,\pm} \to \Delta f & \text{ weakly in } L^2(0,T;L^2(\Omega)), \label{comp:H2} \\
f_h^{\tau,\pm} \to f & \text{ strongly in } L^2(0,T;C^{0,s}(\overline{\Omega})) \text{ and a.e.~in } \Omega_T, \label{comp:str} \\
Q(\vp_h^{\tau,\pm}) \to Q(\vp) & \text{ strongly in } L^2(0,T), \label{comp:Q}
\end{alignat}
\end{subequations}
with $\vp(0) = \vp_0$ and $\theta(0) = \theta_0$.
\end{prop}

\begin{proof}
We first note from \eqref{unif:bdd} that 
\begin{align}\label{tau:pm:diff}
\| \vp_h^{\tau} - \vp_h^{\pm} \|_{L^2(0,T;H^1)} + \| \theta_h^{\tau} - \theta_h^{\pm} \|_{L^2(0,T;H^1)} + \| q^\tau_h - q^{\pm}_h \|_{L^2(0,T)} \leq C \sqrt{\tau},
\end{align}
which shows that the limits of $\vp^{\tau,\pm}$ (likewise for $\theta_h^{\tau,\pm}$ and $q^{\tau,\pm}_h$) coincide as $(h,\tau) \to (0,0)$. Then, \eqref{comp:weak}, \eqref{comp:weak:W}, \eqref{comp:time}, \eqref{comp:u} and \eqref{comp:q} are consequences of the uniform estimate \eqref{unif:bdd} and standard compactness results in Bochner spaces. For \eqref{comp:H2} we argue similarly as in \cite[Lem.~3.1]{Barrett}, where for $\eta \in L^2(0,T;H^2(\Omega))$, it holds that 
\begin{align*}
 & \int_0^T ( \Delta_h \vp_h^{\tau,\pm}, \eta) \dt  = \int_0^T (\Delta_h \vp_h^{\tau,\pm}, (1 - \Ih)(\eta)) \dt + \int_0^T (\Delta_h \vp_h^{\tau,\pm}, \Ih(\eta)) \dt \\
& \quad = \int_0^T (\Delta_h \vp_h^{\tau,\pm}, (1 - \Ih)(\eta)) \dt + \int_0^T (\nabla \vp_h^{\tau,\pm}, \nabla [(1 -  \Ih)(\eta)]) \dt - \int_0^T (\nabla \vp_h^{\tau,\pm}, \nabla \eta) \dt \\
& \quad =: A_1 + A_2 - \int_0^T (\nabla \vp_h^{\tau,\pm}, \nabla \eta) \dt.
\end{align*}
By \eqref{Interpol} and the uniform bound \eqref{unif:bdd} on $\Delta_h \vp_h^{\tau, \pm}$, we see that 
\begin{align*}
| A_1 | & \leq C \| (1 - \Ih)(\eta) \|_{L^2(0,T;L^2)} \leq C h^2 \| \eta \|_{L^2(0,T;H^2)}, \\
| A_2 | & \leq C \| \nabla [(1 - \Ih)(\eta)] \|_{L^2(0,T;L^2)} \leq Ch \| \eta \|_{L^2(0,T;H^2)}.
\end{align*}
Together with \eqref{comp:weak} we deduce that 
\[
\int_0^T (\Delta_h \vp_h^{\tau,\pm}, \eta) \dt \to - \int_0^T (\nabla \vp, \nabla \eta) \dt 
\]
along a non-relabelled subsequence $(h,\tau) \to (0,0)$. Using the denseness of $L^2(0,T;H^2(\Omega))$ in $L^2(0,T;H^1(\Omega))$ we infer that $\Delta \vp \in L^2(0,T;L^2(\Omega))$, and by elliptic regularity on bounded convex domains \cite[Thm.~2.4.2.7]{Grisvard} we obtain that $\vp \in L^2(0,T;H^2_n(\Omega))$. An analogous argument shows $\theta \in L^2(0,T;H^2_n(\Omega))$.

With the compact embedding $W^{1,r}(\Omega) \subset \subset C^{0,s}(\overline{\Omega})$ for $s \in [0,1)$ if $d = 2$ and $s \in [0,\frac{1}{2})$ if $d = 3$, substituting $X = W^{1,r}(\Omega)$, $B = C^{0,s}(\overline{\Omega})$ and $Y = L^2(\Omega)$ in \eqref{Simon:1}, we obtain the strong convergence of $\vp_h^{\tau}$ (resp.~$\theta_h^{\tau}$) to $\vp$ (resp.~$\theta$) in $L^2(0,T;C^{0,s}(\overline{\Omega}))$. On the other hand, the estimate \eqref{time:trans} fulfils the requirement \eqref{Simon:2} for $\vp_h^{\pm}$ (resp.~$\theta_h^{\pm}$). This establishes the strong convergence for $\vp_h^{\pm}$ (resp.~$\theta_h^{\pm}$) to $\vp$ (resp.~$\theta$) in $L^2(0,T;C^{0,s}(\overline{\Omega}))$, whence along a further non-relabelled subsequence we also have the a.e.~convergence in $Q$, i.e., \eqref{comp:str}. 

Lastly, by \eqref{ass:W} there exists a positive constant $C$ such that 
\begin{align}
|W(s_1) - W(s_2)| & \leq C \Big ( 1 + |s_1|^3 + |s_2|^3 \Big ) |s_1 - s_2| \quad \forall s_1, s_2 \in \R, \label{W:diff} \\
|W'(s_1) - W'(s_2)| & \leq C \Big ( 1 + |s_1|^2 + |s_2|^2 \Big ) |s_1 - s_2| \quad \forall s_1, s_2 \in \R. \label{W':diff}
\end{align}
Then, by the definition \eqref{defn:Q} of $Q$, it is easy to see that 
\begin{equation}\label{Q:diff}
\begin{aligned}
|Q(\vp_h^{\tau, \pm}) - Q(\vp)| & \leq \frac{\| W(\vp_h^{\tau,\pm}) - W(\vp) \|_{L^1}}{Q(\vp_h^{\tau,\pm}) + Q(\vp)} \\
& \leq C \Big ( 1 + \| \vp_h^{\tau,\pm} \|_{H^1}^3 + \| \vp \|_{H^1}^3 \Big ) \| \vp_h^{\tau,\pm} - \vp \| \leq C \| \vp_h^{\tau,\pm} - \vp \|.
\end{aligned}
\end{equation}
Invoking the strong convergence of $\vp_h^{\tau,\pm}$ to $\vp$ in $L^2(0,T;L^2(\Omega))$, we deduce the strong convergence of $Q(\vp_h^{\tau,\pm})$ to $Q(\vp)$ in $L^2(0,T)$, i.e., \eqref{comp:Q}. 

For the attainment of initial conditions, let $\zeta \in C^1([0,T];H^1(\Omega))$ with $\zeta(T) = 0$ be arbitrary. Then, by integrating by parts in time,
\[
(\vp_h^0, \zeta(0)) = - \int_0^T (\pd_t \vp_h^\tau, \zeta) \dt - \int_0^T (\vp_h^\tau, \pd_t \zeta) \dt.
\]
Passing to the limit $(h,\tau) \to (0,0)$ we deduce with the help of \eqref{comp:weak} and \eqref{comp:time} that 
\[
(\vp_0, \zeta(0)) = - \int_0^T (\pd_t \vp,  \zeta) \dt - \int_0^T (\vp, \pd_t \zeta) \dt = (\vp(0), \zeta(0)),
\]
where for the right-most equality we applied integration by parts in time. This leads to the identification $\vp(0) = \vp_0$. The identification $\theta(0) = \theta_0$ can be achieved analogously.
\end{proof}

\begin{thm}[Convergence]\label{thm:conv}
Under assumptions \eqref{ass:dom}-\eqref{ass:Ec}, the functions $\vp$, $\theta$, $\bu$ and $q$ obtained from Proposition \ref{prop:comp} is a weak solution to \eqref{model} in the sense of \eqref{model:weak} holding for arbitrary $\psi \in L^2(0,T; H^1(\Omega))$ and $\bv \in L^2(0,T;X(\Omega))$.  Moreover, we have the identification
\[
q(t) = Q(\vp(t)) = \Big ( \int_\Omega W(\vp(t)) \dx + 1 \Big )^{1/2}
\]
holding for a.e.~$t \in (0,T)$ with $q(0) = Q(\vp_0)$.
\end{thm}

\begin{proof}
Let $\zeta \in C^0([0,T];C^\infty(\overline{\Omega}))$ be an arbitrary test function.  We choose $\psi_h = \Ih(\zeta)$ in \eqref{cts:SAV:1} and \eqref{cts:SAV:2}, keeping in mind that $\Ih(\zeta) \to \zeta$ in $L^2(0,T;H^1(\Omega))$. Then, for the time derivative term in \eqref{cts:SAV:1}, we have
\begin{align*}
\Big | \int_0^T (\pd_t \vp_h^\tau, \Ih(\zeta)) - (\pd_t \vp, \zeta) \dt \Big | & \leq  \Big | \int_0^T (\pd_t \vp_h^\tau,   \Ih(\zeta) - \zeta ) \dt \Big | + \Big | \int_0^T (\pd_t \vp_h^\tau - \pd_t \vp, \zeta) \dt \Big |  \\
& \to 0 \text{ as } (h,\tau) \to (0,0)
\end{align*}
due to the uniform estimate \eqref{unif:bdd}, the weak convergence of $\pd_t \vp_h^\tau$ in $L^2(0,T;L^2(\Omega))$, and \eqref{Interpol}.  For the coupled term in \eqref{cts:SAV:1}, we have similarly
\begin{align*}
& \Big |  \int_0^T ((\theta_h^+ - \theta_c) p(\vp_h^-), \Ih(\zeta)) - ((\theta - \theta_c)p(\vp), \zeta) \dt \Big | \\
& \quad \leq  \Big | \int_0^T ((\theta_h^+ - \theta_c) p(\vp_h^-), \Ih(\zeta) - \zeta) \dt \Big |  +  \Big | \int_0^T ((\theta_h^+ - \theta_c)p(\vp_h^-) - ((\theta - \theta_c)p(\vp)), \zeta) \dt \Big |  \\
& \quad =: B_1 + B_2 \to 0 \text{ as } (h,\tau) \to (0,0).
\end{align*}
Indeed, invoking \eqref{Interpol}, the boundedness of $p$ and \eqref{unif:bdd} we see that 
\begin{align*}
B_1 \leq C \| \theta_h^+ - \theta_c \|_{L^2(0,T;L^2)} \| \Ih(\zeta) - \zeta \|_{L^2(Q)} \leq Ch^2 \| \zeta \|_{L^2(0,T;H^2)},
\end{align*}
and so $B_1 \to 0$ as $(h, \tau) \to (0,0)$. On the other hand, using the boundedness and continuity of $p$, the a.e.~convergence of $\vp_h^-$ to $\vp$, invoking the dominated convergence theorem, we infer that $p(\vp_h^-) \to p(\vp)$ strongly in $L^2(0,T;L^2(\Omega))$. Hence, together with the weak convergence of $\theta_h^+$ in $L^2(0,T;L^2(\Omega))$, we deduce that $B_2 \to 0$ as $(h,\tau) \to (0,0)$.  Then, via a similar argument, for \eqref{cts:SAV:2} we infer that as $(h,\tau) \to (0,0)$,
\begin{align*}
 \int_0^T ( \pd_t \theta_h^\tau -  p(\vp_h^-) \pd_t \vp_h^\tau, \Ih(\zeta)) \dt \to \int_0^T ( \pd_t \theta -  p(\vp) \pd_t \vp, \zeta) \dt.
\end{align*}
Meanwhile for the gradient term in \eqref{cts:SAV:1}, it holds that 
\begin{align*}
& \Big | \int_0^T (\nabla \vp_h^+, \nabla \Ih(\zeta)) - (\nabla \vp, \nabla \zeta) \dt \Big | \\
& \quad \leq  \Big | \int_0^T (\nabla \vp_h^+, \nabla \Ih(\zeta) - \zeta) \dt \Big | + \Big | \int_0^T (\nabla \vp_h^+ - \nabla \vp, \nabla \zeta) \dt \Big | \\
& \quad \leq Ch^2\| \nabla \vp_h^+ \|_{L^2(0,T;L^2)} \| \zeta \|_{L^2(0,T;H^2)} + \Big | \int_0^T (\nabla \vp_h^+ - \nabla \vp, \nabla \zeta) \dt \Big | \to 0,
\end{align*}
as $(h,\tau) \to (0,0)$. The gradient term in \eqref{cts:SAV:2} can be treated analogously.

Turning now to \eqref{cts:SAV:3}, where for $\bm{\zeta} \in C^0([0,T];C^1(\overline{\Omega},\R^d))$ we choose $\bv_h = \Ih(\bm{\zeta})$. Then, using the boundedness and continuity of $\CC(\cdot)$, the a.e.~convergence of $\vp_h^+$ to $\vp$ and the dominated convergence theorem we have 
\begin{align}\label{DCT:s}
\| \CC(\vp_h^+) - \CC(\vp) \|_{L^s(0,T;L^s)} \to 0 \text{ as } (h,\tau) \to (0,0) \text{ for any } s < \infty.
\end{align}
By the norm equivalence \eqref{norm:equiv}, it holds that  
\begin{equation}\label{lim:C}
\begin{aligned}
& \| \Ih(\CC(\vp_h^+)) - \CC(\vp) \|_{L^2(0,T;L^2)} \\
& \quad \leq \| \Ih (\CC(\vp_h^+) - \CC(\vp))  \|_{L^2(0,T;L^2)} + \| \Ih(\CC(\vp)) - \CC(\vp) \|_{L^2(0,T;L^2)} \\
&  \quad \leq C \| \CC(\vp_h^+) - \CC(\vp) \|_{L^2(0,T;L^2)} + C\| \Ih(\CC(\vp)) - \CC(\vp) \|_{L^2(0,T;L^\infty)} \\
& \quad  \to 0 \text{ as } (h,\tau) \to (0,0),
\end{aligned}
\end{equation}
where on the right-hand side for the first term we used \eqref{DCT:s}, while for the second term we used the fact that $\vp \in L^2(0,T;C^0(\overline{\Omega}))$ and the property \eqref{Interpol:Linfty}. Hence, along a further non-relabelled subsequence,
\[
\Ih(\CC(\vp_h^+)) \to \CC(\vp) \text{ a.e.~in } \Omega_T \text{ as } (h,\tau) \to (0,0).
\]
Using the boundedness of $\CC(\cdot)$ once again and the fact that $\Ih(\bm{\zeta}) \to \bm{\zeta}$ strongly in $L^2(0,T;H^1(\Omega, \R^d))$, together with 
\begin{align*}
\Ih(\CC(\vp_h^+)) \E(\Ih(\bm{\zeta})) \to \CC(\vp) \E(\bm{\zeta}) & \text{ a.e.~in } \Omega_T \text{ as } (h,\tau) \to (0,0), \\
|\Ih(\CC(\vp_h^+)) \E(\Ih(\bm{\zeta}))| \leq C |\E(\Ih(\bm{\zeta}))| & \text{ a.e.~in } \Omega_T \text{ for all } h, \tau > 0, \\
C |\E(\Ih(\bm{\zeta}))| \to C |\E(\bm{\zeta})| & \text{ strongly in } L^2(0,T;L^2(\Omega)) \text{ as } (h,\tau) \to (0,0),
\end{align*}
by the generalised Lebesgue dominated convergence theorem we deduce that 
\[
\Ih(\CC(\vp_h^+)) \E(\Ih(\bm{\zeta})) \to \CC(\vp_h) \E(\bm{\zeta}) \text{ strongly in } L^2(0,T;L^2(\Omega)) \text{ as } (h,\tau) \to (0,0).
\]
Together with the weak convergence $\E(\bu_h^+)$ to $\E(\bu)$ in $L^2(0,T;L^2(\Omega))$ we have
\[
\int_{\Omega_T} \Big ( \Ih[\CC(\vp_h^+) ] \E(\bu_h^+), \E(\Ih(\bm{\zeta})) \Big ) \dxt \to \int_{\Omega_T} \Big ( \CC(\vp) \E(\bu), \E(\bm{\zeta}) \Big ) \dxt
\]
as $(h,\tau) \to (0,0)$. Similarly, by the boundedness and continuity of $m(\cdot)$, upon replacing $\CC(\vp_h^+)$ with $\CC(\vp_h^+) m(\vp_h^+)$ in \eqref{lim:C}, we immediately infer that $\Ih(\CC(\vp_h^+) m(\vp_h^+)) \to \CC(\vp) m(\vp)$ strongly in $L^2(0,T;L^2(\Omega))$, and
\[
\int_{\Omega_T} \Big ( \Ih[\CC(\vp_h^+)m(\vp_h^+) \mathbb{I} ], \E(\Ih(\bm{\zeta})) \Big ) \dxt \to \int_{\Omega_T} \Big ( \CC(\vp) m(\vp) \mathbb{I}, \E(\bm{\zeta}) \Big ) \dxt
\]
as $(h,\tau) \to (0,0)$.  On the other hand, employing \eqref{DCT:s}, the uniform estimate \eqref{unif:bdd} for $\theta_h^+$, the strong convergence of $\theta_h^+$ to $\theta$ in $L^2(\Omega_T)$, and \eqref{Ritz1} for $\theta_h^0 = R_h(\theta_0)$, we deduce that 
\begin{align*}
& \| \CC(\vp_h^+) (\theta_h^+ - \theta_h^0) - \CC(\vp)  (\theta - \theta_0) \|_{L^2(0,T;L^2)} \\
& \quad \leq C \| \theta_h^+ - \theta_h^0 \|_{L^4(0,T;L^4)} \| \CC(\vp_h^+) - \CC(\vp) \|_{L^4(0,T;L^4)} + C \| (\theta_h^+ - \theta_h^0) - (\theta - \theta_0) \|_{L^2(0,T;L^2)} \\
& \quad \to 0 \text{ as } (h,\tau) \to (0,0).
\end{align*}
Then, analogous to \eqref{lim:C} we have for $Y_h^+ := \CC(\vp_h^+) (\theta_h^+ - \theta_h^0)$ and $Y := \CC(\vp)(\theta - \theta_0)$,
\begin{align*}
& \| \Ih(\CC(\vp_h^+)(\theta_h^+ - \theta_h^0)) - \CC(\vp)(\theta - \theta_0) \|_{L^2(0,T;L^2)} \\
& \quad \leq \| \Ih(Y_h^+ - Y) \|_{L^2(0,T;L^2)} + \| \Ih(Y) - Y\|_{L^2(0,T;L^2)} \\
& \quad \leq C \| Y_h^+ - Y \|_{L^2(0,T;L^2)} + C \| \Ih(Y) - Y \|_{L^2(0,T;L^\infty)} \\
& \quad \to 0 \text{ as } (h,\tau) \to (0,0),
\end{align*}
on account of the fact that $\vp, \theta \in L^2(0,T;C^0(\overline{\Omega}))$ and $\theta_0 \in H^2_n(\Omega) \subset C^0(\overline{\Omega})$. This yields 
\[
\int_{\Omega_T} \Big ( \Ih[\CC(\vp_h^+) \beta (\theta_h^+ - \theta_h^0) \mathbb{I} ], \E(\Ih(\bm{\zeta})) \Big ) \dxt \to \int_{\Omega_T} \Big ( \CC(\vp) \beta(\theta - \theta_0) \mathbb{I}, \E(\bm{\zeta}) \Big ) \dxt
\]
as $(h,\tau) \to (0,0)$. It remains to consider the term involving the scalar auxiliary variable in \eqref{cts:SAV:1}, where we have
\begin{align*}
& \Big | \int_0^T \frac{q^+_h}{Q(\vp_h^-)} (W'(\vp_h^-), \Ih(\zeta)) - \frac{q}{Q(\vp)} (W'(\vp), \zeta) \dt \Big | \\
& \quad \leq \Big | \int_0^T \frac{q^+_h}{Q(\vp_h^-)} (W'(\vp_h^-), \Ih(\zeta) - \zeta) \dt \Big | + \Big | \int_0^T q^+_h \Big ( \frac{1}{Q(\vp_h^-)} - \frac{1}{Q(\vp)} \Big ) (W'(\vp_h^-), \zeta) \dt \Big | \\
& \qquad + \Big | \int_0^T \frac{q^+_h}{Q(\vp)} (W'(\vp_h^-) - W'(\vp), \zeta) \dt \Big | + \Big | \int_0^T (q^+_h - q) \frac{(W'(\vp), \zeta)}{Q(\vp)} \dt \Big | \\
& \quad =: D_1 + D_2 + D_3 + D_4.
\end{align*}
Using the uniform estimate \eqref{unif:bdd}, lower bound on $Q$, interpolation property \eqref{Interpol}, and the fact that $W'(\vp_h^-)$ is bounded uniformly in $L^\infty(0,T;L^2(\Omega))$ deduced from the growth assumption \eqref{ass:W} on $W'$, we infer that
\begin{align*}
D_1  \leq C \| W'(\vp_h^-) \|_{L^\infty(0,T;L^2)} \| \Ih(\zeta) - \zeta \|_{L^2(0,T;L^2)} \leq Ch^2 \to 0,
\end{align*}
as $(h,\tau) \to (0,0)$. On the other hand, using \eqref{W:diff} and \eqref{W':diff}, and the uniform estimate \eqref{unif:bdd} we infer that
\begin{align}
\| W(\vp_h^-) - W(\vp) \|_{L^1} & \leq C \Big (1 + \| \vp_h^- \|_{H^1}^3 + \| \vp \|_{H^1}^3 \Big ) \| \vp_h^- - \vp \| \leq C \| \vp_h^- - \vp \|, \label{W:diff:L1} \\
\| W'(\vp_h^-) - W'(\vp) \|_{L^\frac{6}{5}} & \leq C \Big ( 1 + \| \vp_h^- \|_{L^6}^2 + \| \vp \|_{L^6}^2 \Big ) \| \vp_h^- - \vp \| \leq C \| \vp_h^- - \vp \|, \label{W':diff:L65}
\end{align}
and so we obtain the estimates
\begin{align*}
D_2 & \leq C \| W'(\vp_h^-) \|_{L^\infty(0,T;L^2)} \| \zeta \|_{L^2(0,T;L^2)} \| W(\vp_h^-) - W(\vp) \|_{L^2(0,T;L^1)} \\
& \leq C \| \vp_h^- - \vp \|_{L^2(0,T;L^2)}, \\
D_3 & \leq C \| W'(\vp_h^-) - W'(\vp) \|_{L^2(0,T;L^\frac{6}{5})} \| \zeta \|_{L^2(0,T;H^1)}  \\
& \leq C \| \vp_h^- - \vp \|_{L^2(0,T;L^2)},
\end{align*}
which in turn implies $D_2, D_3 \to 0$ as $(h,\tau) \to (0,0)$ thanks to the strong convergence \eqref{comp:str}.  Lastly, $D_4 \to 0$ as $(h,\tau) \to (0,0)$ due to the weak* convergence \eqref{comp:q} and the fact that $\frac{(W'(\vp), \zeta)}{Q(\vp)} \in L^1(0,T)$.  Hence, we have
\[
 \int_0^T \frac{q^+_h}{Q(\vp_h^-)} (W'(\vp_h^-), \Ih(\zeta)) \dt \to \int_0^T \frac{q}{Q(\vp)} (W'(\vp), \zeta) \dt 
\]
as $(h,\tau) \to (0,0)$.

Passing to the limit $(h,\tau) \to (0,0)$ in \eqref{cts:time} and using a standard density argument shows that the limit functions $(\vp, \theta, q, \bu)$ satisfy \eqref{weak:SAV:1}-\eqref{weak:SAV:3} for a.e.~$t \in (0,T)$ and for arbitrary test functions $\psi \in H^1(\Omega)$ and $\bm{v} \in X(\Omega)$. The identification $q(t) = Q(\vp(t))$ for a.e.~$t \in (0,T)$ can be adapted analogously from \cite{LamWangSAV,MetzgerSAV}, so that \eqref{weak:SAV:1} is equivalent to \eqref{original:weak:1} and the limit functions $(\vp, \theta, \bu)$ is a weak solution to \eqref{model} in the sense of \eqref{model:weak}.
\end{proof}

\section{Well-posedness and regularity of solutions}\label{sec:wellposed}
Theorem \ref{thm:conv} provides the existence of weak solutions for the model \eqref{model}. In this section we complete the well-posedness of \eqref{model} by establishing continuous dependence on initial data and derive higher regularity for the solutions.

\begin{thm}[Continuous dependence on initial data]\label{thm:cts}
Let $\{(\vp_i, \theta_i, \bu_i)\}_{i=1,2}$ denote weak solutions to \eqref{model} in the sense of \eqref{model:weak} corresponding to initial conditions $\{(\vp_{0,i}, \theta_{0,i})\}_{i=1,2}$, respectively. Then, there exists a positive constant $C$ independent of the differences $\hat{\vp}:= \vp_1 -\vp_2$, $\hat{\theta} := \theta_1 -\theta_2$ and $\hat{\bu} = \bu_1 - \bu_2$, such that 
\begin{align*}
\sup_{t \in (0,T]} \Big ( \| \hat{\vp} \|_{H^1}^2 + \| \hat{\theta} \|^2 \Big ) + \int_0^T \| \hat{\vp} \|_{H^2}^2 + \| \hat{\theta} \|_{H^1}^2 + \| \pd_t \hat{\vp} \|^2 + \| \hat{\bu} \|_{H^1}^2 \dt \leq C \Big (\| \hat{\vp}_0 \|_{H^1}^2 + \| \hat{\theta}_0 \|^2 \Big ).
\end{align*}
Consequently, the solution to \eqref{model} is unique.
\end{thm}

\begin{proof}
We first consider the equation for $\hat{\vp}$, which reads in strong form as
\[
 \pd_t \hat{\vp} -   \Delta \hat{\vp} +  p(\vp_1) \hat{\theta} = - (W'(\vp_1) - W'(\vp_2)) -  (\theta_2 - \theta_c)(p(\vp_1) - p(\vp_2)).
\]
Testing this with $\hat{\vp}$, $\pd_t \hat{\vp}$ and $-\Delta \hat{\vp}$, then using the boundedness and Lipschitz continuity of $p$, as well as the inequality \eqref{W':diff} for $W'$ yields
\begin{align}
& \frac{1}{2} \frac{d}{dt} \| \hat{\vp} \|^2 +  \| \nabla \hat{\vp} \|^2 \leq C \Big ( 1 + \| \theta_2 \|_{L^\infty} + \| \hat{\theta} \| + \| \vp_i \|_{L^\infty}^2 \| \hat{\vp} \| \Big ) \| \hat{\vp} \|, \label{cts:1} \\
&  \| \pd_t \hat{\vp} \|^2 + \frac{1}{2} \frac{d}{dt} \| \nabla \hat{\vp} \|^2 + (p(\vp_1) \hat{\theta}, \pd_t \hat{\vp}) \leq C \Big ( 1 + \| \theta_2 \|_{L^\infty} + \| \vp_i \|_{L^\infty}^2 \Big ) \| \hat{\vp} \| \| \pd_t \hat{\vp} \|, \label{cts:2} \\
& \frac{1}{2} \frac{d}{dt} \| \nabla \hat{\vp} \|^2 +  \| \Delta \hat{\vp} \|^2 \leq C \Big ( 1 + \| \theta_2 \|_{L^\infty} + \| \vp_i \|_{L^\infty}^2 \Big ) \| \hat{\vp} \|  \| \Delta \hat{\vp} \| + C \| \hat{\theta} \|  \| \Delta \hat{\vp} \|. \label{cts:3}
\end{align}
Next, testing the equation for $\hat{\theta}$, which reads in strong form
\[
 \pd_t \hat{\theta}  -  p(\vp_1) \pd_t \hat{\vp} - (p(\vp_1) - p(\vp_2)) \pd_t \vp_2 - \Delta \hat{\theta}= 0
\]
with $\hat{\theta}$ yields 
\begin{align}
& \frac{1}{2} \frac{d}{dt} \| \hat{\theta} \|^2 - ( p(\vp_1) \hat{\theta}, \pd_t \hat{\vp}) + \| \nabla \hat{\theta} \|^2 \leq C \| \pd_t \vp_2 \| \| \hat{\vp} \|_{L^6} \| \hat{\theta} \|_{L^3}. \label{cts:4}
\end{align}
Invoking the following Gagliardo--Nirenburg inequalities in three dimensions:
\begin{align*}
\| f \|_{L^\infty} \leq C \| f \|_{H^2}^{\frac{1}{2}} \| f \|_{L^6}^{\frac{1}{2}}, \quad \| f \|_{L^3} \leq C \| \nabla f \|_{L^2}^{\frac{1}{2}} \| f \|_{L^2}^{\frac{1}{2}} + C \| f \|_{L^2},
\end{align*}
then upon adding \eqref{cts:2} and \eqref{cts:4}, noting a cancellation, we arrive at
\begin{equation}\label{cts:7}
\begin{alignedat}{2}
& \frac{1}{2} \frac{d}{dt} \Big ( \| \nabla \hat{\vp} \|^2 +  \| \hat{\theta} \|^2 \Big ) + \| \nabla \hat{\theta} \|^2 +  \| \pd_t \hat{\vp} \|^2 \\
& \quad \leq C \Big ( 1 + \| \theta_2 \|_{H^2}^{\frac{1}{2}} + \| \vp_i \|_{H^2} \Big ) \| \hat{\vp} \| \| \pd_t \hat{\vp} \| + C \| \pd_t \vp_2 \| \| \hat{\vp} \|_{H^1} \Big (\| \hat{\theta} \|^{\frac{1}{2}} \| \nabla \hat{\theta} \|^{\frac{1}{2}} + \| \hat{\theta} \| \Big) \\
& \quad \leq \frac{1}{2} \| \pd_t \hat{\vp} \|^2 + \frac{1}{2} \| \nabla \hat{\theta} \|^2 + C \Big (1 + \| \theta_2 \|_{H^2} + \| \pd_t \vp_2 \|^2 + \| \vp_i \|_{H^2}^2 \Big ) \Big (\| \hat{\vp} \|_{H^1}^2 + \| \hat{\theta} \|^2 \Big ).
\end{alignedat}
\end{equation}
To \eqref{cts:7} we add \eqref{cts:1}, and after adjusting the constant prefactors we find that
\begin{align*}
& \frac{d}{dt} \Big ( \| \hat{\vp} \|_{H^1}^2 + \| \hat{\theta} \|^2 \Big ) + \| \nabla \hat{\theta} \|^2 + \| \pd_t \hat{\vp} \|^2 \\
& \quad \leq C \Big (1 + \| \theta_2 \|_{H^2} + \| \pd_t \vp_2 \|^2 + \| \vp_i \|_{H^2}^2 \Big ) \Big (\| \hat{\vp} \|_{H^1}^2 + \| \hat{\theta} \|^2 \Big ),
\end{align*}
where by invoking the regularity of solutions listed in Proposition \ref{prop:comp} and the application of Gronwall's inequality leads to the claim
\begin{align*}
\sup_{t \in (0,T]} \Big ( \| \hat{\vp}(t) \|_{H^1}^2 + \| \hat{\theta}(t) \|^2 \Big ) + \int_0^T \| \nabla \hat{\theta} \|^2 + \| \pd_t \hat{\vp} \|^2 \dt \leq C \Big (\| \hat{\vp}_0 \|_{H^1}^2 + \| \hat{\theta}_0 \|^2 \Big ).
\end{align*}
This furnishes the uniqueness of $\vp$ and $\theta$. Returning to \eqref{cts:3} we see that 
\begin{align*}
\frac{1}{2} \frac{d}{dt} \| \nabla \hat{\vp} \|^2 + \frac{1}{2} \| \Delta \hat{\vp} \|^2 \leq C \Big ( 1 + \| \theta_2 \|_{H^2} + \| \vp_i \|_{H^2}^2 \Big ) \Big ( \| \hat{\vp} \|^2 + \| \hat{\theta} \|^2 \Big ),
\end{align*}
whence by Gronwall's inequality, as well as elliptic regularity we deduce that 
\begin{align}\label{cts:H2}
\int_0^T \| \hat{\vp} \|_{H^2}^2 \dt \leq C \Big (\| \hat{\vp}_0 \|_{H^1}^2 + \| \hat{\theta}_0 \|^2 \Big ).
\end{align}
Turning now to the equation for $\hat{\bu}$, which reads as 
\begin{align*}
0 & = \int_{\Omega} \CC(\vp_1)(\E(\hat{\bu}) - (m(\vp_1) - m(\vp_2))\mathbb{I} +  \hat{\theta} \mathbb{I}) : \E(\bv) \dx \\
\notag & \quad + \int_\Omega (\CC(\vp_1) - \CC(\vp_2))(\E(\bu_2) - m(\vp_2) \mathbb{I} +  (\theta_2 - \theta_0) \mathbb{I}): \E(\bv) \dx
\end{align*}
for arbitrary $\bv \in L^2(0,T;X(\Omega))$. Choosing $\bv = \hat{\bu}$, then invoking the regularities stated in Proposition \ref{prop:comp} and applying lower bounds in \eqref{ass:CC}, as well as the Lipschitz continuity of $\CC$ and $m$ leads to 
\begin{align*}
\| \E(\hat{\bu}) \|^2 & \leq C \big ( \| \hat{\vp} \|^2 + \|\hat{\theta} \|^2 \big ) +  C \big ( 1+  \| \theta_2 \|^2 + \| \E(\bu_2) \|^2 \big ) \| \hat{\vp} \|_{L^\infty}^2 \\
& \leq C \big ( \| \hat{\vp} \|^2 + \|\hat{\theta} \|^2 + \| \hat{\vp} \|_{H^2} \| \hat{\vp} \|_{H^1} \big ) \leq C \big ( \| \hat{\vp} \|_{H^2}^2 + \| \hat{\theta} \|^2 \big ).
\end{align*}
By Korn's inequality and \eqref{cts:H2} we deduce that 
\[
\int_0^T \| \hat{\bu} \|_{H^1}^2 \dt \leq C \Big (\| \hat{\vp}_0 \|_{H^1}^2 + \| \hat{\theta}_0 \|^2 \Big ).
\]
This also provides uniqueness for $\bu$.
\end{proof}

\begin{thm}[Higher regularity]\label{thm:reg}
Suppose in addition to \eqref{ass:dom}-\eqref{ass:Ec}, it holds that 
\begin{enumerate}[label=$(\mathrm{A 8})$, ref = $\mathrm{A 8}$]
\item \label{ass:Reg} the initial conditions satisfy $\vp_0, \theta_0 \in H^3(\Omega)$.
\end{enumerate}
Then the weak solution $(\vp, \theta, \bu)$ to \eqref{model} satisfies the further regularities
\begin{align*}
\vp, \theta & \in L^\infty(0,T;H^2_n(\Omega) \cap W^{2,6}(\Omega)), \\
\pd_t \vp, \pd_t \theta & \in  L^2(0,T;H^2(\Omega)) \cap  L^{\infty}(0,T;H^1(\Omega)), \\
\pd_{tt} \vp, \pd_{tt} \theta & \in L^2(0,T;L^2(\Omega)), \\
\bu & \in W^{1,4}(0,T;X(\Omega)).
\end{align*}
\end{thm}

\begin{proof}
We present formal calculations that can be made rigorous with a Faedo--Galerkin approximation.  

\paragraph{First estimate.} Taking the time derivative of \eqref{model:1}:
\begin{align}\label{vp:tt}
 \pd_{tt} \vp =  \Delta \pd_t \vp -  W''(\vp) \pd_t \vp -  p(\vp) \pd_t \theta - (\theta - \theta_c) p'(\vp) \pd_t \vp,
\end{align}
and testing with $\pd_t \vp$ yields
\begin{equation}\label{reg:1}
\begin{aligned}
\frac{1}{2} \frac{d}{dt} \| \pd_t \vp \|^2 +  \| \nabla \pd_t \vp \|^2 - c_3 \| \pd_t \vp \|^2  \leq - ( p(\vp) \pd_t \vp, \pd_t \theta) - ((\theta - \theta_c) p'(\vp), |\pd_t \vp|^2),
\end{aligned}
\end{equation}
where we have used the lower bound for $W''$ in \eqref{ass:W}. Then, testing \eqref{model:2} with $\pd_t \theta$ yields
\begin{align*}
 \| \pd_t \theta \|^2 + \frac{1}{2} \frac{d}{dt} \| \nabla \theta \|^2 = ( p(\vp) \pd_t \vp, \pd_t \theta),
\end{align*}
and when added to \eqref{reg:1} we note a cancellation and obtain
\begin{equation}
\begin{aligned}
& \frac{1}{2} \frac{d}{dt} \Big (  \| \pd_t \vp \|^2 + \| \nabla \theta \|^2 \Big ) + \| \nabla \pd_t \vp \|^2 +  \| \pd_t \theta \|^2 \\
& \quad \leq C \| \pd_t \vp \|^2 - ((\theta - \theta_c) p'(\vp), |\pd_t \vp|^2) \leq C \| \pd_t \vp \|^2 + C \big ( 1 + \| \theta \| \big ) \| \pd_t \vp \|_{L^4}^2 \\
& \quad \leq C \| \pd_t \vp \|^2 + C \| \nabla \pd_t \vp \|^{\frac{3}{2}} \| \pd_t \vp \|^{\frac{1}{2}} \leq \frac{1}{2} \| \nabla \pd_t \vp \|^2 + C \| \pd_t \vp \|^2,
\end{aligned}
\end{equation}
where we have employed the boundedness of $\theta$ in $L^\infty(0,T;L^2(\Omega))$ and the Gagliardo--Nirenburg inequality for three dimensions:
\[
\| f \|_{L^4} \leq C \| \nabla f \|^{\frac{3}{4}} \| f \|^{\frac{1}{4}} + \| f \|.
\]
From \eqref{model:1} we see that
\[
\| \pd_t \vp(0) \| \leq C \Big ( \| \vp_0 \|_{H^2} + \| \theta_0 \| + 1 \Big ),
\]
and so by Gronwall's inequality we deduce that 
\begin{align}\label{pdt:vp:reg}
\| \pd_t \vp \|_{L^\infty(0,T;L^2)} + \| \pd_t \vp \|_{L^2(0,T;H^1)} \leq C.
\end{align}
\paragraph{Second estimate.} Returning now to \eqref{model:1} and by a comparison of terms we see that 
\[
\| \Delta \vp \| \leq C \Big (1 +  \| \pd_t \vp \| + \| W'(\vp) \| + \| \theta \| \Big ).
\]
Boundedness of $\vp$ in $L^\infty(0,T;H^1(\Omega))$ and \eqref{ass:W} imply $W'(\vp)$ is bounded in $L^\infty(0,T;L^2(\Omega))$, and hence we deduce that $\Delta \vp$ is bounded in $L^\infty(0,T;L^2(\Omega))$.  Then, elliptic regularity allows us to infer that 
\begin{align}\label{vp:LinftyH2}
\| \vp \|_{L^\infty(0,T;H^2)} \leq C.
\end{align}
\paragraph{Third estimate.} Next, taking the time derivative of \eqref{model:2}:
\begin{align}\label{theta:tt}
 \pd_{tt} \theta - \Delta \pd_t \theta =  p'(\vp) |\pd_t \vp|^2 +  p(\vp) \pd_{tt} \vp,
\end{align}
and testing with $\pd_t \theta$ yields
\begin{align*}
\frac{1}{2} \frac{d}{dt} \| \pd_t \theta \|^2 + \| \nabla \pd_t \theta \|^2 = ( p'(\vp) |\pd_t \vp|^2, \pd_t \theta) + ( p(\vp) \pd_{tt} \vp, \pd_t \theta).
\end{align*}
On the other hand, testing the time derivative of \eqref{model:1} with $\pd_{tt} \vp$ leads to 
\begin{align*}
 \| \pd_{tt} \vp \|^2 + \frac{1}{2} \frac{d}{dt} \| \nabla \pd_t \vp \|^2 = - ( W''(\vp) \pd_t \vp +  p(\vp) \pd_t \theta + (\theta - \theta_c) p'(\vp) \pd_t \vp, \pd_{tt} \vp).
\end{align*}
Summing the above identities and noting a cancellation leads to 
\begin{align*}
& \frac{1}{2} \frac{d}{dt} \Big (  \| \pd_t \theta \|^2 + \| \nabla \pd_t \vp \|^2 \Big ) +  \| \pd_{tt} \vp \|^2 + \| \nabla \pd_t \theta \|^2 \\
& \quad \leq C \|\pd_t \vp \|_{L^3}^2 \| \pd_t \theta \|_{L^3} + C \|W''(\vp) \pd_t \vp \| \| \pd_{tt} \vp \| + C \Big ( 1 + \| \theta \|_{L^\infty} \Big ) \| \pd_t \vp \| \| \pd_{tt} \vp \| \\
& \quad \leq C \| \pd_t \vp \|_{H^1} \Big (\| \pd_t \theta \|^{\frac{1}{2}} \| \nabla \pd_t \theta \|^{\frac{1}{2}} + \| \pd_t \theta \| \Big ) + C\Big ( 1 +  \| W''(\vp) \pd_t \vp \|^2  + \| \theta \|_{L^\infty}^2 \Big ) + \frac{1}{2} \| \pd_{tt} \vp \|^2,
\end{align*}
where we have used the Gagliardo--Nirenburg inequality and the boundedness of $\pd_t \vp$ in $L^\infty(0,T;L^2(\Omega))$. From \eqref{ass:W} we see that $W''$ has quadratic polynomial growth and so 
\[
\int_\Omega |W''(\vp) \pd_t \vp|^2 \dx \leq C \int_\Omega ( 1 + |\vp|^4) |\pd_t \vp|^2 \dx \leq C(1 + \| \vp \|_{L^\infty}^4) \| \pd_t \vp \|^2 \leq C
\]
thanks to \eqref{pdt:vp:reg} and \eqref{vp:LinftyH2}. Hence, we obtain
\begin{align*}
 \frac{d}{dt} \Big (  \| \pd_t \theta \|^2 + \| \nabla \pd_t \vp \|^2 \Big ) +  \| \pd_{tt} \vp \|^2 +  \| \nabla \pd_t \theta \|^2 \leq C \Big ( 1 + \| \theta \|_{H^2} \Big ) + C \Big ( 1 + \| \pd_t \vp \|_{H^1}^2 \Big ) \| \pd_t \theta \|^2.
\end{align*}
From \eqref{model:1} and \eqref{model:2} we see that 
\begin{align*}
\| \pd_t \theta(0) \| & \leq C \| \pd_t \vp(0) \| + C\| \Delta \theta_0 \| \leq C \Big ( \| \theta_0 \|_{H^2} + \| \vp_0 \|_{H^2} \Big ), \\
 \| \nabla \pd_t \vp(0) \| & \leq C\Big ( \| \vp_0 \|_{H^3} + \| \theta_0 \|_{H^1} + 1 \Big ).
\end{align*}
Hence, by a Gronwall argument we deduce that 
\begin{align}\label{pdt:theta:reg}
\| \pd_t \theta \|_{L^\infty(0,T;L^2)} + \| \nabla \pd_t \vp \|_{L^\infty(0;T;L^2)} + \| \pd_{tt} \vp \|_{L^2(0,T;L^2)} + \| \nabla \pd_t \theta \|_{L^2(0,T;L^2)} \leq C.
\end{align} 
Then, a comparison of terms in \eqref{model:2} leads to 
\[
\|\Delta \theta \| \leq C \Big ( \| \pd_t \theta \| + \| \pd_t \vp \| \Big ),
\]
which combining with \eqref{pdt:theta:reg} implies 
\[
\| \theta \|_{L^\infty(0,T;H^2)} \leq C.
\]
\paragraph{Fourth estimate.} We test \eqref{theta:tt} with $\pd_{tt} \theta$ to obtain
\begin{align*}
 \| \pd_{tt} \theta \| + \frac{d}{dt} \frac{1}{2} \| \nabla \pd_t \theta \|^2 &  \leq C \| \pd_t \vp \|_{L^4}^2 \| \pd_{tt} \theta \| + C\| \pd_{tt} \vp \| \| \pd_{tt} \theta \| \\
& \leq \frac{1}{2} \| \pd_{tt} \theta \|^2 + C \| \pd_{tt} \vp \|^2 + C  \| \pd_{tt} \vp \|_{H^1}^4.
\end{align*}
Invoking Gronwall's inequality, the estimate \eqref{pdt:theta:reg} and the fact that $\| \nabla \pd_t \theta(0) \| \leq C( \| \theta_0 \|_{H^3} + \| \vp_0 \|_{H^3} + 1)$ we infer that 
\[
\| \nabla \pd_t \theta \|_{L^\infty(0,T;L^2)} + \| \pd_{tt} \theta \|_{L^2(0,T;L^2)} \leq C.
\]
With $\pd_{tt} \vp, \pd_{tt} \theta \in L^2(0,T;L^2(\Omega))$, when we revisit \eqref{vp:tt} and \eqref{theta:tt} we find that 
\begin{align*}
\| \Delta \pd_t \vp \| & \leq C \big ( \| \pd_{tt} \vp \| + \| W''(\vp) \|_{L^4} \| \pd_t \vp \|_{L^4} + \| \pd_t \theta \| + \| \theta - \theta_c \|_{L^\infty} \| \pd_t \vp \| \big ), \\
\| \Delta \pd_t \theta \| & \leq C \big ( \| \pd_{tt} \theta \| + \| \pd_t \vp \|_{L^4}^2 + \| \pd_{tt} \vp \| \big ).
\end{align*}
Elliptic regularity shows that 
\begin{align*}
\| \pd_t \vp \|_{L^2(0,T;H^2)} + \| \pd_t \theta \|_{L^2(0,T;H^2)} \leq C.
\end{align*}
\paragraph{Fifth estimate.} We now take the time derivative of \eqref{model:3}, then testing with $\bu_t$ and applying the coercivity of $\CC$, boundedness of $\CC'$ and $m'$, as well as the regularities $\pd_t \vp, \pd_t \theta \in L^\infty(0,T;H^1(\Omega)) \cap L^2(0,T;H^2(\Omega))$ leads to 
\begin{align*}
 \| \E(\bu_t) \|^2 &  \leq  C \big ( 1 +  \| \pd_t \vp \|^2 + \| \pd_t \theta \|^2 + \| \theta \|^2 + \| \pd_t \vp \|_{L^\infty}^2 \| \E(\bu) \|^2 \big ) \\
& \leq C \big ( 1 + \| \pd_t \vp \|_{H^2}\| \pd_t \vp \|_{H^1} \big ).
\end{align*}
Hence, squaring both sides and invoking Korn's inequality yields
\[
\| \bu_t \|_{L^4(0,T;X(\Omega))} \leq C.
\]
\paragraph{Sixth estimate.} By expressing \eqref{model:1} and \eqref{model:2} as elliptic systems with right-hand side bounded in $L^\infty(0,T;L^6(\Omega))$, we deduce that $\Delta \vp$ and $\Delta \theta$ are bounded in $L^\infty(0,T;L^6(\Omega))$, and so from elliptic regularity theory we have
\[
\| \vp \|_{L^\infty(0,T;W^{2,6})} + \| \theta \|_{L^\infty(0,T;W^{2,6})} \leq C.
\]
\end{proof}

\section{Error estimates}\label{sec:error}
\subsection{Time discretization error estimates}
For continuous-in-time functions we use the notation $v^k(\cdot) := v(t^k, \cdot)$. Furthermore we introduce the notation
\[
\Psi(s) = \frac{W'(s)}{Q(s)}
\]
with $p^k := p(\vp^k)$ and $\Psi^k := \Psi(\vp^k)$. Then, we rewrite \eqref{weak:SAV:1}-\eqref{SAV:q} (with all parameters set to unity) for the exact solution $(\vp, \theta, q)$ from Theorem \ref{thm:conv} evaluated at time $t^n$:
\begin{subequations}\label{err:sys:1}
\begin{alignat}{2}
\label{Err:cts:1} & ( \vp^{n} - \vp^{n-1}, \psi ) + \tau (\nabla \vp^{n}, \nabla \psi) + \tau ((\theta^{n} - \theta_c)p^{n-1}, \psi) + \tau q^n (\Psi^{n-1}, \psi) \\
\notag & \quad = - (\tau \pd_t \vp^{n} - (\vp^n - \vp^{n-1}), \psi) - \tau ((\theta^n - \theta_c)(p^n - p^{n-1}), \psi)  - \tau q^n (\Psi^n - \Psi^{n-1}, \psi ) \\
\notag & \quad =: (X_\vp^{n-1}, \psi), \\
\label{Err:cts:2} & (\theta^n - \theta^{n-1}, \psi) + \tau (\nabla \theta^n, \nabla \psi) - (p^{n-1}(\vp^n - \vp^{n-1}), \psi) \\
\notag & \quad =- (\tau \pd_t \theta^n - (\theta^n - \theta^{n-1}), \psi) - (p^{n-1}(\vp^n - \vp^{n-1} - \tau \pd_t \vp^n), \psi) - ([p^{n-1} - p^n] \tau \pd_t \vp^n , \psi), \\
\notag & \quad =: (X_\theta^{n-1}, \psi), \\
\label{Err:cts:3} & q^n - q^{n-1} - \tfrac{1}{2} (\Psi^{n-1}, (\vp^n - \vp^{n-1})) \\
\notag & \quad = - (\tau \pd_t q^n - (q^n - q^{n-1})) - \tfrac{1}{2}  (\Psi^{n-1}, \vp^n - \vp^{n-1} - \tau \pd_t \vp^n )- \tfrac{1}{2} (\Psi^{n-1} - \Psi^n, \tau \pd_t \vp^n  ), \\
\notag & \quad =: X_q^{n-1},
\end{alignat} 
\end{subequations}
for arbitrary $\psi \in H^1(\Omega)$.
\begin{lem}\label{lem:timedisc:err}
Under \eqref{ass:dom}-\eqref{ass:Reg}, there exists a positive constant $C$ independent of $\tau$ and $n \in \{1, \dots, N_\tau\}$ such that 
\begin{align*}
\| X_\vp^{n-1} \| & \leq C \tau^{3/2} \big ( \| \pd_{tt} \vp \|_{L^2(t^{n-1},t^n;L^2)} + \| \pd_{t} \vp \|_{L^2(t^{n-1}, t^n;L^2)} \big ),  \\
\| X_\theta^{n-1} \| & \leq C \tau^{3/2} \big (\| \pd_{tt} \theta \|_{L^2(t^{n-1},t^n;L^2)} + \| \pd_{tt} \vp \|_{L^2(t^{n-1}, t^n; L^2)} \big ) + C \tau^{7/4} \| \pd_t \vp \|_{L^4(t^{n-1}, t^n;L^4)}, \\
|X_q^{n-1}| & \leq C \tau^{3/2} \big ( \|q''\|_{L^2(t^{n-1}, t^n)} + \| \pd_{tt} \vp \|_{L^2(t^{n-1}, t^n;L^2)} \big ).
\end{align*}
Consequently
\begin{align}\label{timedisc:err}
\sum_{n=1}^{N_\tau} \frac{1}{\tau} \big (\| X_\vp^{n-1} \|^2 + \|X_\theta^{n-1} \|^2 + |X_q^{n-1}|^2 \big ) \leq C \tau^2.
\end{align}
\end{lem}

\begin{proof}
First, using the relation $q(t) = Q(\vp(t))$, by a direct calculation
\begin{align*}
q''  = \frac{1}{2 Q(\vp)} \int_\Omega W''(\vp) (\pd_t \vp)^2 + W'(\vp) \pd_{tt} \vp \dx - \frac{1}{4 Q^3(\vp)} \Big ( \int_\Omega W'(\vp) \pd_t \vp \dx \Big )^2
\end{align*}
for a.e.~$t \in (0,T)$. Invoking the lower bound on $Q$ and employing the regularity $\vp \in L^\infty(0,T;H^2_n(\Omega)) \cap W^{1,\infty}(0,T;H^1(\Omega)) \cap H^2(0,T;L^2(\Omega))$, we see that $q \in H^2(0,T)$, since
\begin{align}\label{qttL2}
\int_0^T |q''|^2 \dt \leq C \int_0^T \| \pd_t \vp \|_{L^4}^4 + \| \pd_{tt} \vp \|^2 \dt \leq C.
\end{align}
Next, using Taylor's theorem with integral remainder:
\[
f(t^n) - f(t^{n-1}) - \tau \pd_t f(t^{n}) = \int_{t^{n-1}}^{t^n} (t^n - s) \pd_{tt} f(s) \, ds,
\]
we see that for $T_f^{n-1} := f(t^n) - f(t^{n-1}) - \tau \pd_t f(t^n)$, where $f \in \{\vp, \theta, q\}$, it holds that 
\begin{equation}\label{Trunc:est}
\begin{aligned}
\| T_\vp^{n-1} \|^2 & \leq C \tau^3 \| \pd_{tt} \vp \|_{L^2(t^{n-1}, t^n; L^2)}^2, \quad \| T_\theta^{n-1} \|^2 \leq C \tau^3 \| \pd_{tt} \theta \|_{L^2(t^{n-1}, t^n; L^2)}^2, \\
|T_q^{n-1}|^2 & \leq C \tau^3 \|q'' \|_{L^2(t^{n-1}, t^n)}^2.
\end{aligned}
\end{equation}
A short calculation shows 
\begin{align}\label{Psi:diff}
 \Psi^{k-1} - \Psi^k = \frac{W'(\vp^{k-1}) - W'(\vp^k)}{Q(\vp^{k-1})} +  \frac{W'(\vp^k)[Q^2(\vp^k) - Q^2(\vp^{k-1})]}{(Q(\vp^k) + Q(\vp^{k-1})) Q(\vp^k) Q(\vp^{k-1})}.
\end{align}
From the definition of $Q$ and using \eqref{W:diff} we see 
\begin{align*}
| Q^2(\vp^k) - Q^2(\vp^{k-1})| & \leq C \| W(\vp^k) - W(\vp^{k-1}) \|_{L^1} \\
& \leq C( 1 + \| \vp^k \|_{L^6}^3 + \| \vp^{k-1} \|_{L^6}^3) \| \vp^k - \vp^{k-1} \| \\
& \leq C\| \vp^k - \vp^{k-1} \| \leq C \tau^{1/2} \| \pd_t \vp \|_{L^2(t^{k-1}, t^k;L^2)}
\end{align*}
while invoking the Gagliardo--Nirenburg inequality and \eqref{W':diff} we have
\begin{align*}
\| W'(\vp^k) - W'(\vp^{k-1}) \| & \leq C ( 1 + \| \vp^k \|_{L^\infty}^2 + \| \vp^{k-1} \|_{L^\infty}^2) \|\vp^k - \vp^{k-1} \| \\
& \leq C \| \vp^k - \vp^{k-1} \| \leq C \tau^{1/2} \| \pd_t \vp \|_{L^2(t^{k-1}, t^k;L^2)}.
\end{align*}
Hence, we obtain
\begin{align}\label{Psi:diffL2:b}
\| \Psi^k - \Psi^{k-1} \| \leq C\tau^{1/2} \| \pd_t \vp \|_{L^2(t^{k-1}, t^k; L^2)}.
\end{align}
Likewise, using the Lipschitz continuity of $p = P'$, we also have for $r \in [2,\infty)$
\begin{align}\label{p':diff}
\| p^k - p^{k-1} \|_{L^r} & \leq C \| \vp^k - \vp^{k-1} \|_{L^r} \leq C\tau^{\frac{r-1}{r}} \| \pd_t \vp \|_{L^r(t^{k-1},t^k;L^r)}.
\end{align}
Then, from \eqref{Err:cts:1}-\eqref{Err:cts:3} we infer that 
\begin{align*}
\| X_\vp^{n-1} \| & \leq \| T_\vp^{n-1} \| + \tau \| \theta^n - \theta_c \|_{L^\infty} \| p^n - p^{n-1}\| + \tau |q^n| \| \Psi^n - \Psi^{n-1} \| \\
& \leq C \tau^{3/2} \big ( \| \pd_{tt} \vp \|_{L^2(t^{n-1},t^n;L^2)} + \| \pd_{t} \vp \|_{L^2(t^{n-1}, t^n;L^2)} \big ),  \\
\| X_\theta^{n-1} \| & \leq \| T_\theta^{n-1} \| + \| p^{n-1} \|_{L^\infty} \| T_\vp^{n-1} \| + \tau \| p^{n-1} - p^{n}\|_{L^4} \| \pd_t \vp^n \|_{L^4} \\
& \leq C \tau^{3/2} \big (\| \pd_{tt} \theta \|_{L^2(t^{n-1},t^n;L^2)} + \| \pd_{tt} \vp \|_{L^2(t^{n-1}, t^n; L^2)} \big ) + C \tau^{7/4} \| \pd_t \vp \|_{L^4(t^{n-1}, t^n;L^4)}, \\
|X_q^{n-1}| & \leq |T_q^{n-1}| + C \| W'(\vp^n) \| \| T_\vp^{n-1} \| + C \tau \| \Psi^{n-1} - \Psi^n \| \| \pd_t \vp^n \| \\
& \leq C \tau^{3/2} \big ( \|q''\|_{L^2(t^{n-1}, t^n)} + \| \pd_{tt} \vp \|_{L^2(t^{n-1}, t^n;L^2)} \big ).
\end{align*}
To derive the estimate \eqref{timedisc:err} it suffices to square both sides of the above inequalities, divide by $\tau$ and sum from $n = 1$ to $n = N_\tau$. For the term involving $\tau^{7/4}$ we used the Cauchy--Schwarz inequality to see that
\[
\sum_{n=1}^{N_\tau} \tau^{5/2} \| \pd_t \vp \|_{L^4(t^{n-1},t^n;L^4)}^2 \leq \tau^{5/2} \| \pd_t \vp \|_{L^4(0,T;L^4)}^2 \sqrt{N_\tau} \leq C \tau^2.
\]
\end{proof}

\subsection{Optimal error estimates for the Caginalp system}
We introduce the discrete norm
\[
\| f \|_{l^\infty(H^s)} := \sup_{1 \leq n \leq N_\tau} \| f(t^n) \|_{H^s},
\]
with the convention that $H^0(\Omega) = L^2(\Omega)$.

\begin{thm}[Error estimates for the Caginalp system]\label{thm:errest}
Let $(\vp_h, \theta_h, q_h)$ be the fully discrete solution to \eqref{SAV}, and let $(\vp, \theta, q = Q(\vp))$ be the unique solution to \eqref{model} with the regularity stated in Theorem \ref{thm:reg}. Then, there exists a positive constant $C$ independent of $h$ and $\tau$ such that 
\begin{equation}
\begin{aligned}
\| \vp - \vp_h \|_{l^\infty(L^2)}  + \| \theta - \theta_h \|_{l^\infty(L^2)} +  |q - q_h| &  \leq C (h^2 + \tau), \\
 \| \nabla (\vp - \vp_h) \|_{l^\infty(L^2)} + \| \nabla (\theta - \theta_h) \|_{l^\infty(L^2)} & \leq C(h + \tau).
\end{aligned}
\end{equation}
\end{thm}
\begin{proof}
We recall the Ritz projection $R_h : H^1(\Omega) \to \Sh$ and define for $f \in \{\vp, \theta\}$ the decomposition of the error $f^k_h(\cdot) - f(t^k, \cdot)$ into
\[
\kappa_f^{k} := f_h^{k} - R_h f(t^{k+1}), \quad \rho_f^{k} := f(t^{k+1}) - R_h f(t^{k+1})
\]
so that $f_h^k - f(t^k) = \kappa_f^k - \rho_f^k$, while we define $e^n := q_h^n - q^n$. Furthermore we use the notation $p_h^k := p(\vp_h^k)$ and $\Psi_h^k := \Psi(\vp_h^k)$. Then, taking the difference between the fully discrete scheme \eqref{SAV} for $(\vp_h^n, \theta_h^n, q_h^n)$ and the system \eqref{err:sys:1} for $(\vp^n, \theta^n, q^n)$ we find that for arbitrary $\psi_h \in \Sh$,
\begin{subequations}
\begin{alignat}{2}
\label{err:vp} 0 &= (\kappa_\vp^n - \rho_{\vp}^n - \kappa_\vp^{n-1}  + \rho_{\vp}^{n-1},  \psi_h) + \tau (\nabla \kappa_{\vp}^n, \nabla \psi_h) + (X_\vp^{n-1}, \psi_h) \\
\notag & \quad + \tau e^n (\Psi^{n-1}_h, \psi_h) + \tau q^n (\Psi^{n-1}_h - \Psi^{n-1}, \psi_h)  \\
\notag & \quad + \tau ((\kappa_\theta^n - \rho_\theta^n) p'(\vp_h^{n-1}) , \psi_h) + \tau ((\theta^n - \theta_c)(p^{n-1}_h - p^{n-1}), \psi_h), \\
\label{err:theta} 0& = (\kappa_\theta^n - \rho_\theta^n - \kappa_\theta^{n-1} + \rho_{\theta}^{n-1}, \psi_h) + \tau (\nabla \kappa_{\theta}^n, \nabla \psi_h) + (X_\theta^{n-1}, \psi_h) \\
\notag & \quad - (p^{n-1}_h (\kappa_\vp^n - \rho_\vp^n - \kappa_{\vp}^{n-1} + \rho_{\vp}^{n-1}), \psi_h)  - ((p^{n-1}_h - p^{n-1}) (\vp^n - \vp^{n-1}), \psi_h), \\
\label{err:q} 0 & = e^n - e^{n-1} - \tfrac{1}{2} (\Psi^{n-1}_h, \kappa_\vp^n - \rho_\vp^n - \kappa_\vp^{n-1} + \rho_{\vp}^{n-1}) + X_q^{n-1} \\
\notag & \quad - \tfrac{1}{2}(\Psi^{n-1}_h - \Psi^{n-1}, \vp^n - \vp^{n-1}).
\end{alignat}
\end{subequations}

\paragraph{Induction argument.}Similar to \cite{ChenSAV} we invoke a mathematical induction on 
\begin{align}\label{induction}
\| \vp_h^k \|_{L^\infty} \leq \| \vp \|_{L^\infty(0,T;L^\infty)} + 1
\end{align}
for all $k = 1, \dots, N_\tau$. For $n = 0$, we use \eqref{Ritz2} and $\vp_h^0 = R_h \vp_0$ to deduce the existence of $h_0 > 0$ such that 
\[
\| \vp_h^0 \|_{L^\infty} \leq \| \vp \|_{L^\infty(0,T;L^\infty)} + 1
\]
valid for all $h < h_0$. We now assume \eqref{induction} holds for $k = 0, 1, \dots, M-1$, and the induction argument for $k = M$ for arbitrary $M \in \{1, \dots, N_\tau\}$ is established once we demonstrate 
\begin{align}\label{induction:req}
\max_{1 \leq k \leq M} \| \kappa_\vp^k \|_{H^1}^2 + \tau \sum_{n=1}^M \| \Delta_h \kappa_\vp^n \|^2 \leq C(h^4 + \tau^2)
\end{align}
for a positive constant $C$ independent of $M$.

\paragraph{First estimate.} Choosing $\psi_h = \frac{1}{\tau} (\kappa_{\vp}^n - \kappa_{\vp}^{n-1})$ in \eqref{err:vp}, $\psi_h = \kappa_\theta^n$ in \eqref{err:theta} and multiplying \eqref{err:q} by $2 e^n$, upon summing and noting a cancellation of terms involving $(\Psi_h^{n-1}, \kappa_{\vp}^n - \kappa_{\vp}^{n-1})e^n$ and $(p^{n-1}_h (\kappa_\vp^n - \kappa_{\vp}^{n-1}), \kappa_\theta^n)$, we obtain
\begin{equation}\label{err:1}
\begin{aligned}
& \frac{1}{2} (\| \nabla \kappa_\vp^n \|^2 + \| \kappa_\theta^n \|^2 + 2 |e^n|^2) - \frac{1}{2}(\| \nabla \kappa_{\vp}^{n-1} \|^2 + \| \kappa_\theta^{n-1} \|^2 + 2 |e^{n-1}|^2) \\
& \qquad + \frac{1}{2} (\| \nabla (\kappa_\vp^n - \kappa_\vp^{n-1}) \|^2 + \| \kappa_\theta^n - \kappa_\theta^{n-1} \|^2 + 2 |e^n - e^{n-1}|^2) \\
& \qquad + \tau \| \nabla \kappa_\theta^n \|^2  + \frac{1}{\tau} \| \kappa_\vp^n - \kappa_{\vp}^{n-1} \|^2 \\
& \quad = J_1 + J_2 + J_3,
\end{aligned}
\end{equation}
where
\begin{align*}
J_1 & = - 2X_q^{n-1} e^n + e^n (\Psi_h^{n-1}, \rho_\vp^{n} - \rho_{\vp}^{n-1}) + e^n (\Psi_h^{n-1} - \Psi^{n-1}, \vp^n - \vp^{n-1}) \\
& \quad - q^n (\Psi_h^{n-1} - \Psi^{n-1}, \kappa_\vp^n - \kappa_{\vp}^{n-1}), \\ 
J_2 & = (\rho_\vp^n - \rho_{\vp}^{n-1} - X_\vp^{n-1}, \tfrac{1}{\tau}(\kappa_\vp^n - \kappa_{\vp}^{n-1})) + (p_h^{n-1} \rho_\theta^n, \kappa_\vp^n - \kappa_\vp^{n-1}) \\
& \quad + ((\theta^n - \theta_c) (p_h^{n-1} - p^{n-1}), \kappa_\vp^n - \kappa_\vp^{n-1}), \\
J_3 & = -(X_\theta^{n-1} + p_h^{n-1} (\rho_\vp^n - \rho_\vp^{n-1}) - (p_h^{n-1} - p^{n-1})(\vp^n - \vp^{n-1}), \kappa_\theta^n).
\end{align*}
Let us collect a few useful estimates: Using \eqref{Ritz1} we have
\begin{align}
\| \rho_\vp^{k} \| & \leq Ch^s \| \vp^k \|_{H^s},\label{rho:est:1} \\
\| \rho_\vp^n - \rho_\vp^{n-1} \| & \leq C h^s \| \vp^n - \vp^{n-1} \|_{H^s} = C h^s \tau^{1/2} \| \pd_t \vp \|_{L^2(t^{n-1}, t^n; H^s)}. \label{rho:est:2}
\end{align}
On the other hand, by the Lipschitz continuity of $p$ 
\begin{align}\label{p:err:est}
\| p_h^{n-1} - p^{n-1} \|_{L^r} \leq C \| \kappa_\vp^{n-1} - \rho_\vp^{n-1} \|_{L^r},
\end{align}
and by an analogous calculation to \eqref{Psi:diff} where upon employing \eqref{W:diff}, \eqref{W':diff} and \eqref{induction}:
\begin{align}\label{Psi:err:est}
\| \Psi_h^{n-1} - \Psi^{n-1} \| \leq C \| \kappa_\vp^{n-1} - \rho_\vp^{n-1} \|.
\end{align}
Furthermore, we use the fact that $\pd_t \vp \in L^\infty(0,T;H^1(\Omega))$ to see that for $r \in [2,6]$,
\begin{align}\label{vpn:est}
\frac{1}{\tau} \| \vp^n - \vp^{n-1}\|_{L^r} \leq \frac{1}{\tau} \Big ( \int_{t^{n-1}}^{t^n} \| \pd_t \vp \|_{L^r}^r \, dt \Big)^{1/r} \tau^{\frac{r-1}{r}} \leq \| \pd_t \vp \|_{L^\infty(0,T;H^1)} \leq C,
\end{align}
which would be helpful for estimating the third term of $J_1$ and the last term of $J_3$. Then, the right-hand side of \eqref{err:1} can be estimated as follows:
\begin{align*}
J_1 & \leq C\frac{\tau^{1/2}}{\tau^{1/2}} |X_q^{n-1}| |e^n| + \| \Psi^{n-1}_h \| \frac{\tau^{1/2}}{\tau^{1/2}} |e^n| \| \rho_\vp^n - \rho_\vp^{n-1} \|  \\
& \quad + \tau |e^n| \| \Psi_h^{n-1} - \Psi^{n-1} \| \tau^{-1} \| \vp^{n} - \vp^{n-1} \|+ C \frac{\tau^{1/2}}{\tau^{1/2}} |q^n| \| \Psi_h^{n-1} - \Psi^{n-1} \| \| \kappa_\vp^n - \kappa_\vp^{n-1} \| \\
& \leq C \tau |e^n|^2 + \frac{C}{\tau} |X_q^{n-1}|^2 + \frac{C}{\tau} \| \rho_\vp^n - \rho_\vp^{n-1} \|^2 + C \tau \| \kappa_\vp^{n-1} - \rho_\vp^{n-1} \|^2 + \frac{1}{4 \tau} \| \kappa_\vp^n - \kappa_\vp^{n-1} \|^2, \\
J_2 & \leq \frac{1}{4 \tau} \| \kappa_\vp^n - \kappa_{\vp}^{n-1} \|^2 + \frac{C}{\tau} \Big ( \|X_\vp^{n-1} \|^2 + \| \rho_\vp^{n} - \rho_\vp^{n-1} \|^2 + \tau^2 \| \rho_\theta^n \|^2 + \tau^2 \| \kappa_\vp^{n-1} - \rho_\vp^{n-1} \|^2 \Big ), \\
J_3 & \leq \frac{\tau^{1/2}}{\tau^{1/2}}( \| X_\theta^{n-1} \| + \| \rho_\vp^n - \rho_\vp^{n-1} \|) \| \kappa_\theta^n \| + C\tau \| \kappa_\theta^n \|_{L^3} \| \| \kappa_\vp^{n-1} - \rho_\vp^{n-1} \| \tau^{-1} \| \vp^n - \vp^{n-1} \|_{L^6} \\
& \leq C \tau \| \kappa_\theta^n \|^2 + \frac{\tau}{2} \| \nabla \kappa_\theta^n \|^2 + \frac{C}{\tau} \| X_\theta^{n-1} \|^2 + \frac{C}{\tau} \| \rho_\vp^{n} - \rho_\vp^{n-1} \|^2 + C \tau \| \kappa_\vp^{n-1} -\rho_\vp^{n-1} \|^2,
\end{align*}
where we have used \eqref{induction} so that $\| \Psi_h^{n-1} \| \leq C$, \eqref{p:err:est}, \eqref{Psi:err:est}, $q \in L^\infty(0,T)$ and $\theta \in L^\infty(0,T;H^2(\Omega))$.

Recalling the discrete Neumann-Laplacian \eqref{disc:lap} we obtain from \eqref{err:vp} the estimate
\begin{equation}
\begin{aligned}
\tau \| \Delta_h \kappa_\vp^n \|^2 & \leq \frac{1}{\tau} \| X_\vp^{n-1} \|^2 + \frac{1}{\tau} \| \kappa_\vp^{n} - \kappa_\vp^{n-1} \|^2 + \frac{1}{\tau} \| \rho_\vp^n - \rho_{\vp}^{n-1} \|^2 \\
& \quad  + C \tau |e^n|^2 + C \tau \|\kappa_\vp^{n-1} - \rho_\vp^{n-1} \|^2 
\end{aligned}
\end{equation}
which after multiplying by $\frac{1}{4}$ we add to \eqref{err:1}. Then, upon neglecting some non-negative terms and applying the estimates for $J_1$, $J_2$ and $J_3$, as well as \eqref{rho:est:1} and \eqref{rho:est:2}, we infer that for any $\tau < 1$,
\begin{equation}\label{err:2}
\begin{aligned}
& \frac{1}{2} (\| \nabla \kappa_\vp^n \|^2 + \| \kappa_\theta^n \|^2 + 2 |e^n|^2) - \frac{1}{2}(\| \nabla \kappa_{\vp}^{n-1} \|^2 + \| \kappa_\theta^{n-1} \|^2 + 2 |e^{n-1}|^2) \\
& \qquad + \frac{\tau}{2} \| \nabla \kappa_\theta^n \|^2 + \frac{1}{4\tau} \| \kappa_\vp^n - \kappa_{\vp}^{n-1} \|^2 + \frac{\tau}{4} \| \Delta_h \kappa_\vp^n \|^2 \\
& \quad \leq C\tau \big ( |e^n|^2 + \| \kappa_\theta^n \|^2 \big ) + \frac{C}{\tau} \big (|X_q^{n-1}|^2 + \| X_\vp^{n-1} \|^2 + \| X_\theta^{n-1} \|^2 \big ) + C \tau \| \rho_\theta^n \|^2 \\
& \qquad + \frac{C}{\tau} \| \rho_\vp^n - \rho_\vp^{n-1} \|^2 + C \tau \| \kappa_\vp^{n-1} - \rho_\vp^{n-1} \|^2 \\
& \quad \leq C\tau \big ( |e^n|^2 + \| \kappa_\theta^n \|^2 \big ) + \frac{C}{\tau} \big (|X_q^{n-1}|^2 + \| X_\vp^{n-1} \|^2 + \| X_\theta^{n-1} \|^2 \big ) + C \tau h^{4} \| \theta^n \|_{H^2}^2 \\
& \qquad + Ch^{4} \| \pd_t \vp \|_{L^2(t^{n-1},t^n;H^2)}^2 + C \tau \| \kappa_\vp^{n-1} \|^2 + C \tau h^{4} \| \vp^{n-1} \|_{H^2}^2.
\end{aligned}
\end{equation}
Recalling the initialization $\vp_h^0 = R_h \vp_0$, $\theta_h^0 = R_h \theta_0$ and $q_h^0 = Q(\vp_h^0)$ for the fully discrete scheme \eqref{SAV}, we see that 
\[
\| \nabla \kappa_\vp^0 \| \leq C h^2 \| \vp_0 \|_{H^3}, \quad \| \kappa_\theta^0 \| \leq Ch^2 \| \theta_0 \|_{H^2}, \quad |e^0| \leq C \| \vp_h^0 - \vp_0 \| \leq C \| \rho_\vp^0 \| \leq C h^2 \| \vp_0 \|_{H^2}.
\]
Summing \eqref{err:2} from $n = 1$ to $n = k$ for arbitrary $k \in \{1, \dots, M\}$, and applying the discrete Gronwall inequality and Lemma \ref{lem:timedisc:err} leads to 
\begin{equation}\label{err:3}
\begin{aligned}
& \frac{1}{2} (\| \nabla \kappa_\vp^k \|^2 + \| \kappa_\theta^k \|^2 + 2 |e^k|^2) + \sum_{n=1}^{k} \Big ( \frac{\tau}{2} \| \nabla \kappa_\theta^n \|^2 + \frac{1}{4 \tau} \| \kappa_\vp^{n} - \kappa_\vp^{n-1} \|^2 + \frac{\tau}{4} \| \Delta_h \kappa_\vp^n \|^2 \Big ) \\
& \quad  \leq C (h^4 + \tau^2).
\end{aligned}
\end{equation}
\paragraph{Second estimate.} Choosing $\psi_h = \kappa_\vp^n$ in \eqref{err:vp} gives
\begin{align*}
& \frac{1}{2} (\| \kappa_\vp^n \|^2 - \| \kappa_\vp^{n-1} \|^2 + \| \kappa_\vp^{n} - \kappa_\vp^{n-1} \|^2 ) + \tau \| \nabla \kappa_\vp^{n} \|^2 \\
& \quad \leq C \tau \| \kappa_\vp^n \|^2 + \frac{C}{\tau} \big ( \| \rho_\vp^n - \rho_\vp^{n-1} \|^2 +  \| X_\vp^{n-1} \|^2 \big ) + C \tau  \big (|e^n|^2 + \| \kappa_\vp^{n-1} - \rho_\vp^{n-1} \|^2 + \| \kappa_\theta^n - \rho_\theta^n \|^2 \big )\\
& \quad \leq C \tau  \big ( \| \kappa_\vp^n \|^2  + \| \kappa_\vp^{n-1} \|^2 + |e^n|^2   + \| \kappa_\theta^n \|^2 \big ) + \frac{C}{\tau} \big ( \| \rho_\vp^n - \rho_\vp^{n-1} \|^2 + \| X_\vp^{n-1} \|^2 \big ) \\
& \qquad + C\tau h^{4} \big ( \| \vp^{n-1} \|_{H^2}^2 + \| \theta^n \|_{H^2}^2 \big ).
\end{align*}
Summing from $n = 1$ to $n = k$ for arbitrary $k \in \{1, \dots, M\}$, invoking \eqref{timedisc:err}, \eqref{rho:est:2}, \eqref{err:3} and the regularities $\vp, \theta \in L^\infty(0,T;H^2(\Omega))$, $\pd_t \vp \in L^2(0,T;H^2(\Omega))$ leads to 
\begin{align*}
\| \kappa_\vp^k \|^2 \leq C\| \kappa_\vp^0 \|^2 +  \sum_{n=1}^{k} C \tau \| \kappa_\vp^n \|^2 + C(h^4 + \tau^2) \leq  \sum_{n=1}^{k} C \tau \| \kappa_\vp^n \|^2 + C(h^4 + \tau^2).
\end{align*}
Then, by the discrete Gronwall inequality we have
\begin{align}\label{err:4}
\| \kappa_\vp^k \|^2 \leq C (h^4 + \tau^2).
\end{align}

\paragraph{Third estimate.} Choosing $\psi_h = \frac{1}{\tau}(\kappa_\theta^n - \kappa_\theta^{n-1})$ in \eqref{err:theta} gives
\begin{align*}
& \frac{1}{2} \big ( \| \nabla \kappa_\theta^n \|^2 - \| \nabla \kappa_\theta^{n-1} \|^2 + \| \nabla (\kappa_\theta^n - \kappa_\theta^{n-1}) \|^2 \big ) + \frac{1}{2\tau} \| \kappa_\theta^n - \kappa_\theta^{n-1} \|^2 \\
& \quad \leq \frac{C}{\tau} \big (\| X_\theta^{n-1} \|^2 + \| \rho_\theta^n - \rho_\theta^{n-1} \|^2 + \| \kappa_\vp^n - \kappa_\vp^{n-1} \|^2 + \| \rho_\vp^n - \rho_\vp^{n-1} \|^2 \big ) \\
& \qquad + \frac{C}{\tau} \| \kappa_\vp^{n-1} - \rho_\vp^{n-1} \|^2 \| \vp^{n} - \vp^{n-1} \|_{L^\infty}^2 \\
& \quad \leq \frac{C}{\tau} \| X_\theta^{n-1} \|^2 + \frac{C}{\tau} \| \kappa_\vp^n - \kappa_\vp^{n-1} \|^2 + C h^4 \big ( \| \pd_t \vp \|_{L^2(t^{n-1}, t^n; H^2)}^2 + \| \pd_t \theta \|_{L^2(t^{n-1}, t^n;H^2)}^2 \big ) \\
& \qquad + C (h^4 + \tau^2) \| \pd_t \vp \|_{L^2(t^{n-1}, t^n; L^\infty)}^2,
\end{align*}
where in the above we have used the estimates
\begin{align*}
& \| \vp^n - \vp^{n-1} \|_{L^\infty}^2 \leq \tau \int_{t^{n-1}}^{t^n} \| \pd_t \vp \|_{L^\infty}^2 \dt, \\
& \| \kappa_\vp^{n-1} - \rho_\vp^{n-1} \|^2 \leq C \| \kappa_\vp^{n-1} \|^2 + C \| \rho_\vp^{n-1} \|^2 \leq C(h^4 + \tau^2) 
\end{align*}
due to \eqref{rho:est:1} and \eqref{err:4}. Summing from $n = 1$ to $n = k$ for arbitrary $k \in \{1, \dots, M\}$ and using the discrete Gronwall inequality leads to 
\begin{align}\label{err:5}
\| \nabla \kappa_\theta^k \|^2 + \tau \sum_{n=1}^k \| \kappa_\theta^n - \kappa_\theta^{n-1} \|^2 \leq C(h^4 + \tau^2).
\end{align}

\paragraph{Induction step.} We proceed similarly as in \cite{ChenSAV}. If $\tau \leq h$, using the inverse inequality (see e.g.~\cite[Lemma 6.4]{Thomee}) and \eqref{err:3}-\eqref{err:4} we have
\[
\| \kappa_\vp^M \|_{L^\infty}^2 \leq C |\log(1/h)| \| \kappa_\vp^M \|_{H^1}^2 \leq C h^{-1}( h^4 + \tau^2) \leq C (h^3 + \tau).
\]
On the other hand, if $h \leq \tau$ then from \eqref{err:3} we infer 
\[
\| \Delta_h \kappa_\vp^k \|^2 \leq \frac{1}{\tau} \sum_{n=1}^k \tau \| \Delta_h \kappa_\vp^n \|^2 \leq C (h^4 \tau^{-1} + \tau) \leq C(h^3 + \tau),
\]
so that by the invoking the interpolation estimate for three spatial dimensions (see e.g.~\cite{Feng})
\[
\| \kappa_\vp^k \|_{L^\infty} \leq C \| \kappa_\vp^k \|^{\frac{1}{4}} (\| \kappa_\vp^k \|^2 + \| \Delta_h \kappa_\vp^h \|^2)^{\frac{3}{8}},
\]
it holds that
\[
\| \kappa_\vp^k \|_{L^\infty}^2 \leq C \| \kappa_\vp^k \|^2 + C \| \kappa_\vp^k \|^{1/2} \| \Delta_h \kappa_\vp^k \|^{3/2} \leq C(h^3 + \tau).
\]
Using the Sobolev embedding $W^{2,6}(\Omega) \subset W^{1,\infty}(\Omega)$ and \eqref{Ritz2} we have
\[
\| \kappa_\vp^M \|_{L^\infty} = \| \vp^M - R_h \vp^M \|_{L^\infty} \leq C h \ell_h \| \vp^M \|_{W^{1,\infty}},
\]
and so we can find constants $h_1 \in (0, h_0)$ and $\tau_1 > 0$ such that for all $h < h_1$, $\tau < \tau_1$,
\begin{align*}
\| \vp_h^M - \vp^M \|_{L^\infty} \leq \| \kappa_\vp^M \|_{L^\infty} + \| \vp^M - R_h \vp^M \|_{L^\infty} \leq C(h^{3/2} + \tau^{1/2}) + C h \ell_h\leq 1.
\end{align*}
This establishes \eqref{induction} for $k = M$. Hence, the estimates \eqref{err:3}-\eqref{err:5} hold for $k = M$ for all $h < h_1$ and $\tau < \tau_1$. On the other hand, if $h > h_1$ or $\tau > \tau_1$, then from the stability estimate \eqref{stab:est} of Lemma \ref{lem:stab} it holds that 
\begin{align*}
& \| \kappa_\vp^k \|_{H^1}^2 + \| \kappa_\theta^k \|_{H^1}^2 + |e^k|^2  \leq C + C \big ( \| R_h \vp^k \|_{H^1}^2 + \| R_h \theta^k \|_{H^1}^2 + |q^k|^2 \big ) \leq C
\end{align*}
with constant $C$ independent of $k$. In particular, we deduce that 
\begin{align}\label{err:6}
\| \kappa_\vp^k \|_{H^1}^2 + \| \kappa_\theta^k \|_{H^1}^2 + |e^k|^2 \leq C \leq C(\tau_1^{-2} + h_1^{-4})(\tau^2 + h^4)
\end{align}
if $\tau > \tau_1$ or $h > h_1$. Combining \eqref{err:3}-\eqref{err:3} and \eqref{err:6} we have that for any $\tau$ and $h$, 
\begin{align*}
\max_{1 \leq k \leq N_\tau} \big (\| \kappa_\vp^k \|_{H^1}^2 + \| \kappa_\theta^k \|_{H^1}^2 + |e^k|^2 \big ) \leq C(h^4 + \tau^2).
\end{align*}
Then, by \eqref{Ritz1} and the triangle inequality we have
\begin{align*}
& \max_{1 \leq k \leq N_\tau}  \big (\| \vp_h^k - \vp^k \|_{H^1}^2 + \| \theta_h^k - \theta^k \|_{H^1}^2 \big ) \\
& \quad  \leq C(h^4 + \tau^2) + \max_{1 \leq k \leq N_\tau}  \big (\| \vp^k - R_h \vp^k \|_{H^1}^2 + \| \theta^k - R_h \theta^k \|_{H^1}^2 \big ) \leq C(h^2 + \tau^2),
\end{align*}
as well as
\begin{align*}
& \max_{1 \leq k \leq N_\tau} \big ( \| \vp_h^k - \vp^k \|^2 + \| \theta_h^k - \theta^k \|^2 + |q_h^k - q^k|^2 \big )\\
& \quad  \leq C(h^4 + \tau^2) + \max_{1 \leq k \leq N_\tau}  \big (\| \vp^k - R_h \vp^k \|^2 + \| \theta^k - R_h \theta^k \|^2 \big ) \leq C(h^4 + \tau^2).
\end{align*}
\end{proof}

\begin{remark}
We are not able to derive estimates for the error $\bu - \bu_h$ due to the lack of $H^2$-spatial regularity for the displacement $\bu$. This is despite \eqref{model:3} being a linear elasticity system with homogeneous Dirichlet conditions, as the spatially varying elasticity tensor $\CC(\vp)$ complicates the regularity arguments for elliptic systems.
\end{remark}

\section{Numerical simulations}\label{sec:simulation}

In this section, we provide numerical simulations of the model \eqref{model} in a square domain $\Omega = (0,1)^2$ discretized into uniform meshes of size $h$. 
All simulations are performed until $T=1.0$, employing a uniform time step size $\tau$.
For the readers' convenience, we state the formulae for the nonlinear functions and relevant parameters below:
\[
  \begin{aligned}
    W(\vp)             & = \frac{1}{4}(\vp^2 - 1)^2, \quad W'(\vp) = \vp(\vp^2 - 1),                                                                                                                                                \\
    P(\vp)             & = \frac{1}{2}(1-\vp), \quad p(\vp) = P'(\vp)=-\frac12,                                                                                                                                                     \\
    \CC(\vp)           & = \left[(1 - k(\vp))\kappa + k(\vp)\right] \CC^{(1)}, \quad k(\vp)=\begin{cases}
                                                                                                    0                                                       & \text{if } -1 \leq \vp \leq \vp_{\mathrm{gel}}, \\
                                                                                                    \frac{\vp - \vp_{\mathrm{gel}}}{1 - \vp_{\mathrm{gel}}} & \text{if } \vp_{\mathrm{gel}} \leq \vp \leq 1,
                                                                                                  \end{cases} \\
    \CC^{(1)}_{ijmn}   & = \frac{E\nu}{(1+\nu)(1-2\nu)}\delta_{ij} \delta_{mn} + \frac{E}{2(1+\nu)}(\delta_{im} \delta_{jn} + \delta_{in} \delta_{jm}),                                                                             \\
    \mathcal{E}_c(\vp) & = m(\vp)\mathbb{I}, \quad m(\vp)=\zeta (1-P(\vp)) = \frac{\zeta}{2}(1+\vp).
  \end{aligned}
\]
Our implementation is facilitated by the Python packages NumPy and SciPy, and a sample of the codes including the experiment settings can be found on GitHub\footnote{\url{https://github.com/Laphet/sav-stereolithography}}.

We first validate the convergence rates stated in Theorem \ref{thm:errest}. As it is difficult to obtain the exact solutions of the model \eqref{model} due to the nonlinearities and the complex coupling between the phase field, temperature, and elastic variables, we thus introduce source terms to \eqref{model:1}-\eqref{model:3} such that
\begin{equation}\label{eq:refer solutions}
  \begin{alignedat}{2}
    \vp(x, y, t) & =\cos t \cos(2\pi x)\cos(\pi y), \\ 
    \theta(x, y, t) & =\sin t \cos(\pi x)\cos(2 \pi y), \\
    \bu(x, y, t) & = \begin{bmatrix} \sin t \sin(\pi x)\sin(2\pi y) \\ \cos t \sin(2 \pi x) \sin(\pi y) \end{bmatrix},
  \end{alignedat}
\end{equation}
is our exact solution that fulfils the boundary conditions. The model parameters are set as
\[
  \begin{alignedat}{2}
    \alpha = 1.0,\quad \lambda = 1.0, \quad \eps = 0.1, \quad \gamma = 1.0, \quad \theta_c = 0.0, \quad \delta = 1.2, \\
    \kappa = 0.01, \quad \vp_{\mathrm{gel}} = 0.5, \quad E = 1.0, \quad \nu = 0.3, \quad \zeta = 1.0, \quad \beta = 0.5.
  \end{alignedat}
\]
We conducted two groups of experiments to validate the convergence rates with respect to the mesh size $h$ and the time step size $\tau$, respectively. In the first group, we set $\tau$ to the values of $1/100$ and $1/200$, while varying the mesh size $h$ from $1/8$, $1/16$, $1/32$, $1/64$, and $1/128$. In the second group, we set $h$ to the values of $1/100$ and $1/200$, while varying the time step size $\tau$ from $1/10$, $1/20$, $1/40$, $1/80$, and $1/160$. To simplify the implementation, we calculated the errors by comparing the numerical solutions with the nodal interpolations of \eqref{eq:refer solutions} in difference norms, rather than the exact solutions. The results are shown in Figures \ref{fig:convergence rate phi}, \ref{fig:convergence rate theta}, and \ref{fig:convergence rate u}, where $\| \cdot\|_0$ denotes the $l^\infty(L^2)$ norm and $|\cdot|_1$ represents the $l^\infty(H^1)$ seminorm. 

We observe from Figure \ref{fig:convergence rate phi} that the convergence rates of the phase field are well matched with the theoretical results, i.e., $\mathcal{O}(h^2+\tau)$ for the $l^\infty(L^2)$ norm and $\mathcal{O}(h+\tau)$ for the $l^\infty(H^1)$ seminorm. From the subplot (a) in Figure \ref{fig:convergence rate theta}, we notice a $\mathcal{O}(h^2)$ superconvergence phenomenon for the temperature field, which may attributed to the structure of \eqref{model:2} that only consists of a homogeneous elliptic operator and uniform meshes are utilized. However, the theoretical convergence rate $\mathcal{O}(h^2+\tau)$ is still satisfied. Even if we did not derive a theoretical convergence rate estimate for the displacement field, nevertheless, based on Figure \ref{fig:convergence rate u}, we can conjecture that the optimal rates in $l^\infty(L^2)$ and $l^\infty(H^1)$ are both $\mathcal{O}(h+\tau)$. Here, we cannot see an optimal rate $\mathcal{O}(h^2+\tau)$ for $l^\infty(L^2)$ as in the phase field and temperature variables, the reason could be attributed to the presence of the nonsmooth elasticity tensor $\CC(\vp)$ in the elasticity equation \eqref{model:3}.

\begin{figure}[!ht]
  \centering
  \includegraphics[width=\textwidth]{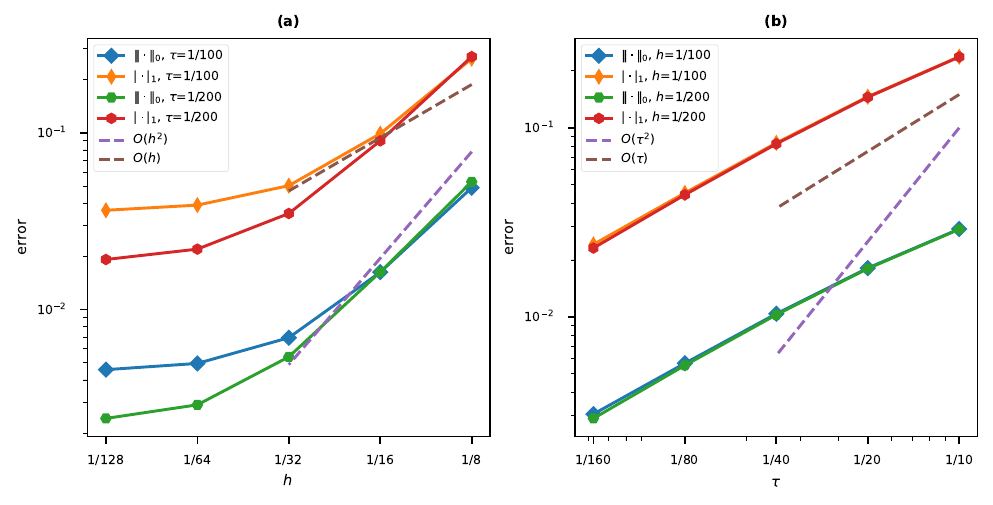}
  \caption{Errors of the phase field $\vp$: \textbf{(a)} the x-axis corresponds to the mesh size $h$; \textbf{(b)} the x-axis corresponds to the time step size $\tau$.}\label{fig:convergence rate phi}
\end{figure}

\begin{figure}[!ht]
  \centering
  \includegraphics[width=\textwidth]{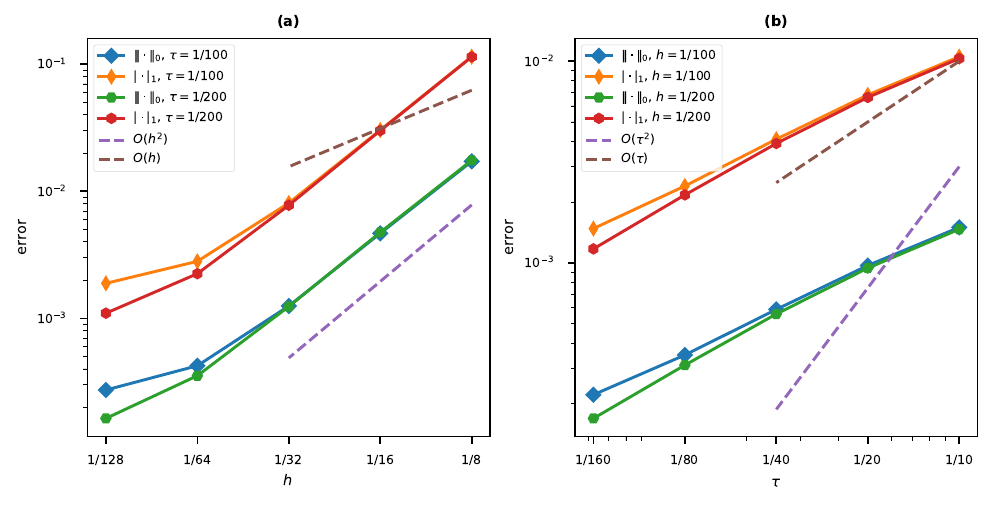}
  \caption{Errors of the temperature field $\theta$: \textbf{(a)} the x-axis corresponds to the mesh size $h$; \textbf{(b)} the x-axis corresponds to the time step size $\tau$.}\label{fig:convergence rate theta}
\end{figure}

\begin{figure}[!ht]
  \centering
  \includegraphics[width=\textwidth]{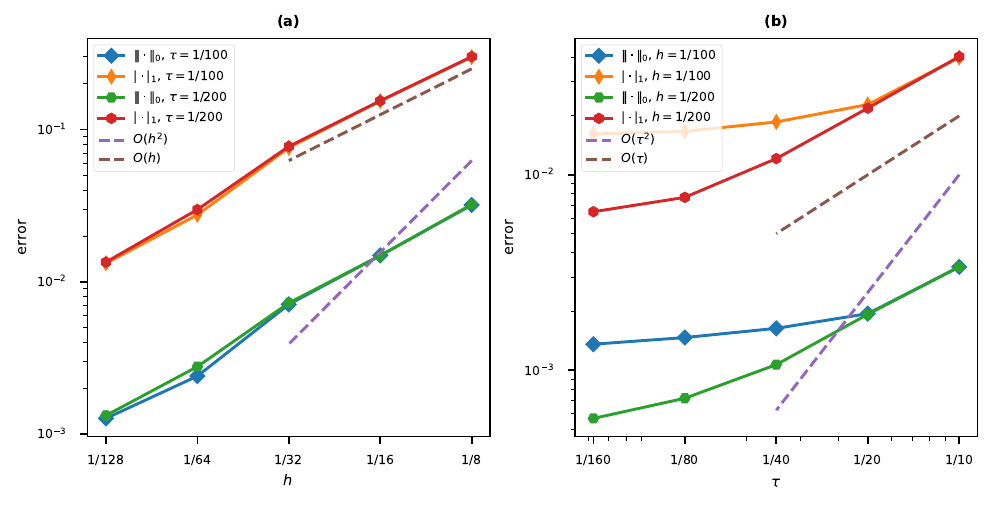}
  \caption{Errors of the displacement field $\bu$: \textbf{(a)} the x-axis corresponds to the mesh size $h$; \textbf{(b)} the x-axis corresponds to the time step size $\tau$.}\label{fig:convergence rate u}
\end{figure}

We then simulate the gel-sol convention process induced by laser irradiation that serves as an external heat source, which enters the temperature equation \eqref{model:2} as
\[
  I(x, y, t)=I_{\mathrm{m}} \exp\left (-\frac{(x-x^*(t))^2 + (y-y^*(t))^2}{w_0^2} \right ),
\]
with the laser center $(x^*(t), y^*(t))$ moves along a path to reflect the typical 3D printing process, $I_{\mathrm{m}}$ represents the maximum intensity, and $w_0$ controls the width of the heat source. In the following simulations we set the parameters for the heat source as
\[
  I_{\mathrm{m}} = 4.0\times 10^4, \quad w_0 = 0.015.
\]
As for the other parameters of the model, we choose the following values
\[
  \begin{aligned}
    \alpha = 0.5, \quad \lambda = 1.0, \quad \eps = 5.0\times 10^{-3}, \quad \gamma = 4.0\times 10^2, \quad \theta_c = 1.0, \quad \delta = 1.0\times 10^2, \\
    \kappa = 10^{-6}, \quad \vp_{\mathrm{gel}} = 0.5, \quad E = 10^4, \quad \nu = 0.35, \quad \zeta = 10^3, \quad \beta = 5.0\times 10^2.
  \end{aligned}
\]
For the initial conditions, we set
\[
  \phi(x, y, 0)=-1.0, \quad \theta(x, y, 0)=0.0,
\]
implying that the material is initially in a fully sol phase and in thermal equilibrium. We take the mesh size $h=1/400$ to resolve the interfacial layer between the sol and gel phases, and the time step size $\tau=1/100$.

In the first simulation, we consider a fixed heat source located at the center of the domain, i.e., $(x^*(t), y^*(t))\equiv(0.5, 0.5)$ for all $t \in [0,1.0]$.
The plots for the phase field $\vp$, the temperature field $\theta$, and the displacement field $\bu$ at $t=0.01$, $0.05$, $0.10$, and $0.20$ are displayed in Figures \ref{fig:fixed-heat-source-phi}, \ref{fig:fixed-heat-source-theta}, and \ref{fig:fixed-heat-source-u}, respectively.
We observe that the phase field $\vp$ evolves from the sol phase $\{\vp = -1\}$ to the gel phase $\{\vp = 1\}$, and the temperature field $\theta$ increases around the heat source.
The displacement field $\bu$ is also affected by the temperature field, and the material is deformed due to the thermal expansion and shinkage strains.
Notably, the deformation is more pronounced at the boundary of the heat source, aligning with the physical phenomenon of laser-induced curing. Ultimately, the physical fields reach their steady states seemingly around $t=0.20$.

\begin{figure}[!ht]
  \centering
  \begin{subfigure}[b]{0.24\textwidth}
    \includegraphics[width=\textwidth]{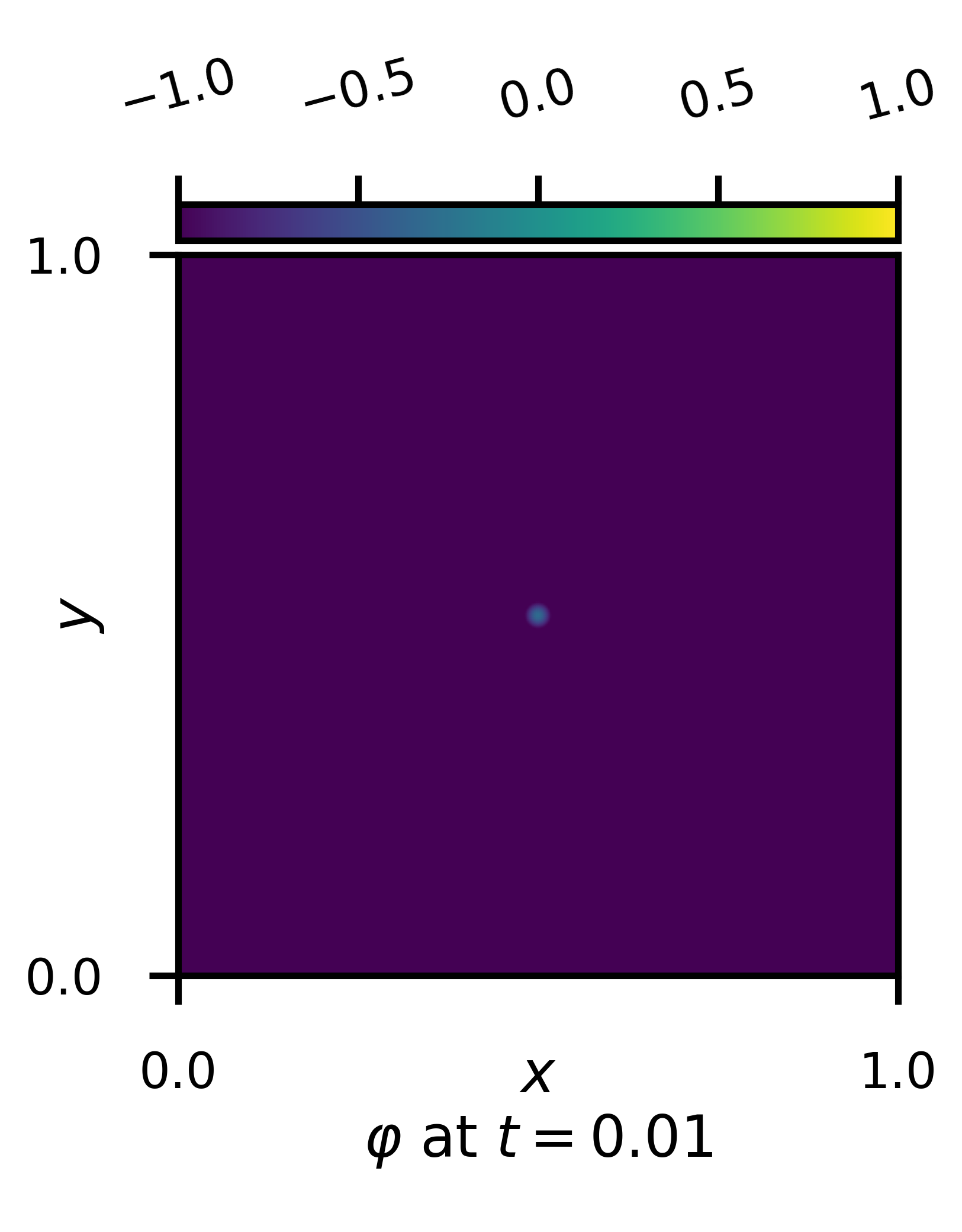}
  \end{subfigure}
  \begin{subfigure}[b]{0.24\textwidth}
    \includegraphics[width=\textwidth]{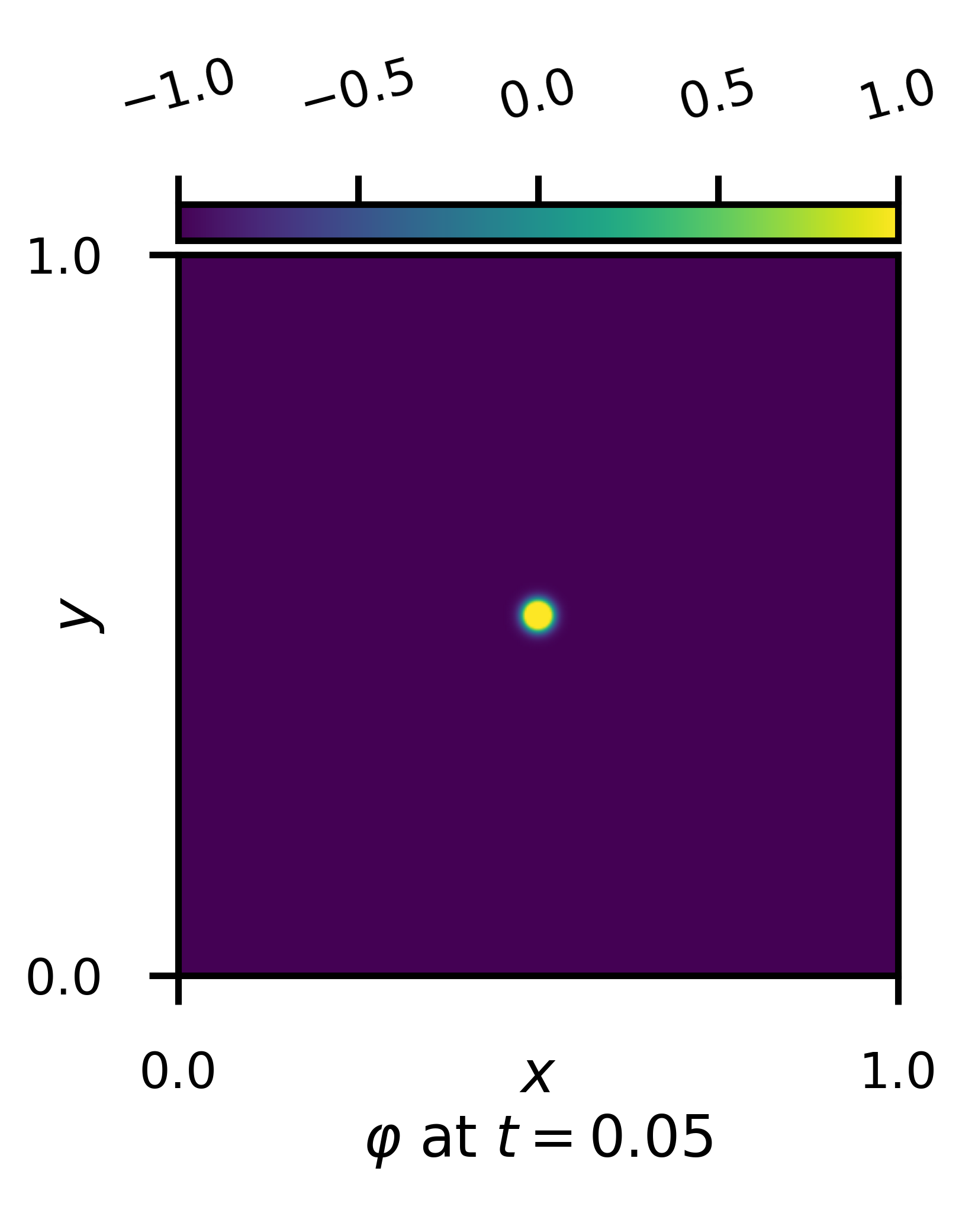}
  \end{subfigure}
  \begin{subfigure}[b]{0.24\textwidth}
    \includegraphics[width=\textwidth]{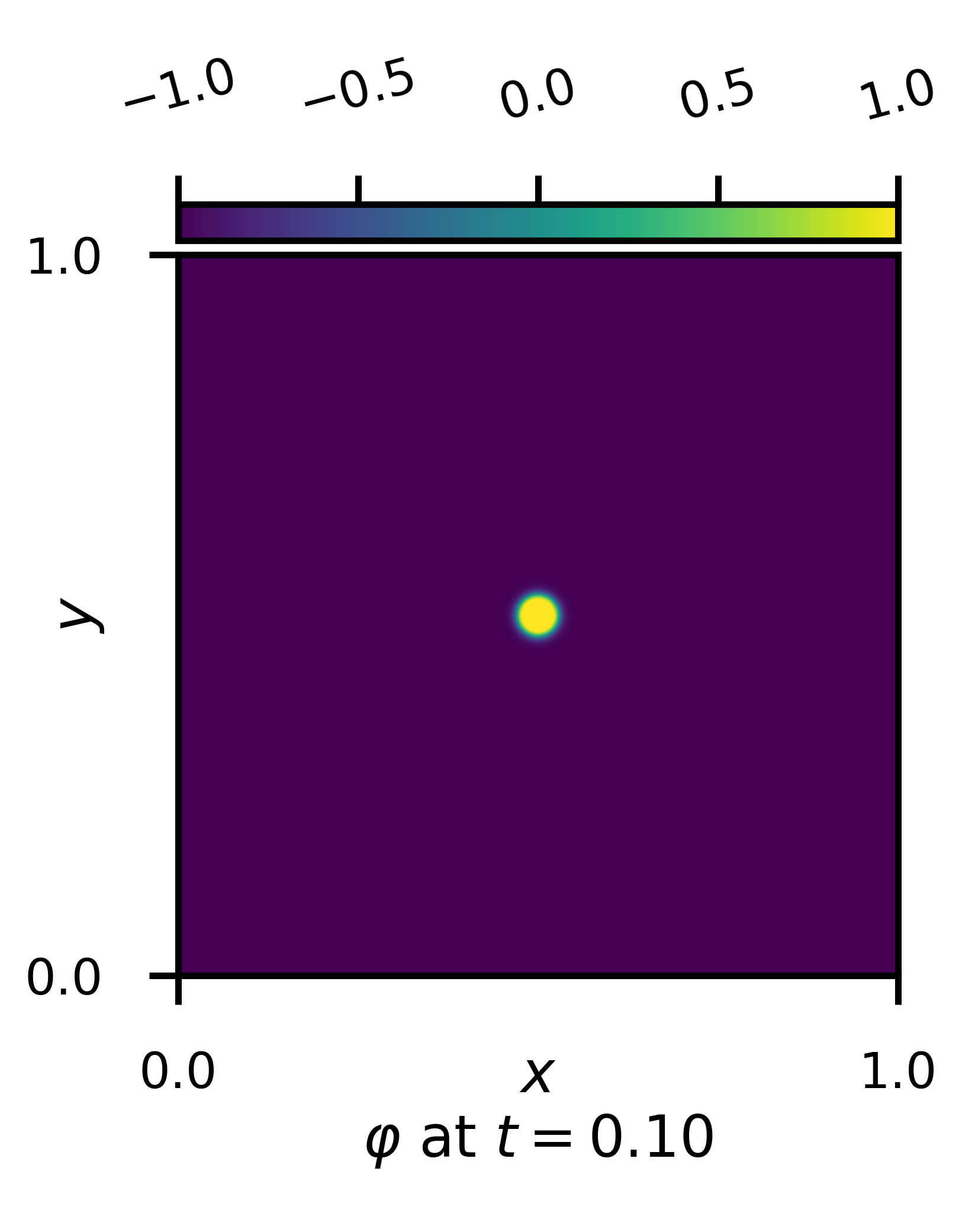}
  \end{subfigure}
  \begin{subfigure}[b]{0.24\textwidth}
    \includegraphics[width=\textwidth]{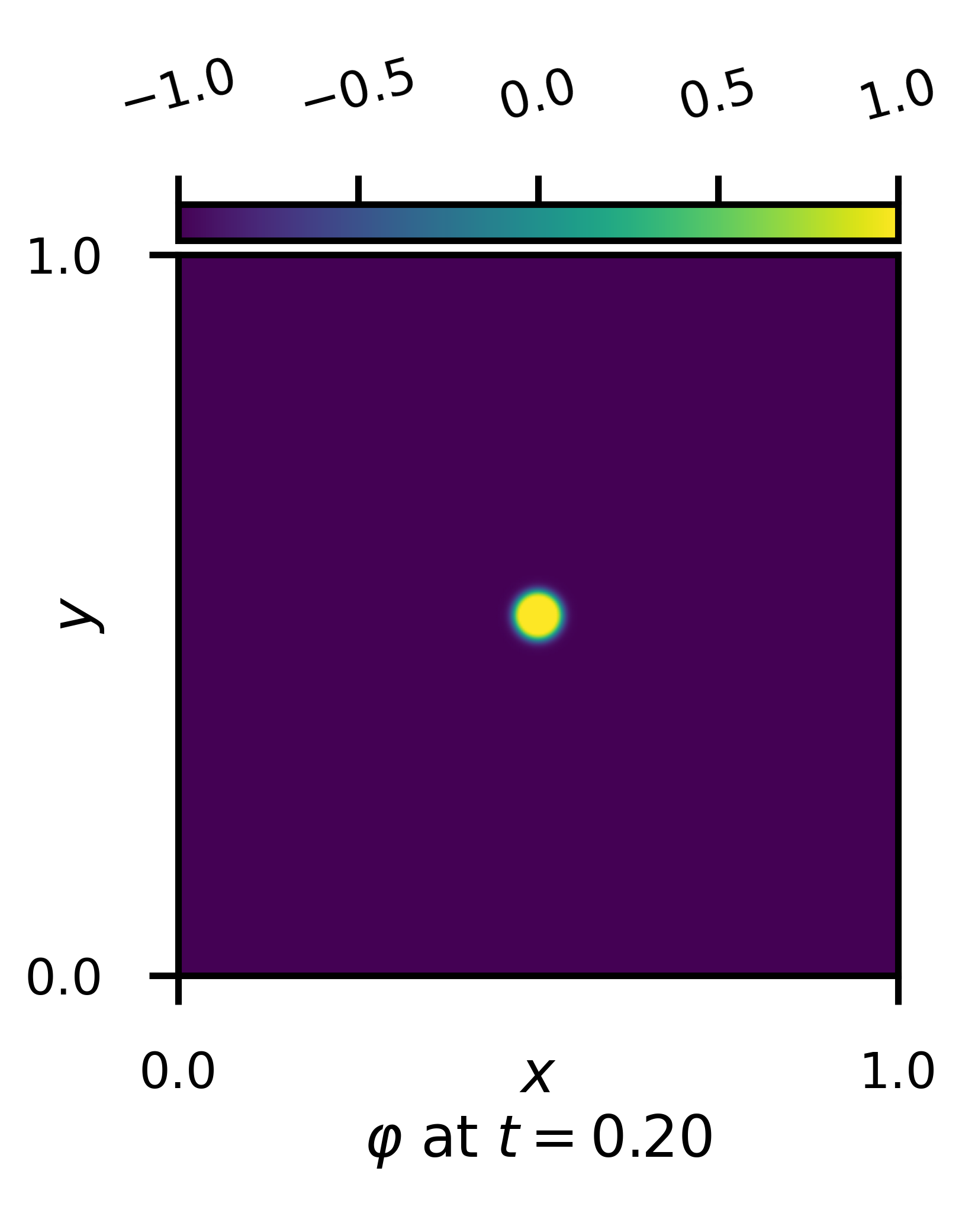}
  \end{subfigure}
  \caption{In the fixed source heat source simulation, the four subplots display the phase field $\vp$ at $t=0.01$, $0.05$, $0.10$, and $0.20$.}\label{fig:fixed-heat-source-phi}
\end{figure}

\begin{figure}[!ht]
  \centering
  \begin{subfigure}[b]{0.24\textwidth}
    \includegraphics[width=\textwidth]{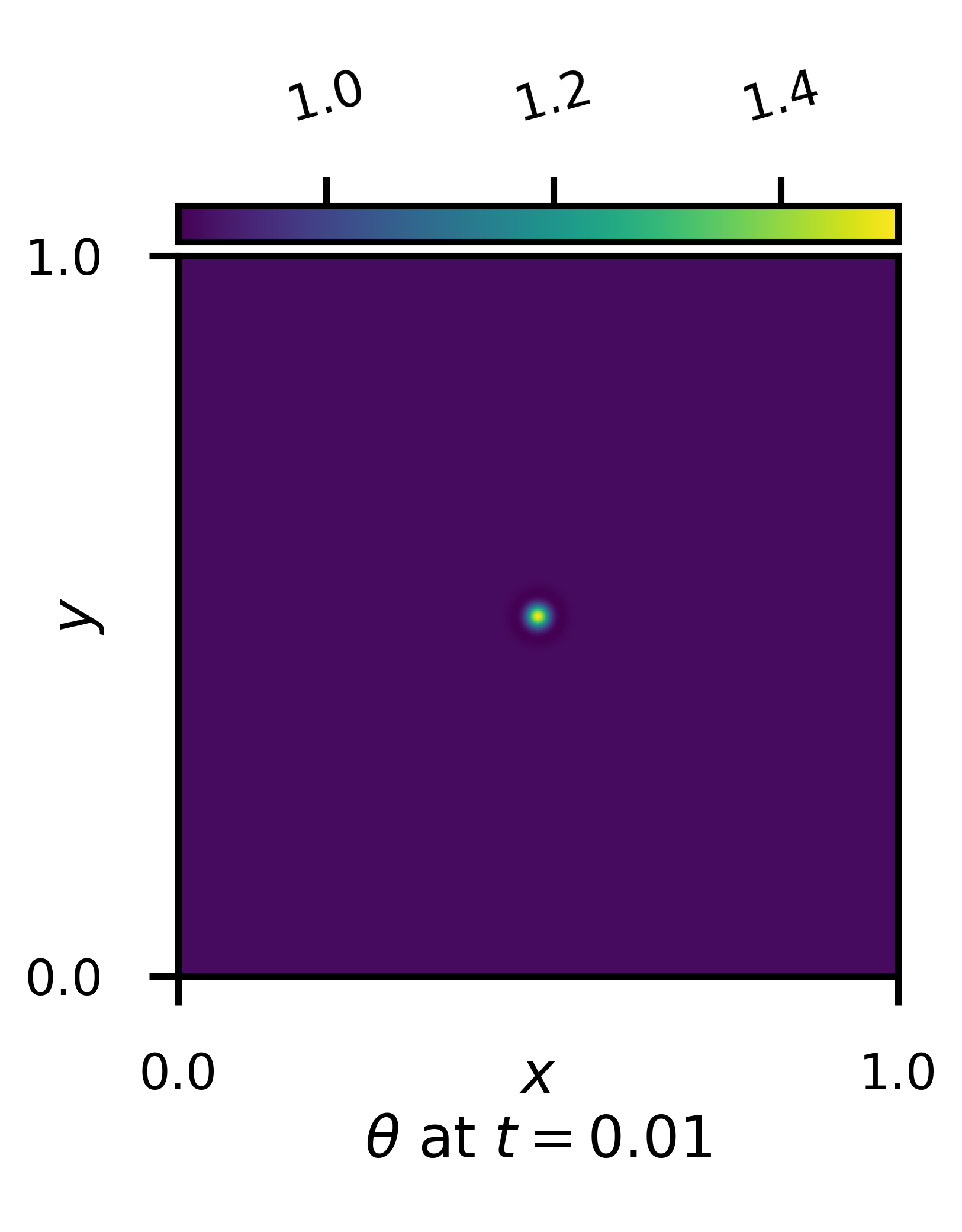}
  \end{subfigure}
  \begin{subfigure}[b]{0.24\textwidth}
    \includegraphics[width=\textwidth]{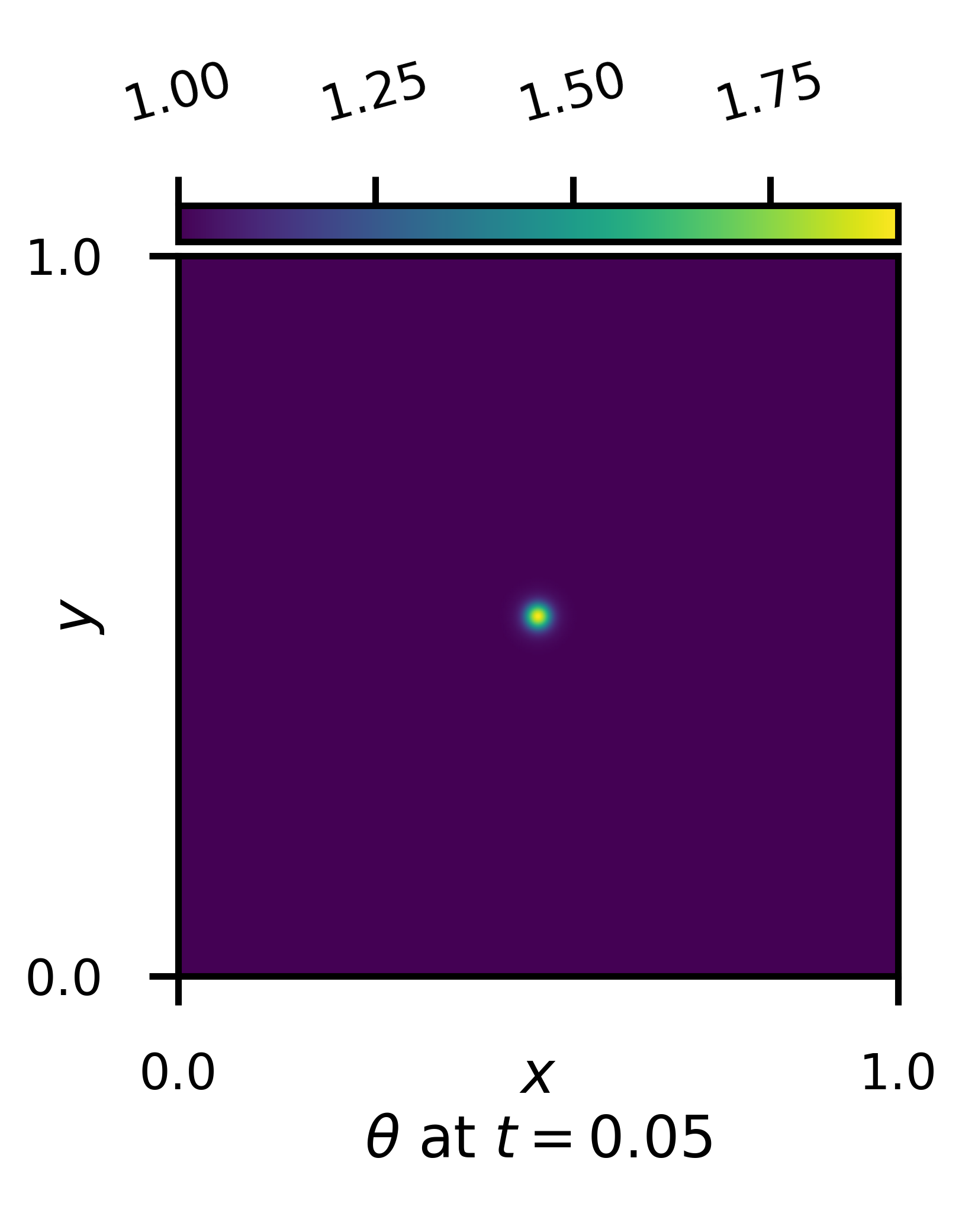}
  \end{subfigure}
  \begin{subfigure}[b]{0.24\textwidth}
    \includegraphics[width=\textwidth]{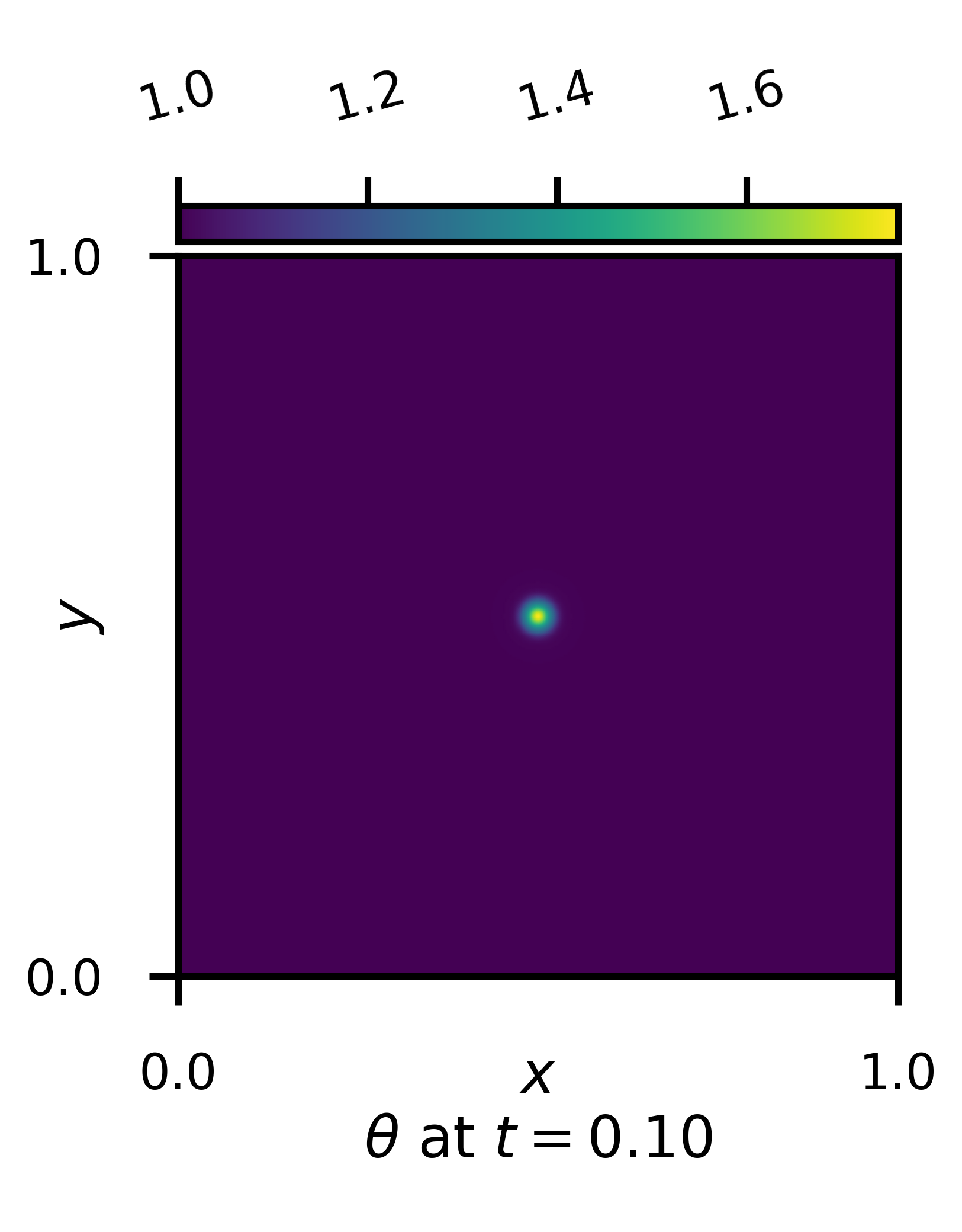}
  \end{subfigure}
  \begin{subfigure}[b]{0.24\textwidth}
    \includegraphics[width=\textwidth]{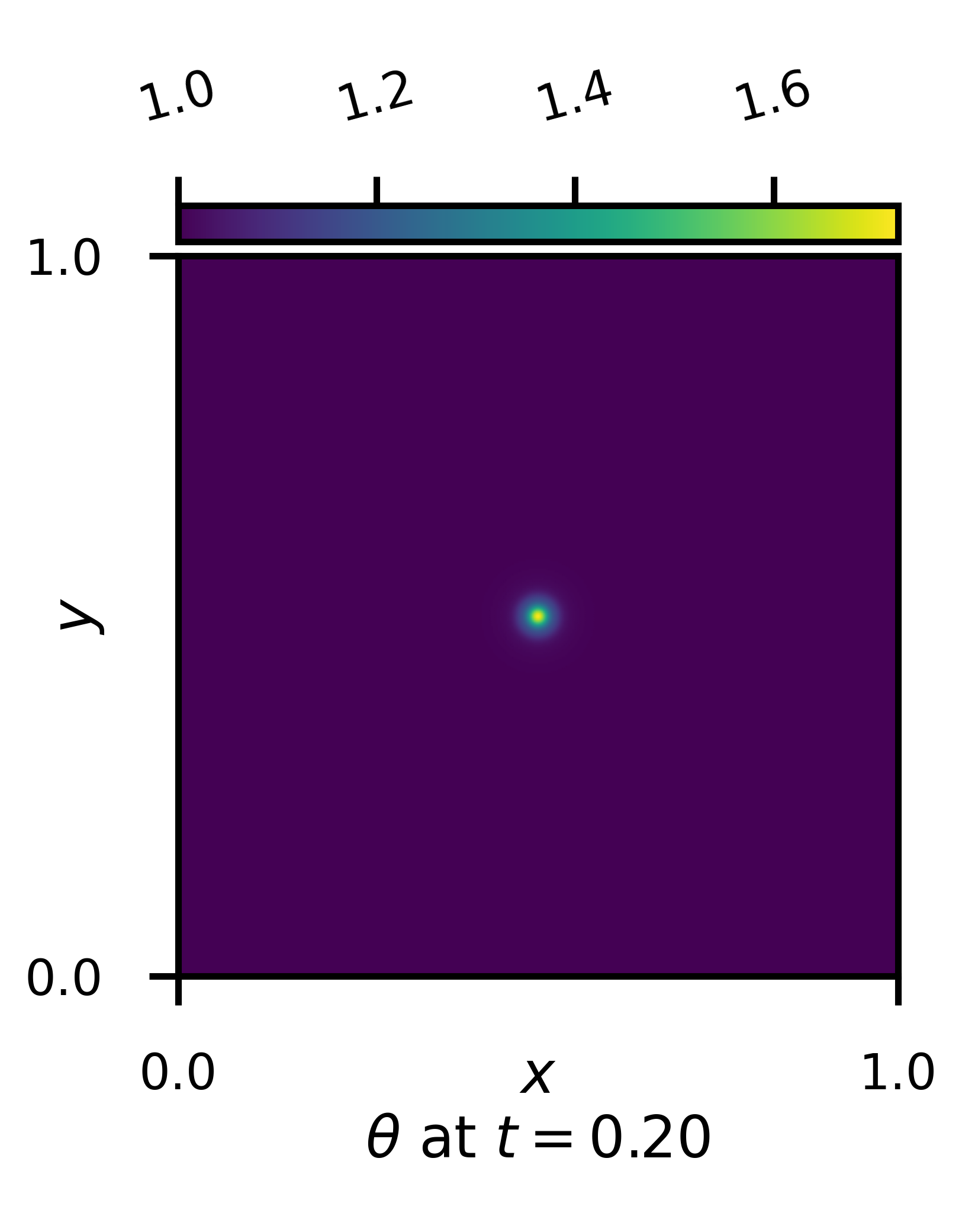}
  \end{subfigure}
  \caption{In the fixed source heat source simulation, the four subplots display the temperature field $\theta$ at $t=0.01$, $0.05$, $0.10$, and $0.20$.}\label{fig:fixed-heat-source-theta}
\end{figure}

\begin{figure}[!ht]
  \centering
  \begin{subfigure}[b]{0.24\textwidth}
    \includegraphics[width=\textwidth]{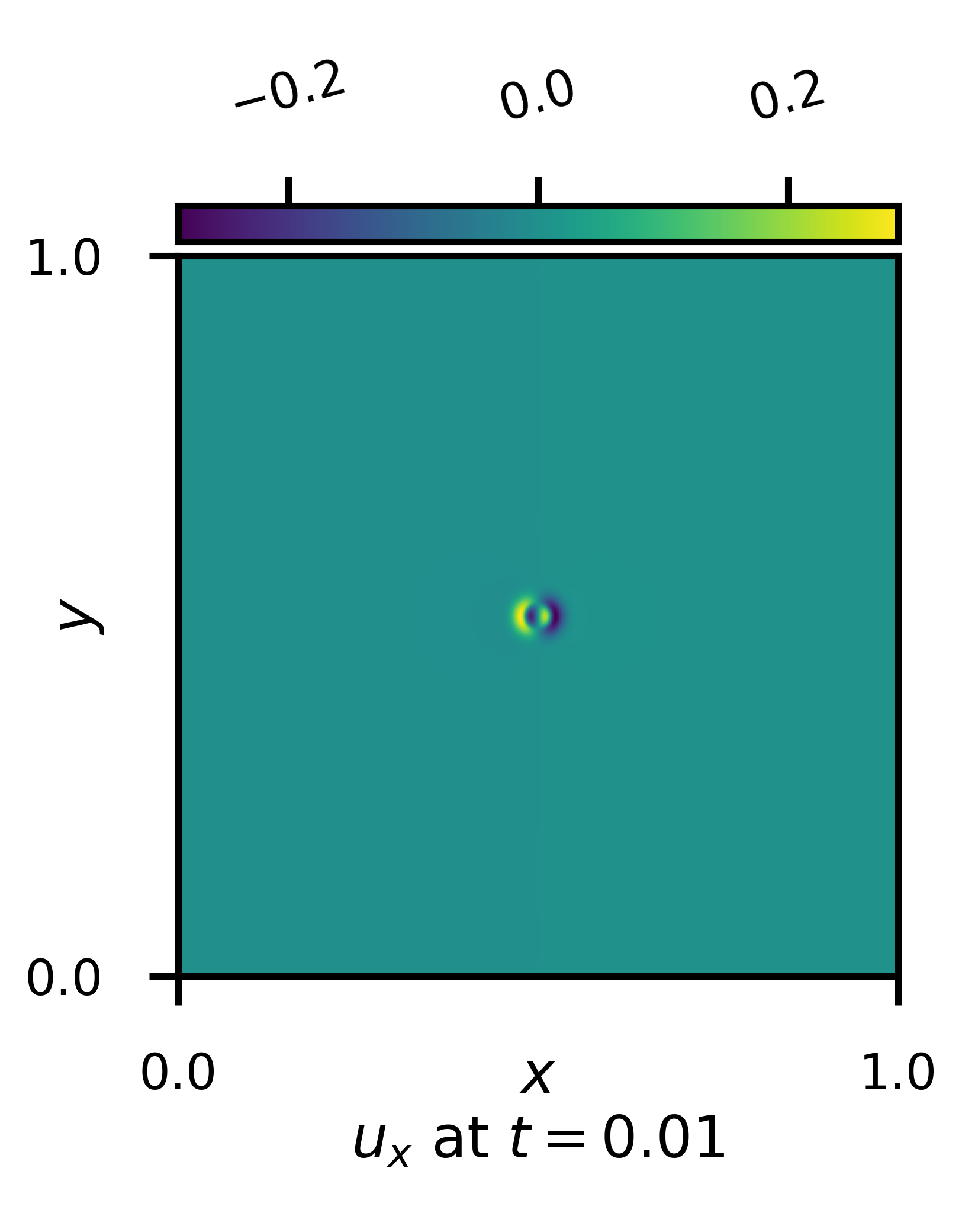}
  \end{subfigure}
  \begin{subfigure}[b]{0.24\textwidth}
    \includegraphics[width=\textwidth]{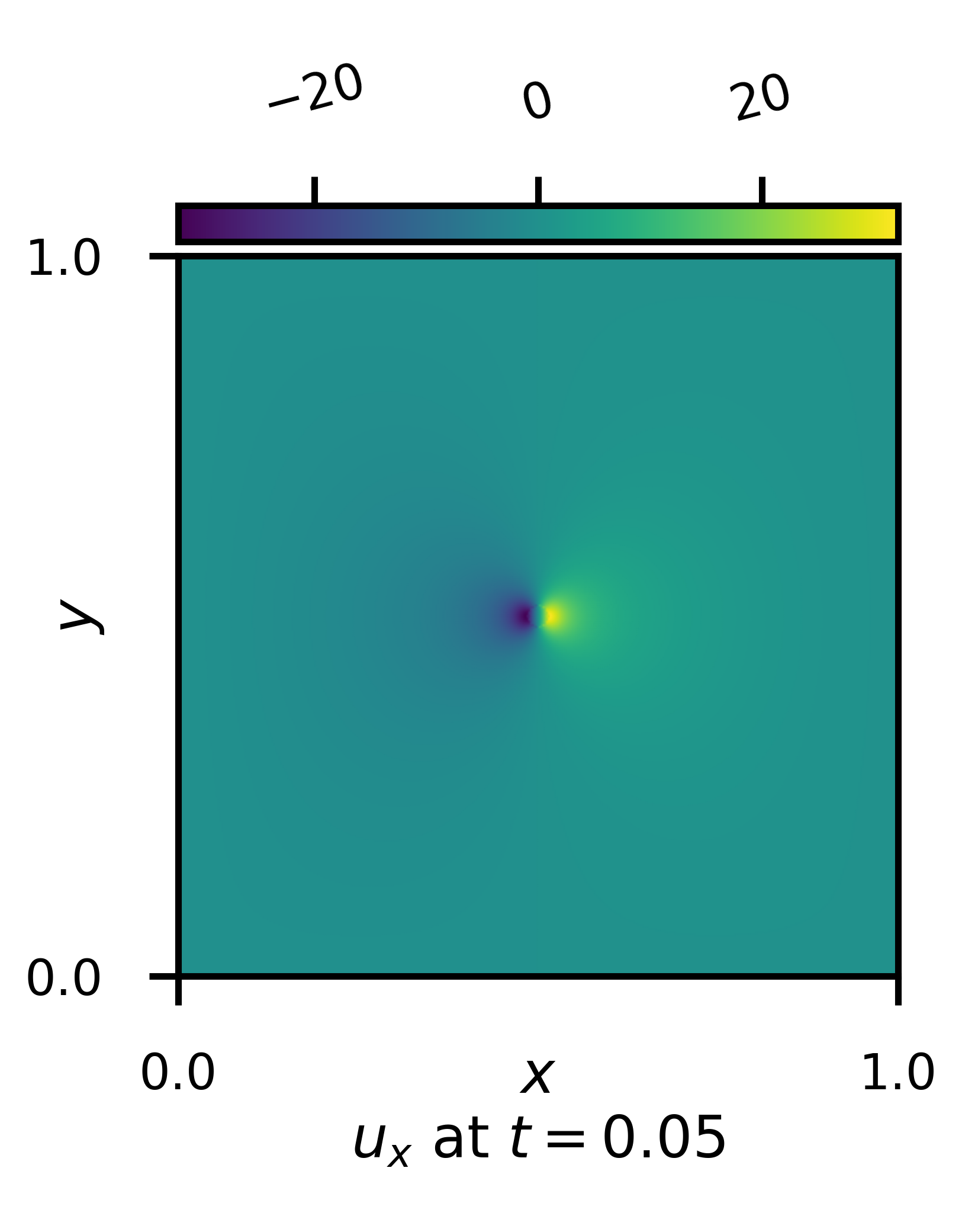}
  \end{subfigure}
  \begin{subfigure}[b]{0.24\textwidth}
    \includegraphics[width=\textwidth]{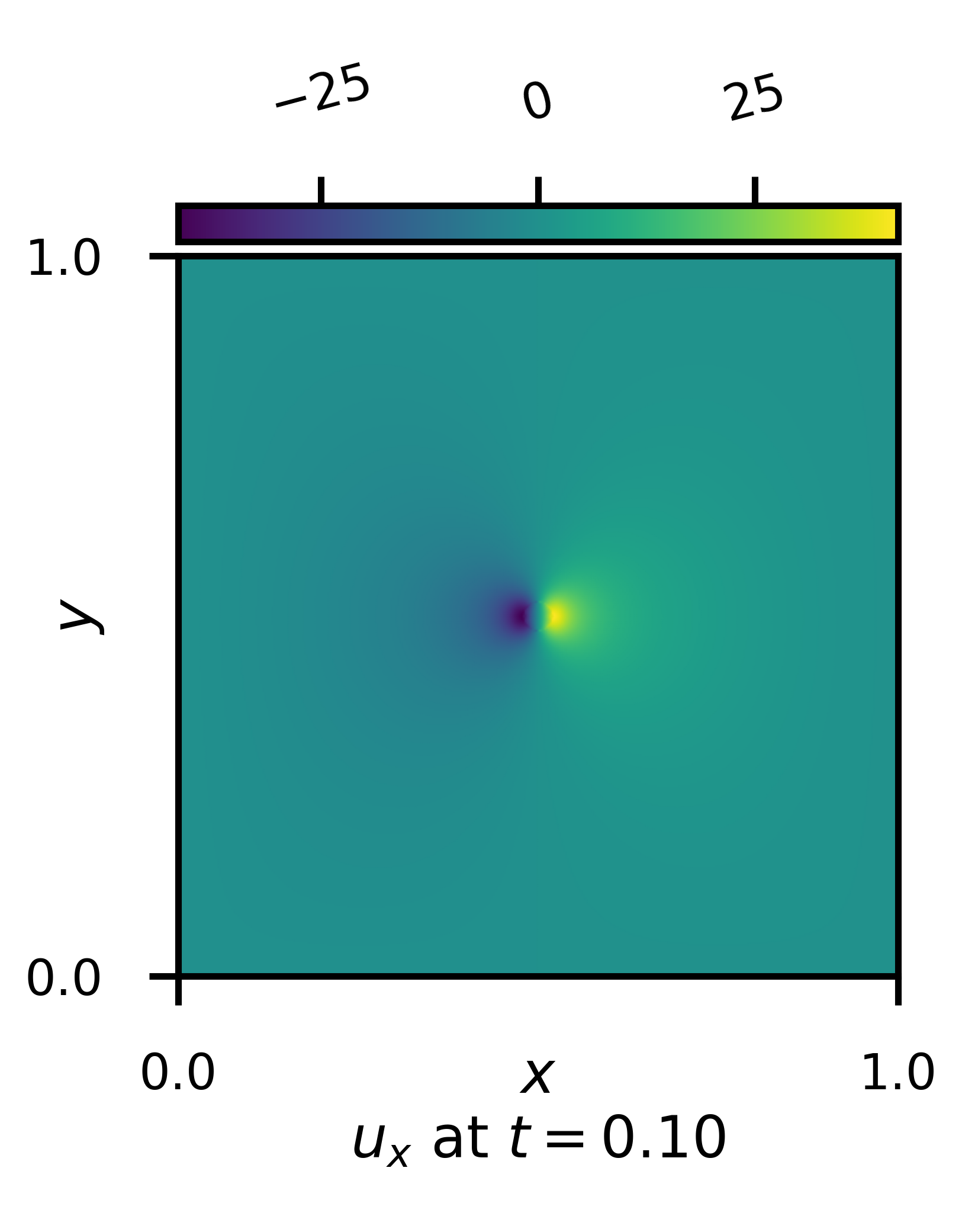}
  \end{subfigure}
  \begin{subfigure}[b]{0.24\textwidth}
    \includegraphics[width=\textwidth]{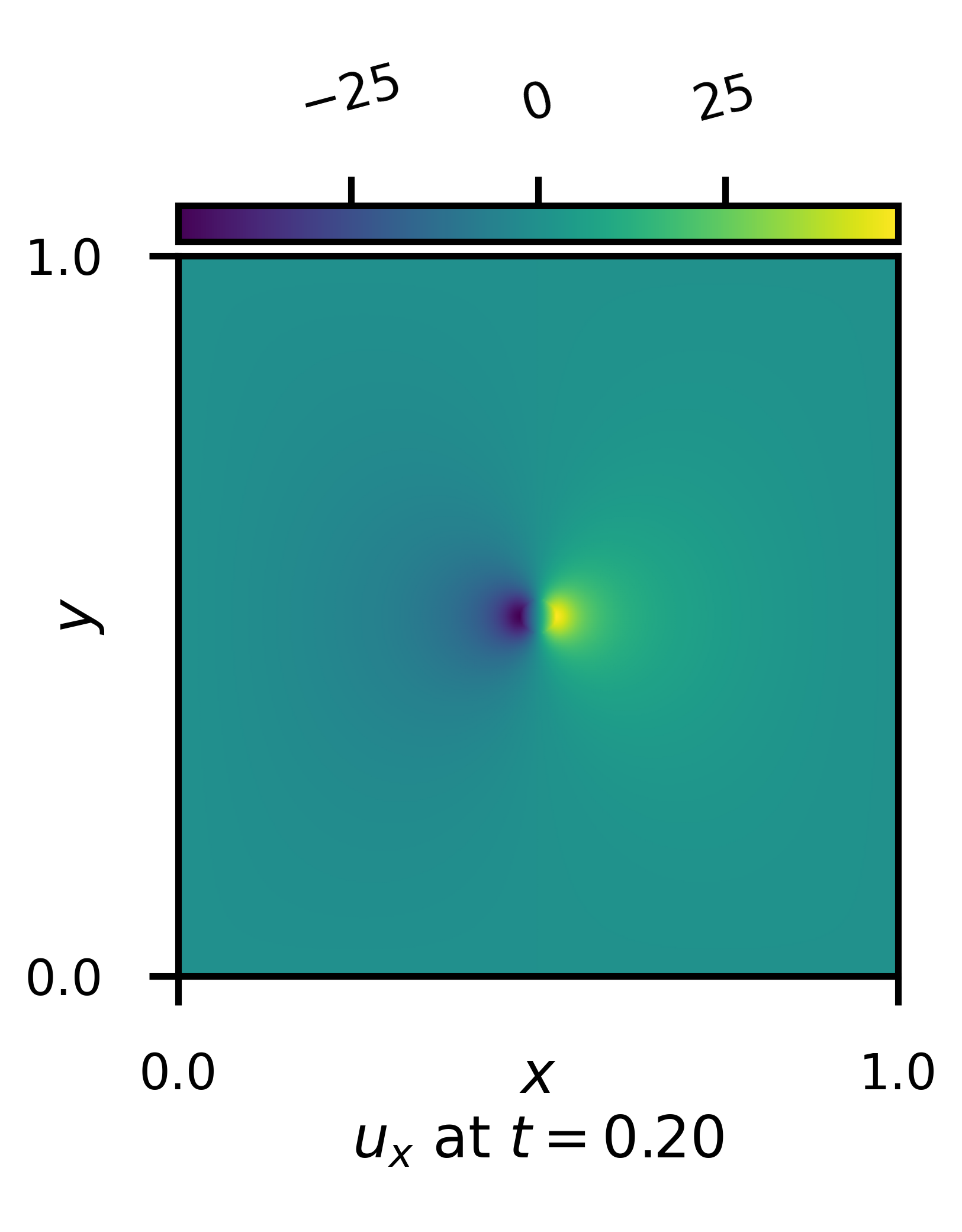}
  \end{subfigure}
  \begin{subfigure}[b]{0.24\textwidth}
    \includegraphics[width=\textwidth]{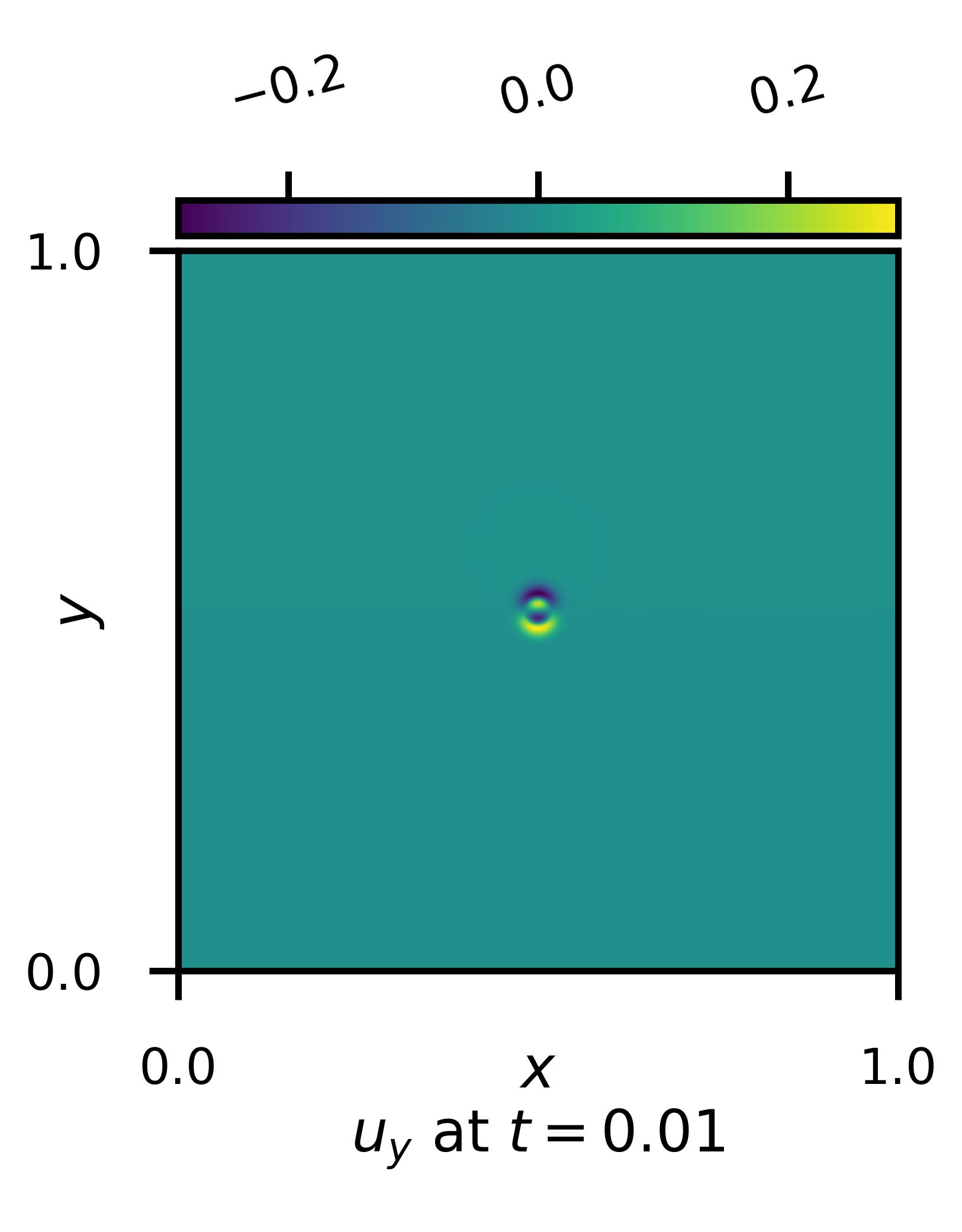}
  \end{subfigure}
  \begin{subfigure}[b]{0.24\textwidth}
    \includegraphics[width=\textwidth]{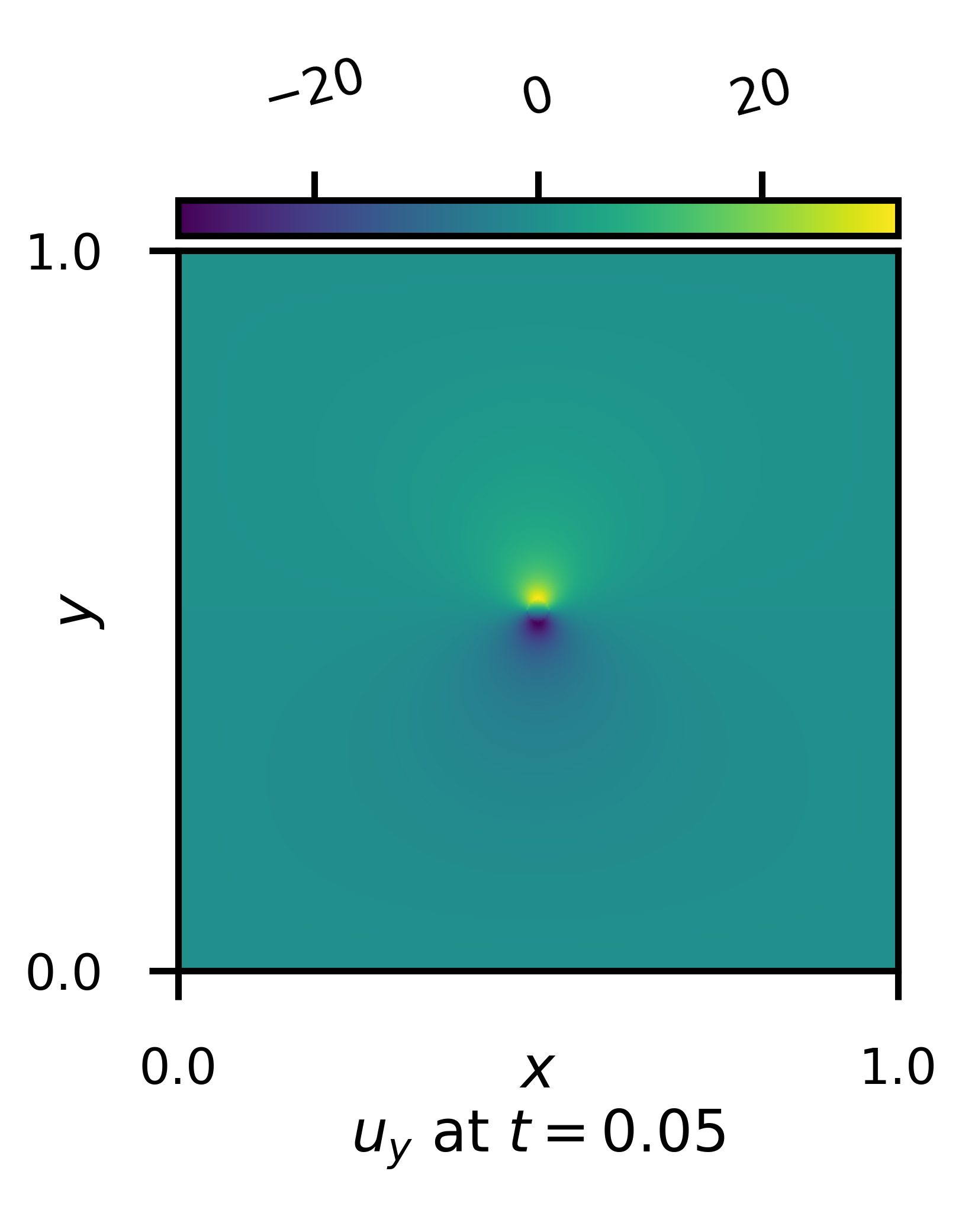}
  \end{subfigure}
  \begin{subfigure}[b]{0.24\textwidth}
    \includegraphics[width=\textwidth]{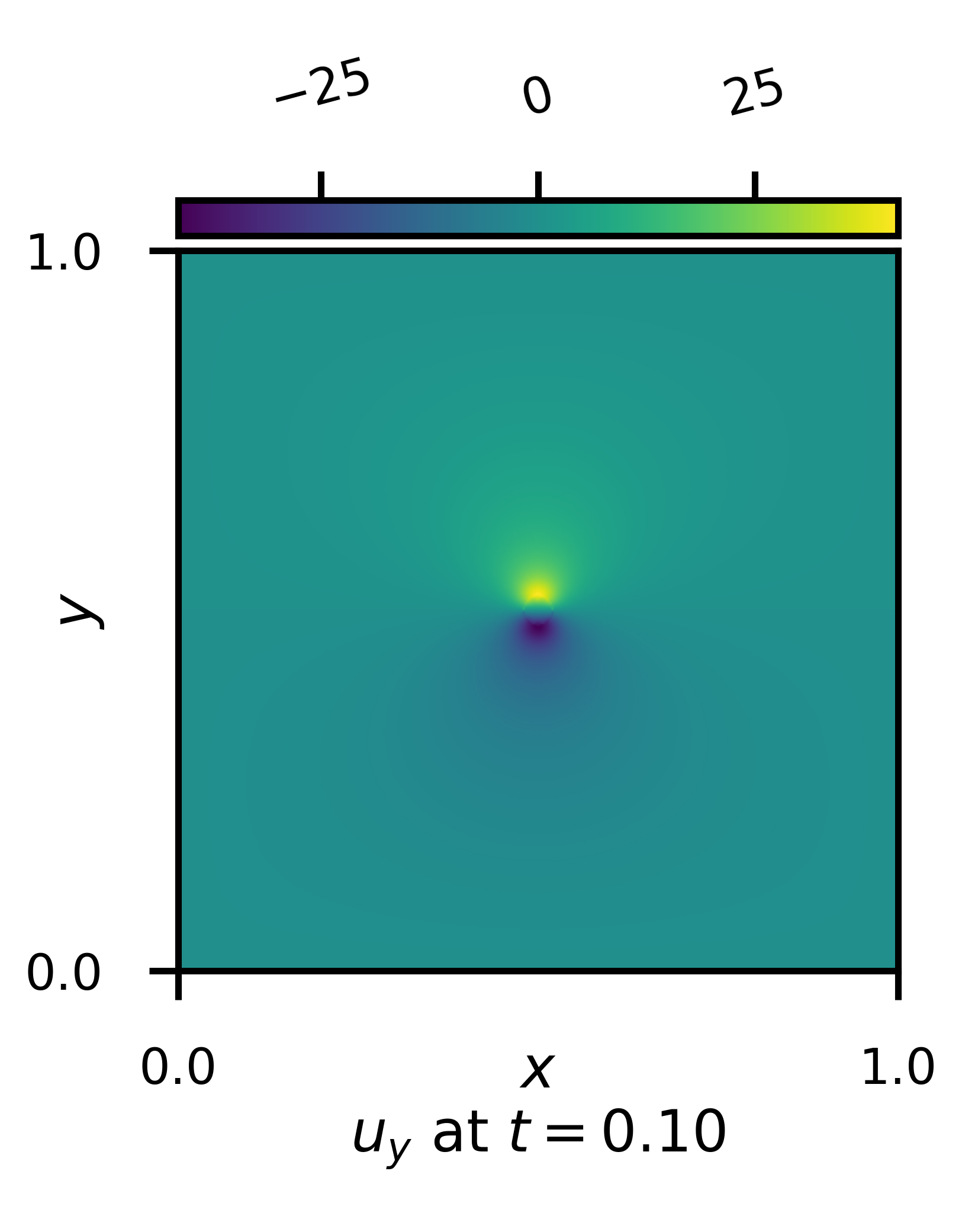}
  \end{subfigure}
  \begin{subfigure}[b]{0.24\textwidth}
    \includegraphics[width=\textwidth]{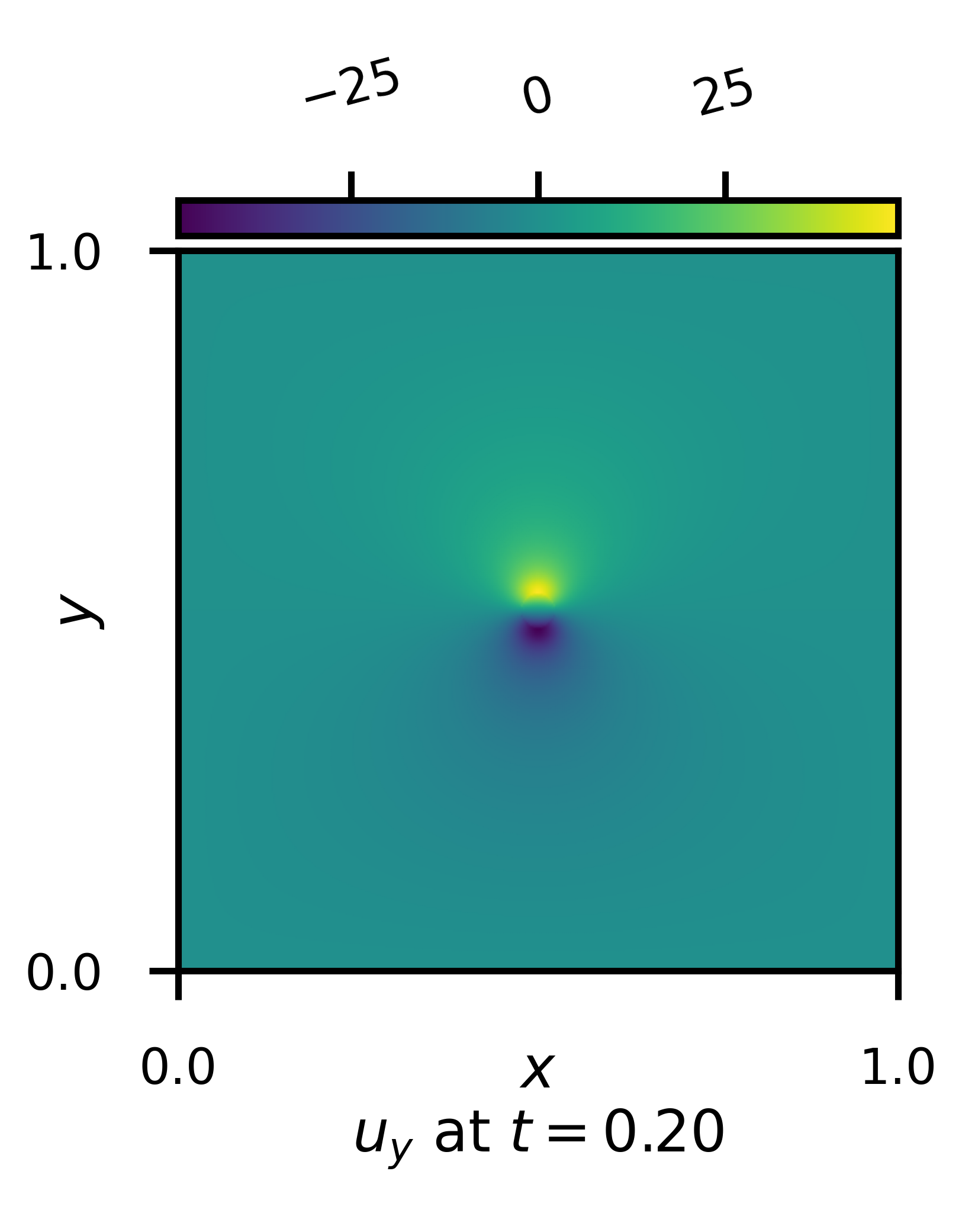}
  \end{subfigure}
  \caption{In the fixed source heat source simulation, the four subplots in the first/second row display the first/second component of the displacement field at $t=0.01$, $0.05$, $0.10$, and $0.20$, respectively.}\label{fig:fixed-heat-source-u}
\end{figure}

In the second simulation, we introduce a moving heat source.
The movement of the heat source is as follows: for $t\in [0, 1/3]$, $(x^*, y^*)$ moves from $(1/4, 5/6)$ to $(1/2, 1/2)$ at a constant speed; for $t\in (1/3, 2/3]$, $(x^*, y^*)$ moves from $(1/2, 1/6)$ to $(1/2, 1/2)$ at a constant speed; and for $t\in (2/3, 1]$, $(x^*, y^*)$ moves from $(3/4, 5/6)$ to $(1/2, 1/2)$ at a constant speed.
The trajectory of the heat source visually resembles the letter ``Y''.
Figures \ref{fig:moving-heat-source-phi}, \ref{fig:moving-heat-source-theta}, and \ref{fig:moving-heat-source-u} display the phase field $\vp$, temperature field $\theta$, and displacement field $\bu$, respectively, at time instances $t=0.34$, $0.50$, $0.67$, and $1.0$.
From Figure \ref{fig:moving-heat-source-phi}, we can observe that as the heat source moves, gel forms along the path of the heat source, and gel cannot return to the sol phase when the heat source moves away, resulting in the gradual formation of the letter ``Y''.
In contrast, the behavior of the temperature field $\theta$, as depicted in Figure \ref{fig:moving-heat-source-theta}, is distinct. We observe an immediate dispersion of $\theta$ once the heat source moves away, and the maximum temperature is consistently located at the heat source.
Similarly, Figure \ref{fig:moving-heat-source-u} demonstrates that mechanical effects concentrate around the gel phase.

\begin{figure}[!ht]
  \centering
  \begin{subfigure}[b]{0.24\textwidth}
    \includegraphics[width=\textwidth]{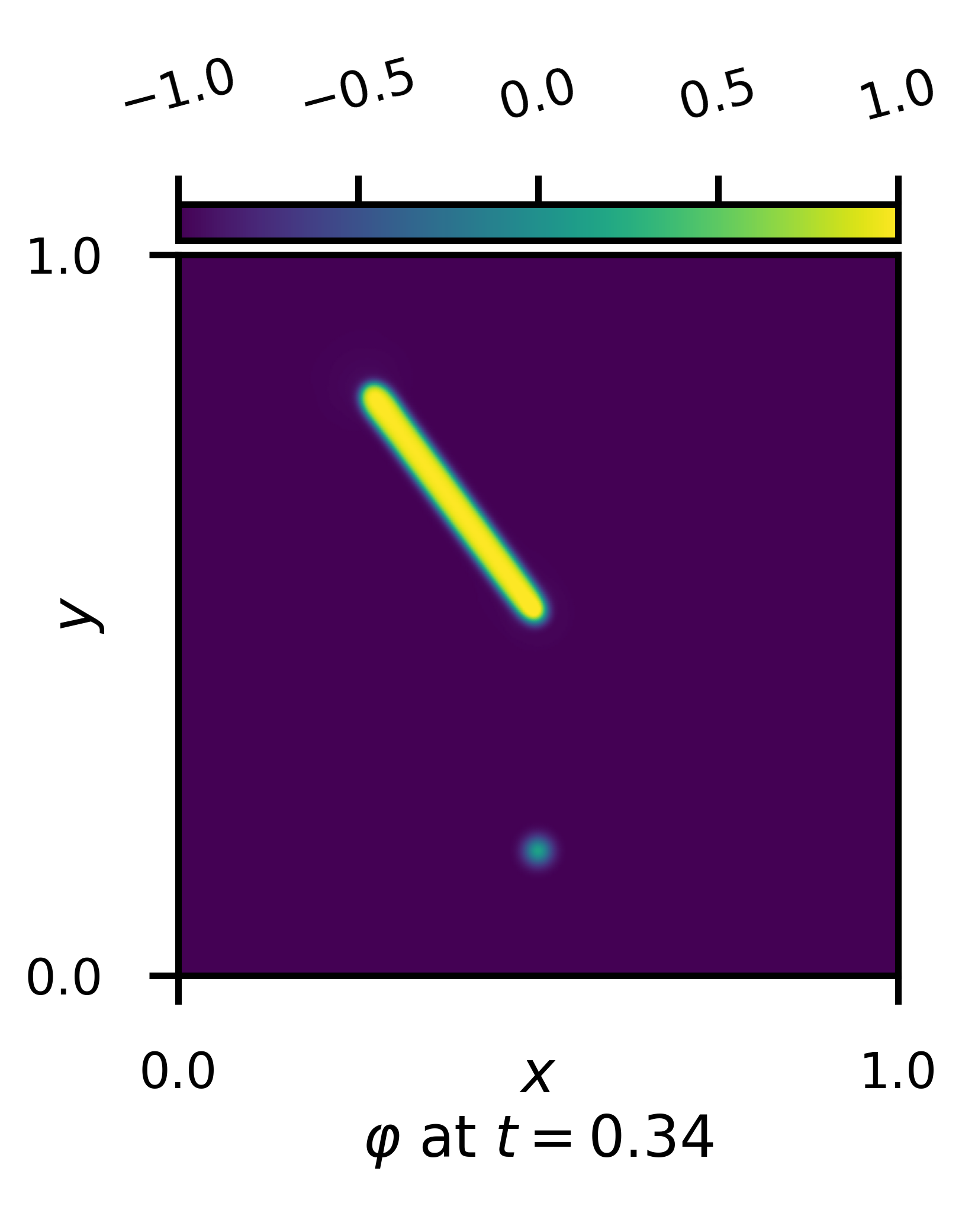}
  \end{subfigure}
  \begin{subfigure}[b]{0.24\textwidth}
    \includegraphics[width=\textwidth]{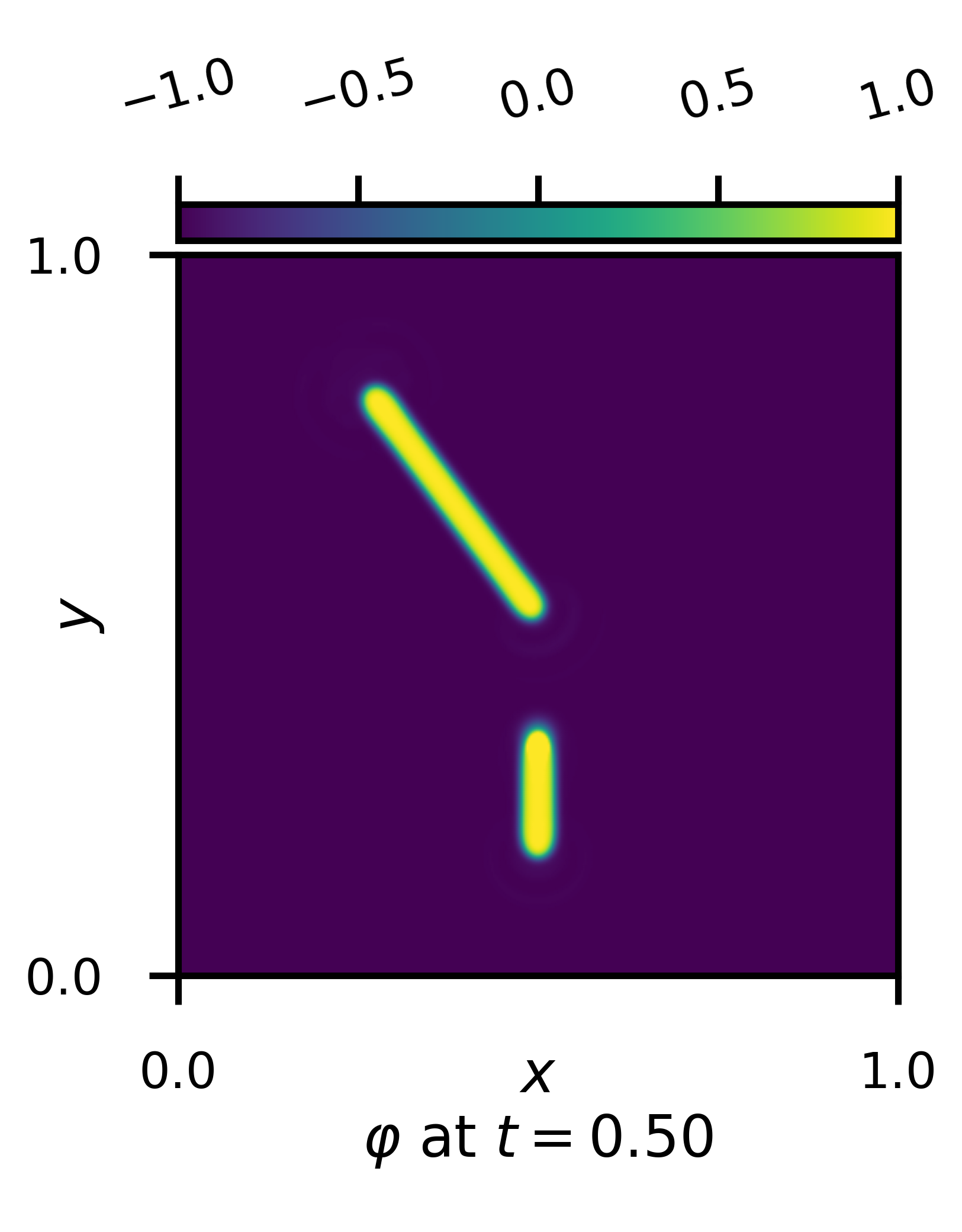}
  \end{subfigure}
  \begin{subfigure}[b]{0.24\textwidth}
    \includegraphics[width=\textwidth]{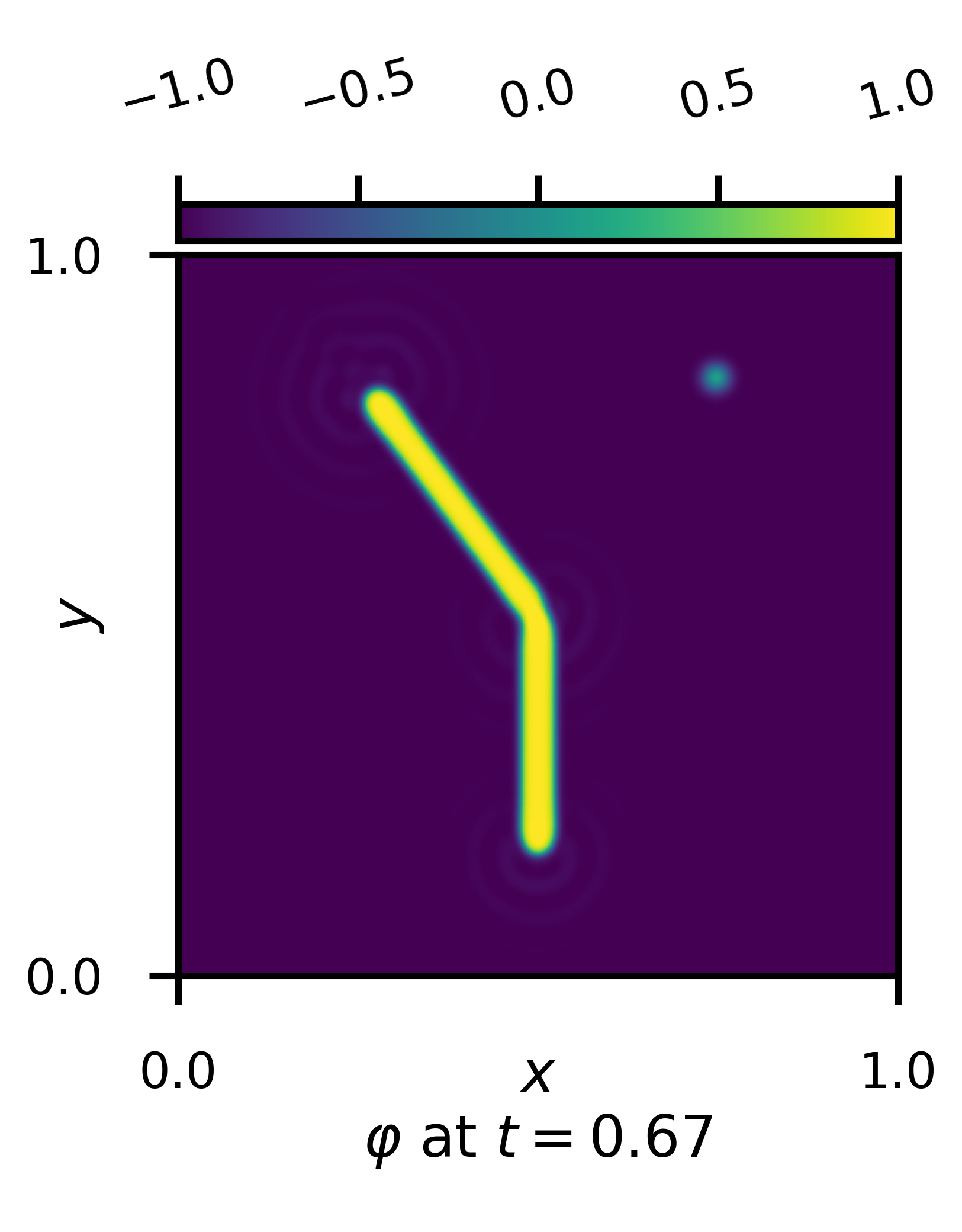}
  \end{subfigure}
  \begin{subfigure}[b]{0.24\textwidth}
    \includegraphics[width=\textwidth]{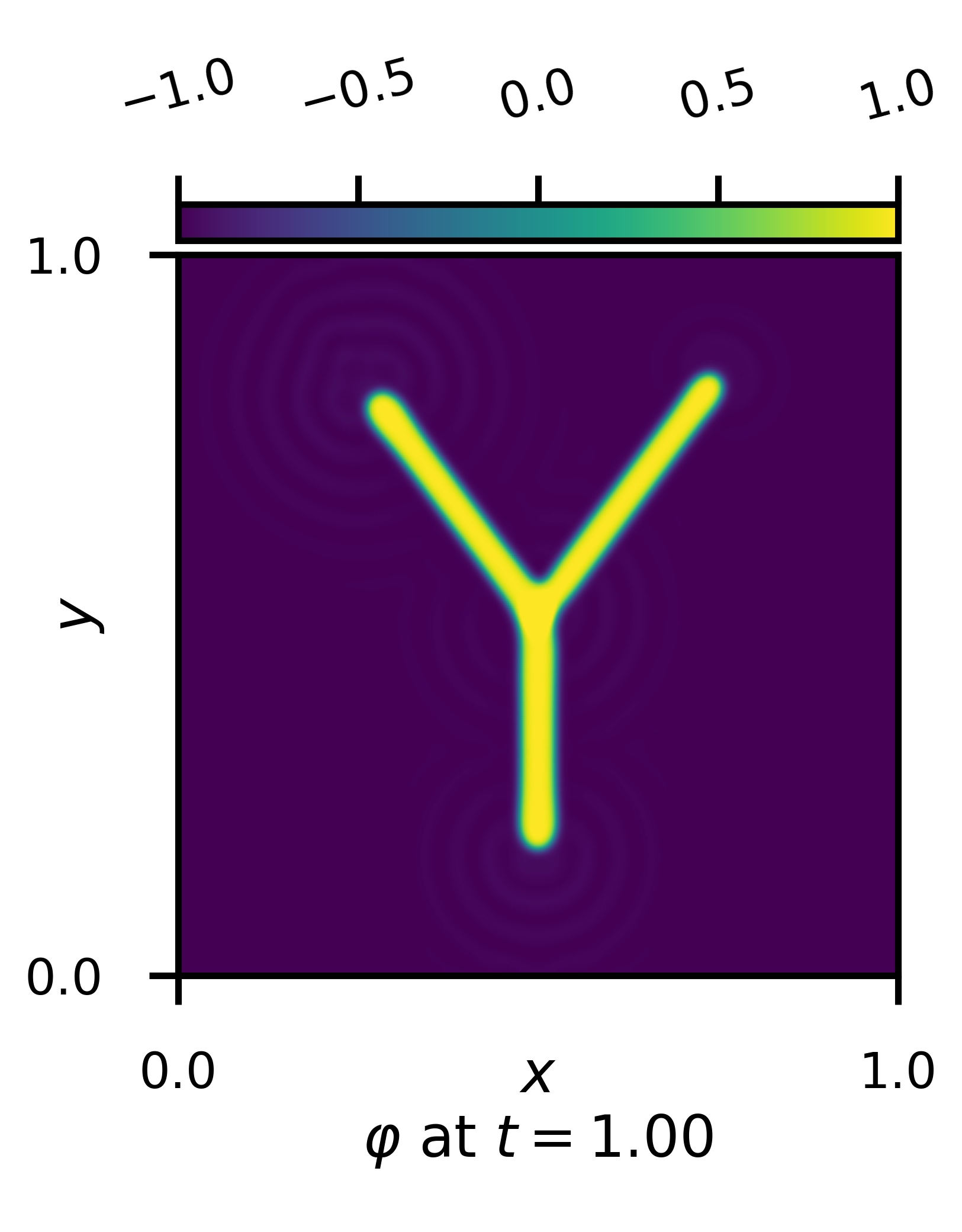}
  \end{subfigure}
  \caption{In the moving source heat source simulation, the four subplots display the phase field $\vp$ at $t=0.34$, $0.50$, $0.67$, and $1.00$.}\label{fig:moving-heat-source-phi}
\end{figure}

\begin{figure}[!ht]
  \centering
  \begin{subfigure}[b]{0.24\textwidth}
    \includegraphics[width=\textwidth]{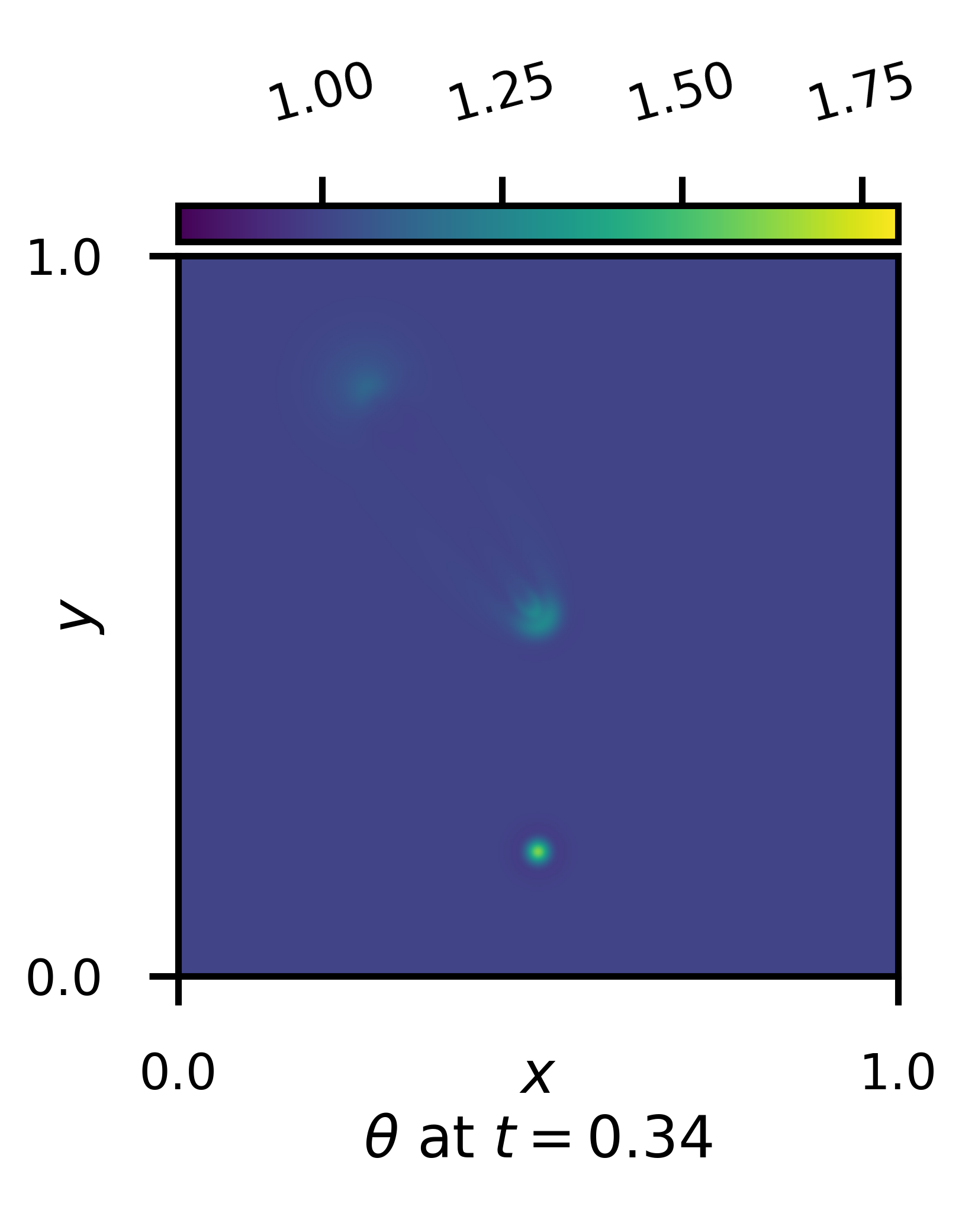}
  \end{subfigure}
  \begin{subfigure}[b]{0.24\textwidth}
    \includegraphics[width=\textwidth]{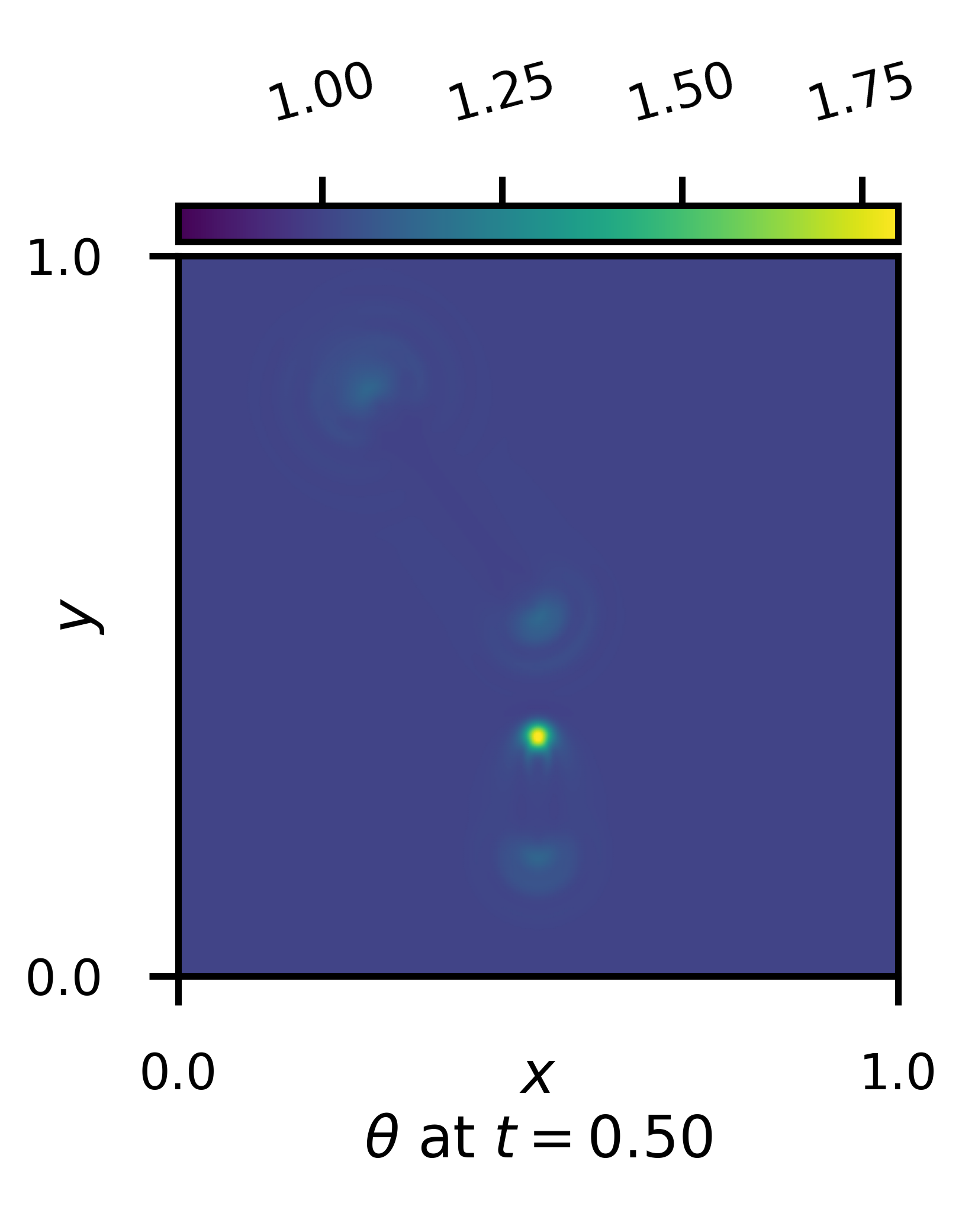}
  \end{subfigure}
  \begin{subfigure}[b]{0.24\textwidth}
    \includegraphics[width=\textwidth]{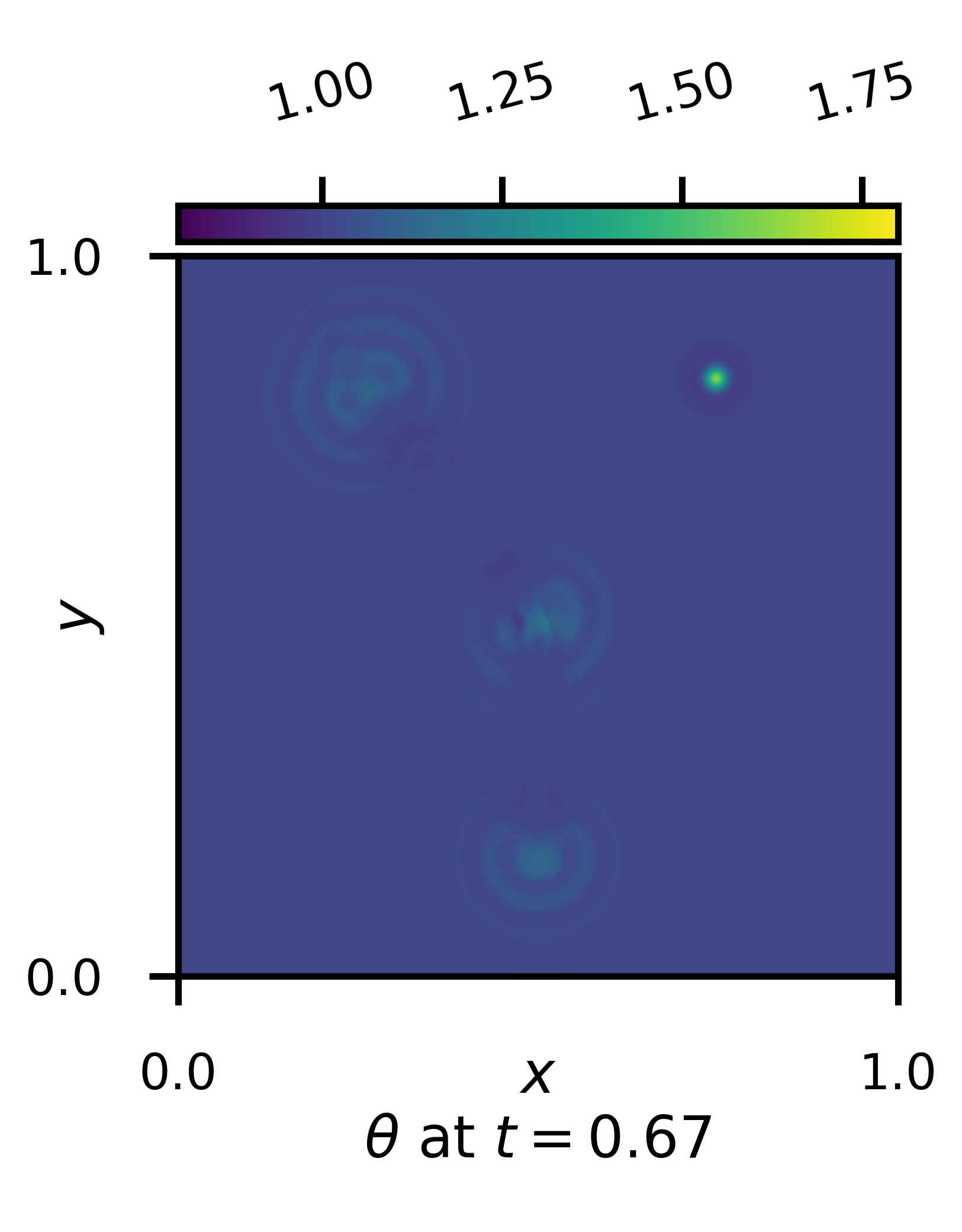}
  \end{subfigure}
  \begin{subfigure}[b]{0.24\textwidth}
    \includegraphics[width=\textwidth]{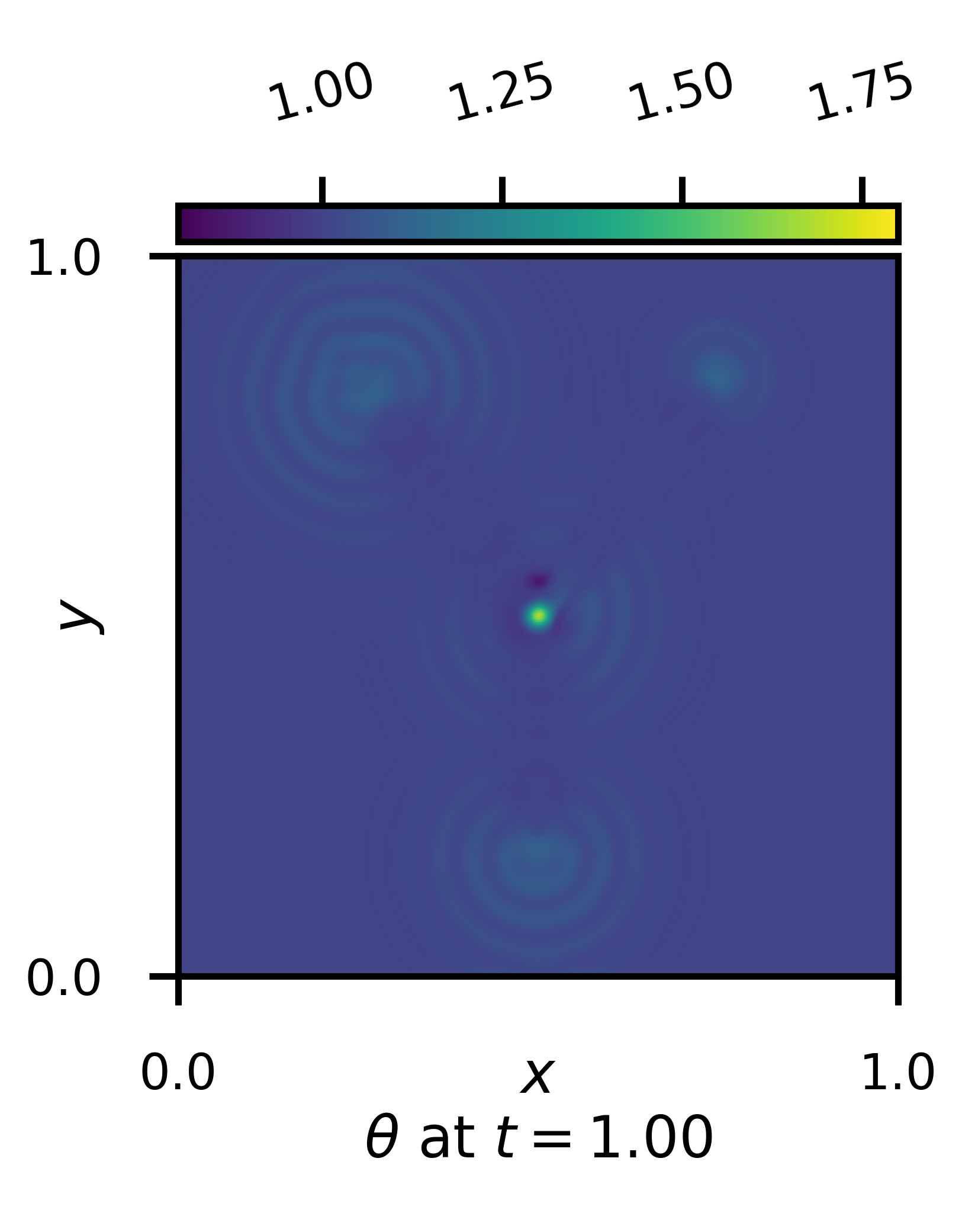}
  \end{subfigure}
  \caption{In the moving source heat source simulation, the four subplots display the temperature field $\theta$ at $t=0.34$, $0.50$, $0.67$, and $1.00$.}\label{fig:moving-heat-source-theta}
\end{figure}

\begin{figure}[!ht]
  \centering
  \begin{subfigure}[b]{0.24\textwidth}
    \includegraphics[width=\textwidth]{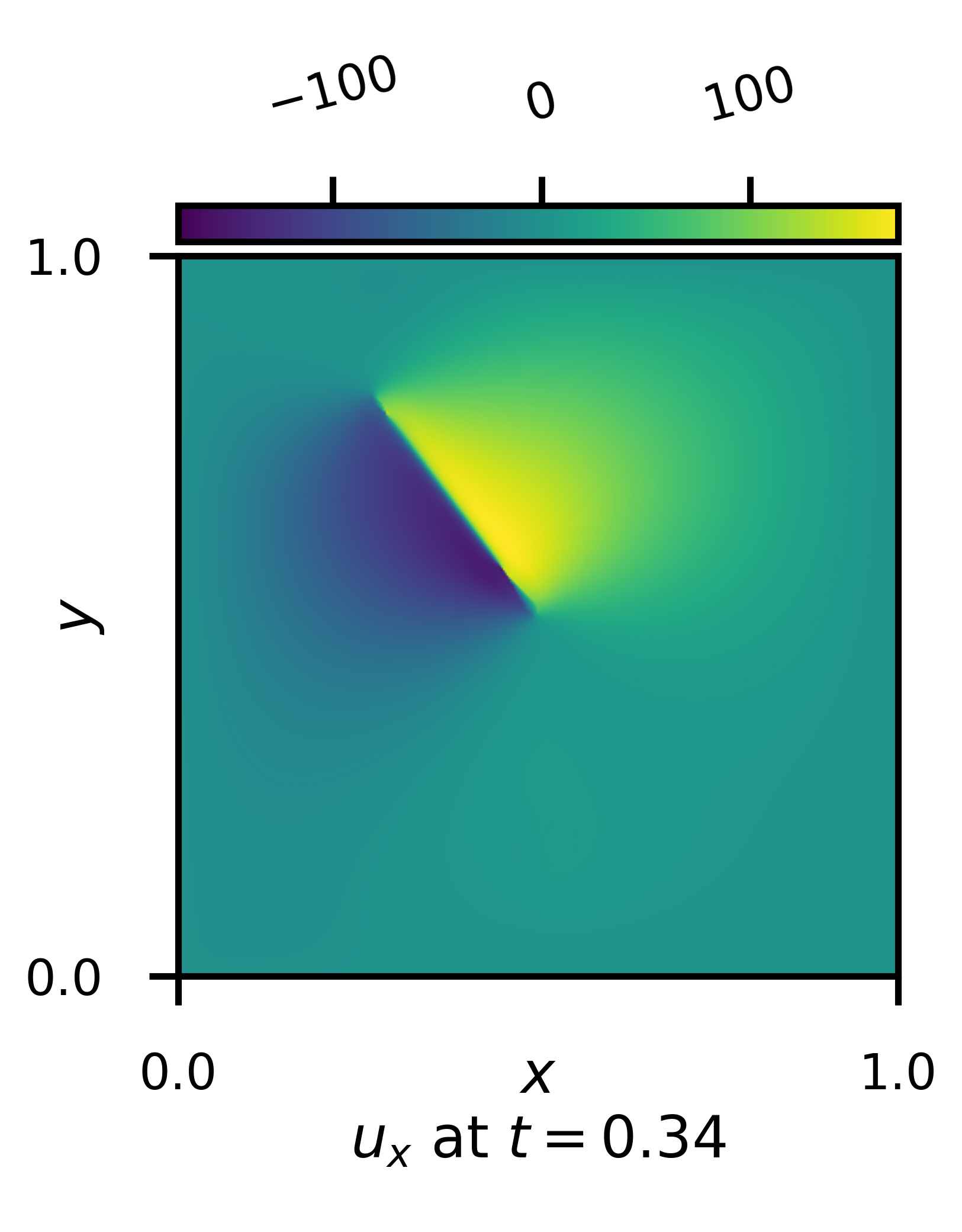}
  \end{subfigure}
  \begin{subfigure}[b]{0.24\textwidth}
    \includegraphics[width=\textwidth]{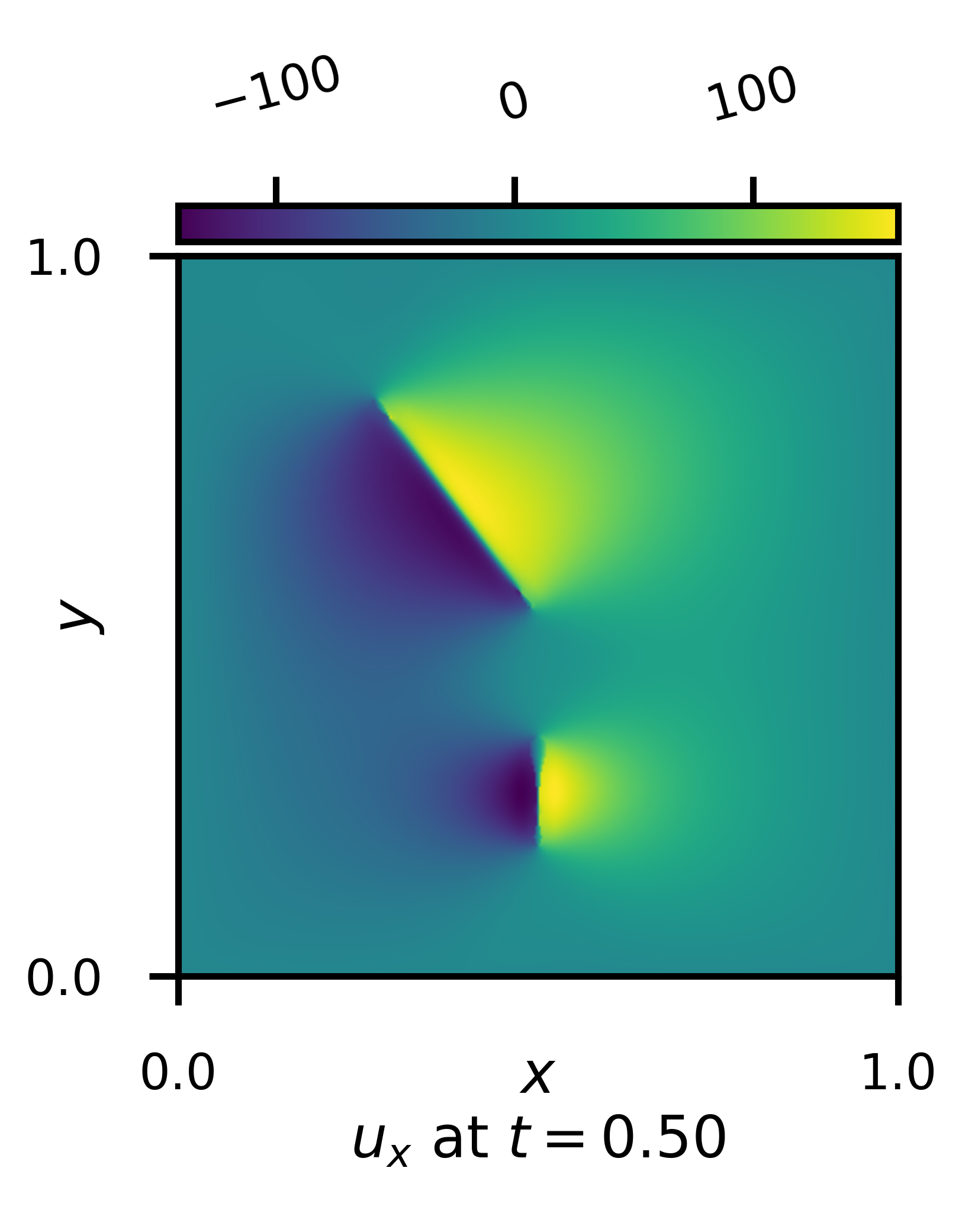}
  \end{subfigure}
  \begin{subfigure}[b]{0.24\textwidth}
    \includegraphics[width=\textwidth]{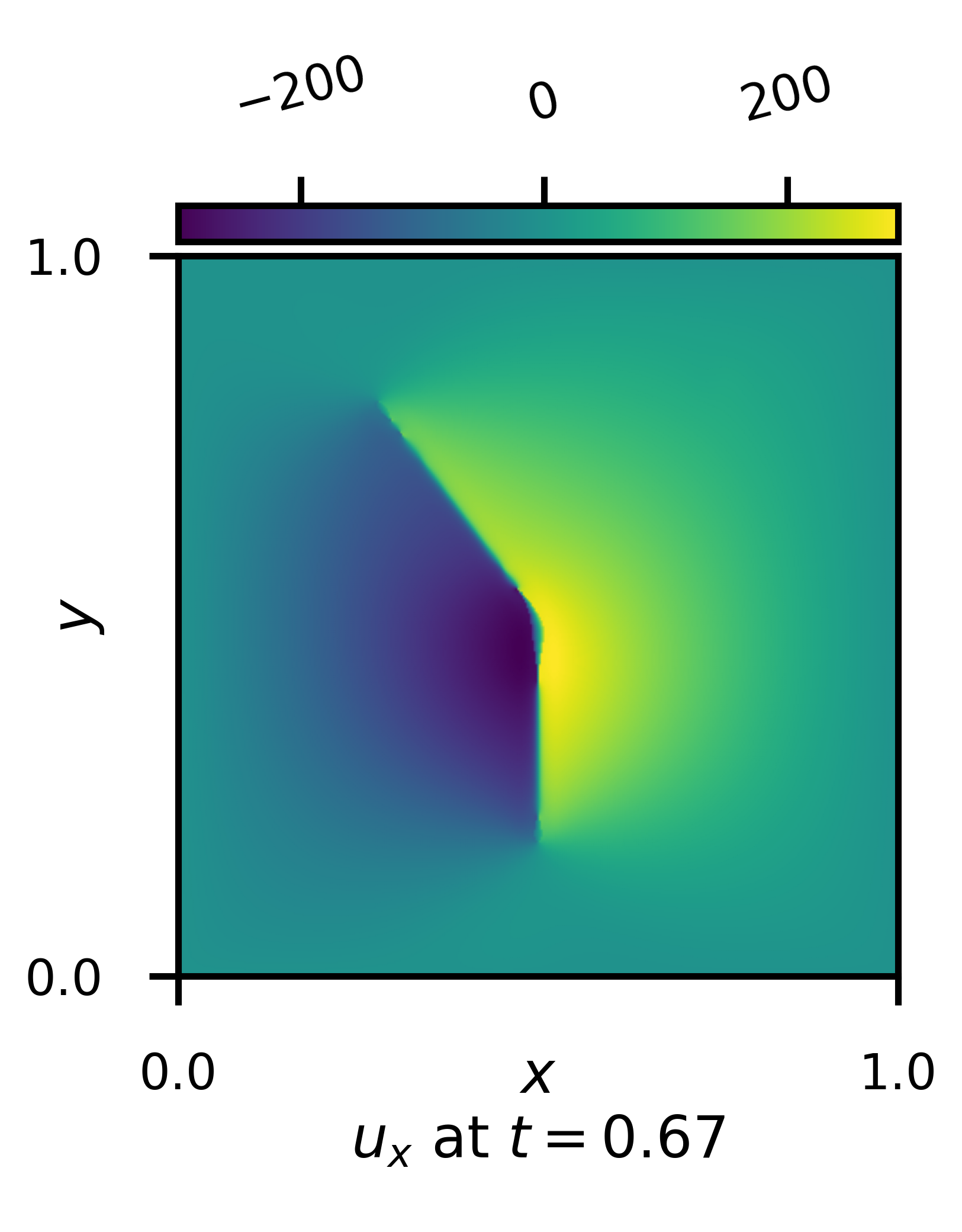}
  \end{subfigure}
  \begin{subfigure}[b]{0.24\textwidth}
    \includegraphics[width=\textwidth]{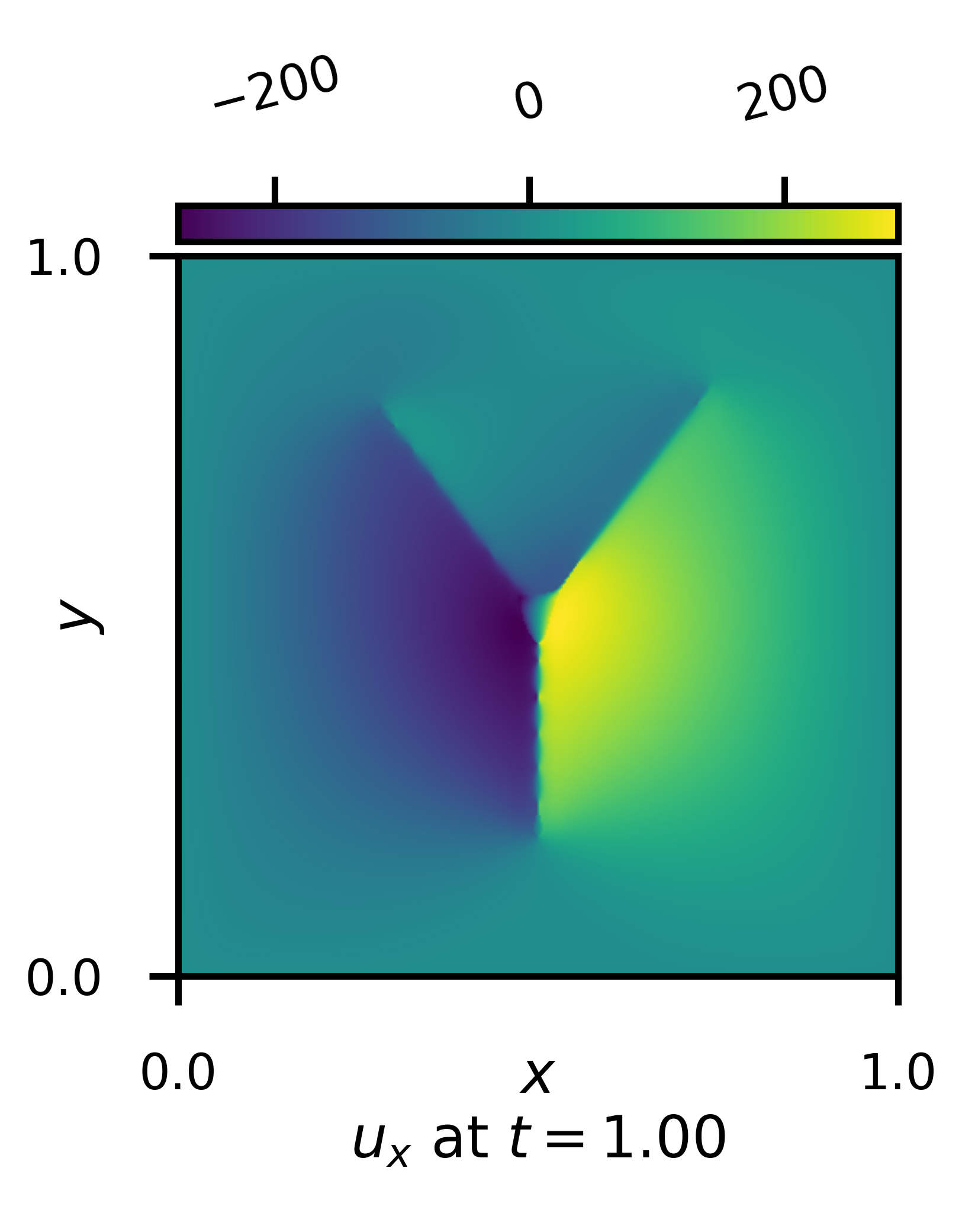}
  \end{subfigure}
  \begin{subfigure}[b]{0.24\textwidth}
    \includegraphics[width=\textwidth]{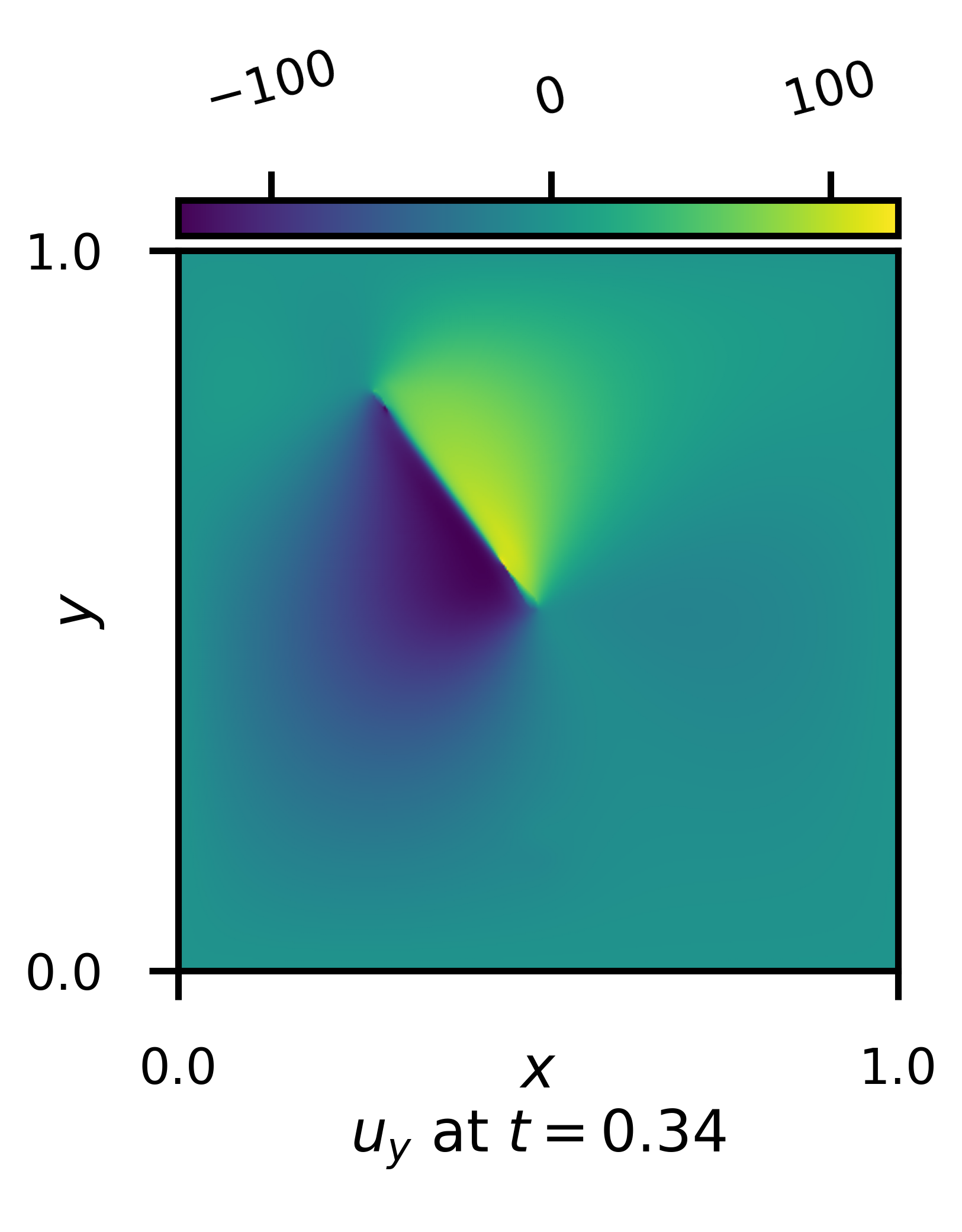}
  \end{subfigure}
  \begin{subfigure}[b]{0.24\textwidth}
    \includegraphics[width=\textwidth]{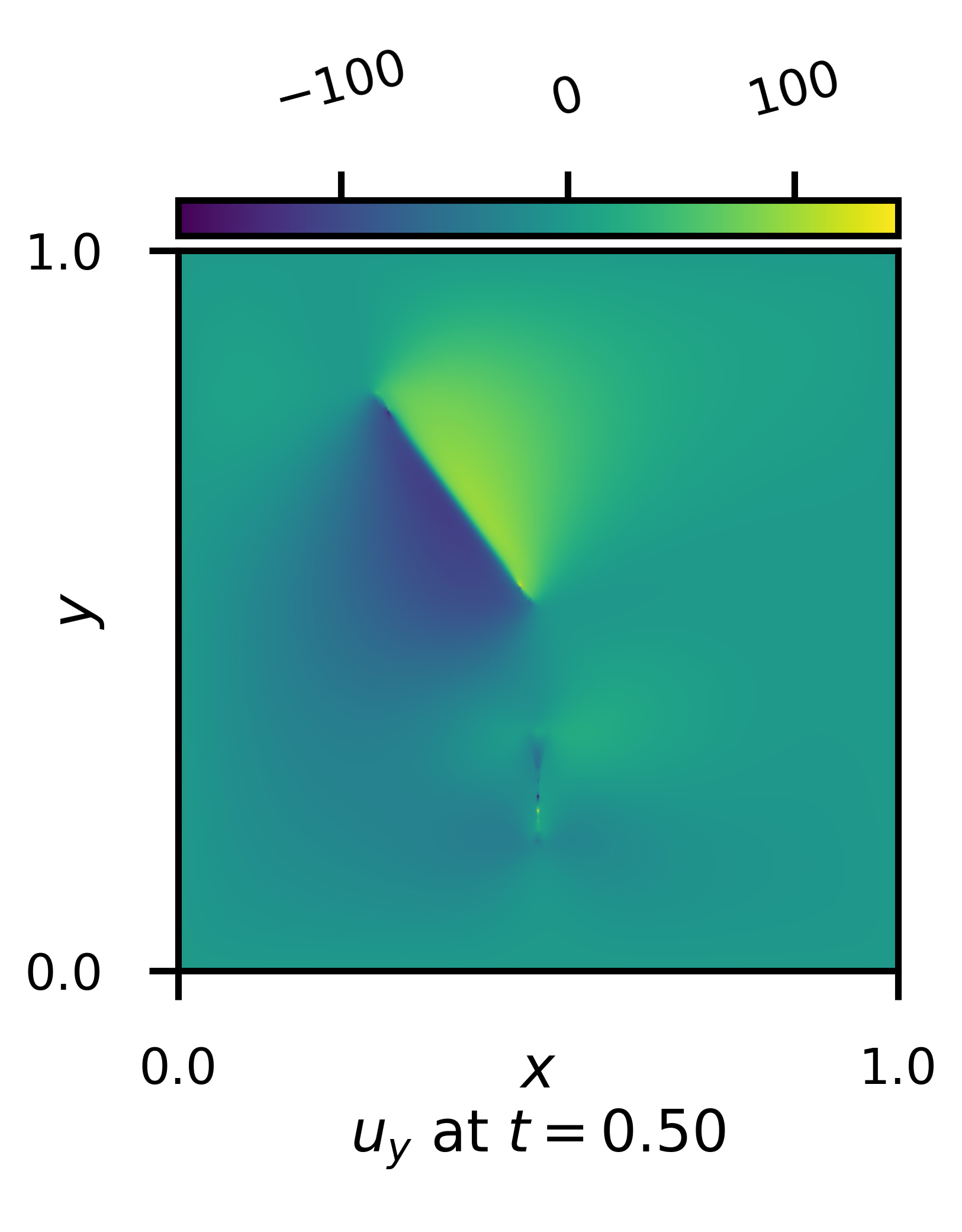}
  \end{subfigure}
  \begin{subfigure}[b]{0.24\textwidth}
    \includegraphics[width=\textwidth]{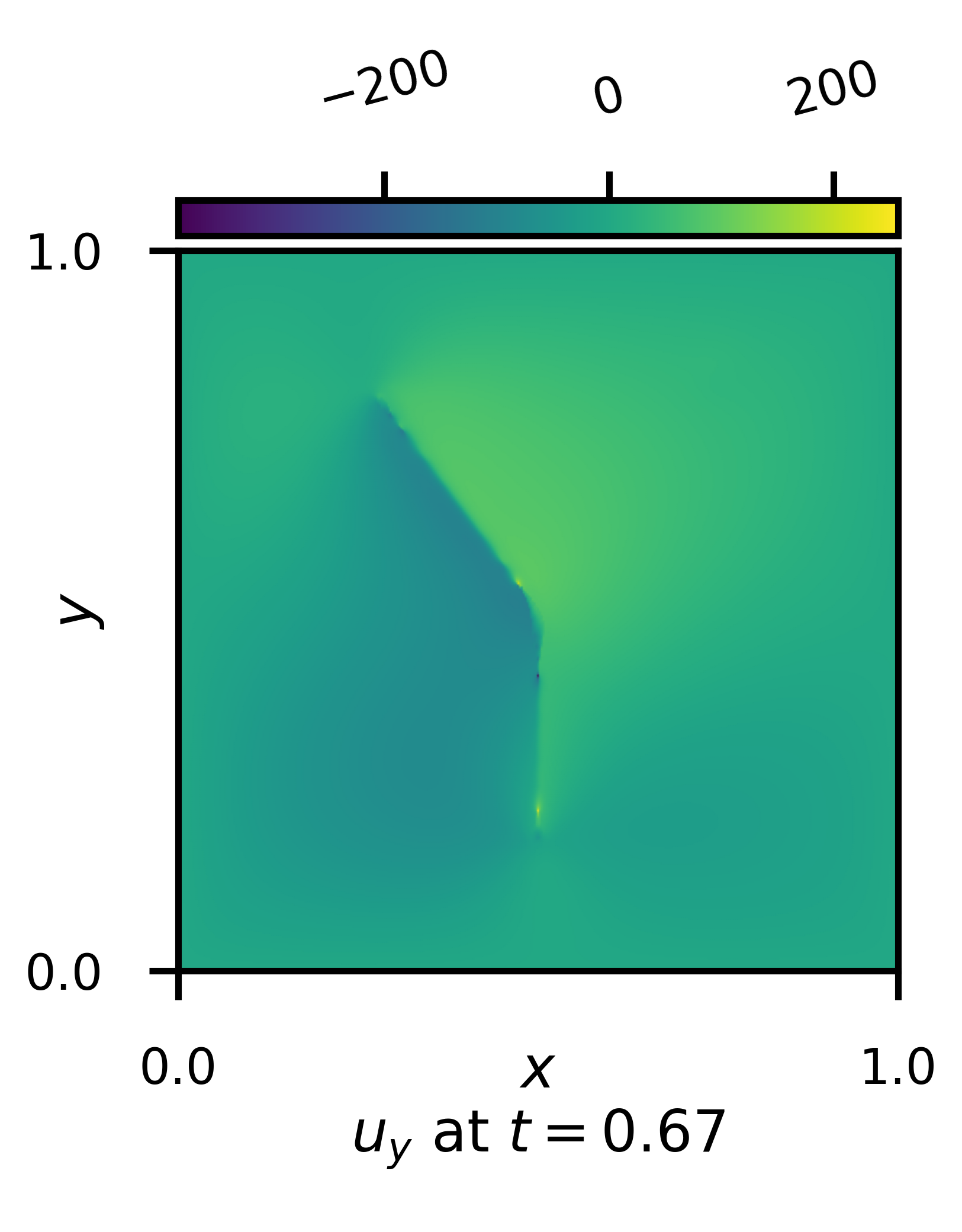}
  \end{subfigure}
  \begin{subfigure}[b]{0.24\textwidth}
    \includegraphics[width=\textwidth]{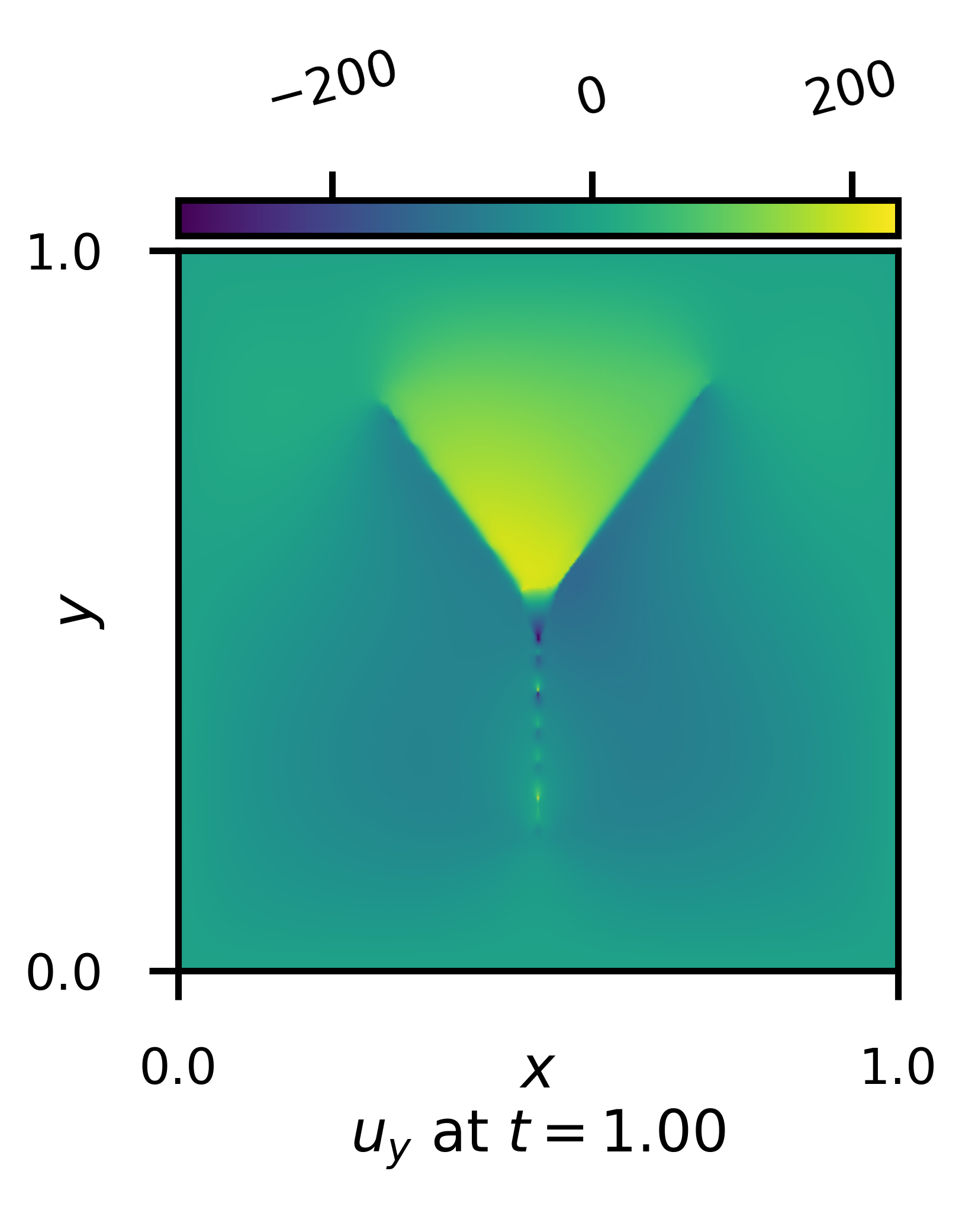}
  \end{subfigure}
  \caption{In the moving source heat source simulation, the four subplots in the first/second row display the displacement field $u_x$/$u_y$ at $t=0.34$, $0.50$, $0.67$, and $1.00$, respectively.}\label{fig:moving-heat-source-u}
\end{figure}

\section{Conclusion}
In this contribution, we proposed a new mathematical model for the physical processes occurring in stereolithography based on a Caginalp phase field system with mechanical effects. A fully discrete unconditionally stable numerical scheme based on the finite element method for spatial discretization and the scalar auxiliary variable approach for temporal discretization is proposed and analyzed. We established existence, uniqueness and convergence of fully discrete solutions, as well as error estimates against the exact solution to the Caginalp submodel. The implementation of the discrete system is efficient as advancing to the next time iteration requires only solving linear systems, and the numerical simulations presented support our theoretical findings, as well as reproducing the expected behaviour during the polymerization processes in stereolithography.

\section*{Acknowledgments}
\noindent KFL gratefully acknowledges the support by the Research Grants Council of the Hong Kong Special Administrative Region, China [Project No.: HKBU 22300522 and HKBU 12302023] and Hong Kong Baptist University [Project No.: RC-OFSGT2/20-21/SCI/006].

\footnotesize
\bibliographystyle{plain}

\end{document}